\documentclass[a4paper]{amsart}                

\usepackage[latin1]{inputenc}                  
\usepackage[T1]{fontenc}                       
\usepackage[spanish,english]{babel}            
\usepackage{graphicx}                  
\usepackage{amsmath,amssymb,amsthm}            
\usepackage{latexsym}                          
\usepackage{delarray}                          
\usepackage{bbm}                               
\usepackage{hyperref}                          
\usepackage{calrsfs,eufrak,mathrsfs,upgreek}   
\usepackage{bbding,trfsigns}                   
\usepackage{wasysym}                           
\usepackage{datetime}                          
\usepackage{color}                             
\usepackage{tikz}

\DeclareMathAlphabet{\mathpzc}{OT1}{pzc}{m}{it} 

\newtheorem{theorem}{Theorem}

\newtheorem{proposition}[theorem]{Proposition}

\newcommand{\R}{\mathbb{{R}}}

\newcommand{\N}{\mathbb{{N}}}

\newcommand{\C}{\mathbb{{C}}}
\newcommand{\INT}{\int_0^\infty}
\newcommand{\DO}{(0,\infty)}

\thanks{The first author was partially supported by MTM2016-79436-P. The second author was partially supported by MTM2015-66157-C2-1-P}

\author[Betancor]{J.J. Betancor}
\address{Departamento de An\'alisis Matem\'atico,  Universidad de La Laguna, 38071, S/C de Tenerife, Spain.}
\email{jbetanco@ull.es}

\author[De Le\'on]{M. De Le\'on-Contreras}
\address{Departamento de Matem\'aticas, Facultad de Ciencias, Universidad Aut\'onoma de Madrid, 28049 Madrid, Spain.}
\email{marta.leon@uam.es}

\date{\today}

\begin{document}

\title[Singular integrals parabolic Bessel equation]
{Parabolic equations involving Bessel operators and singular
integrals}



\begin{abstract}
In this paper we consider the evolution equation $\partial_t u=\Delta_\mu u+f$ and the corresponding Cauchy problem, where $\Delta_\mu$ represents the Bessel operator $\partial_x^2+(\frac{1}{4}-\mu^2)x^{-2}$,  for every $\mu>-1$. We establish weighted and mixed weighted Sobolev type inequalities for solutions of Bessel parabolic equations. We use singular integrals techniques in a parabolic setting.
\end{abstract}

\maketitle
\section{Introduction}
The model of a parabolic differential equation is the following

    \begin{equation}\label{I1.1}
    \frac{\partial u(t,x)}{\partial t}= \Delta u(t,x)+f(t,x), \;\; (t,x)\in\R^{n+1} \;\; or \;\; (t,x)\in\DO\times \R^n,
    \end{equation}
where $\Delta$ denotes the Laplace operator. By $\Gamma$ we denote
the fundamental solution of the heat equation $\partial_t u-\Delta
u=0$, on $\DO\times \R^n$, that is, $
\Gamma(t,x)=\frac{e^{-\frac{|x|^2}{4t}}}{(4\pi t)^{n/2}}$,
$x\in\R^n$, and $t\in\DO$. Assume that $f$ is a bounded function
defined on $\DO\times\R^n$ with compact support. We define
    $$
    u(t,x)=\int_{0}^{t}\int_{\R^n}\Gamma(t-s,x-y)f(s,y)dyds, \;\; (t,x)\in\DO\times\R^n.
    $$
    Thus, $u$ is a solution of \eqref{I1.1} on $\DO\times\R^n$ such that $u(0,x)=0$, $x\in\R^n$. Moreover, Jones \cite{Jones1} proved that
    \begin{align*}
    &\partial^2_{x_i,x_j}u(t,x)=\displaystyle\lim_{\epsilon\to 0^+}\int_0^{t-\epsilon}\int_{\R^n}\partial_{y_i y_j}\Gamma(t-s,x-y)f(s,y)dyds, \;\; (t,x)\in\DO\times \R^n, \\
    &\partial_{t}u(t,x)=\displaystyle\lim_{\epsilon\to 0^+}\int_0^{t-\epsilon}\int_{\R^n}\partial_{t}\Gamma(t-s,x-y)f(s,y)dyds+f(t,x), \;\; (t,x)\in\DO\times \R^n,
    \end{align*}
    and there exists a constant $C>0$ such that
    \begin{equation}\label{I1.2}
    \|\partial^2_{x_i,x_j}u\|_{L^p(\DO\times\R^n)}+\|\partial_t u\|_{L^p(\DO\times\R^n)}\le C\|f\|_{L^p(\DO\times\R^n)}.
    \end{equation}
    In \cite{Jones2}, $L^p$-inequalities like \eqref{I1.2} were established for a more general parabolic singular integral, where the derivatives of the fundamental solution $\Gamma$ are replaced by more general kernels.

    Recently, Jones' results have been extended to weighted and mixed weighted $L^p$ spaces by Ping, Stinga and Torrea (\cite[Theorem 2.4]{PST}).
    In \cite{PST}, they also considered the equation \eqref{I1.1} on the whole space $\R^{n+1}$. In \cite[Theorem 2.3]{PST},
    it was proved that if $f$ is a $C^2$-function defined on $\R^{n+1}$ with compact support,  when $n\ge 2,$ the function $u$ given by
    $$
        u(t,x)=\INT\int_{\R^n}\Gamma(s,y)f(t-s,x-y)dyds, \;\; (t,x)\in\R^{n+1},
    $$
    is a solution of \eqref{I1.1} on $\R^{n+1}$, where
    \begin{align*}
    &\partial^2_{x_i,x_j}u(t,x)=\displaystyle\lim_{\epsilon\to 0^+}\int_{\Omega_\epsilon}\partial_{y_i y_j}\Gamma(s,y)f(t-s,x-y)dyds-A_n f(t,x), \;\; (t,x)\in\R^{n+1}, \; i,j=1,\dots,n,
    \end{align*}
    and
    \begin{align*}
    &\partial_{t}u(t,x)=\displaystyle\lim_{\epsilon\to 0^+}\int_{\Omega_\epsilon}\partial_{s}\Gamma(s,y)f(t-s,x-y)dyds+(1-A_n)f(t,x), \;\; (t,x)\in\R^{n+1},
    \end{align*}
where $A_n=\frac{1}{n\Gamma(n/2)}\displaystyle\int_{1/4}^\infty
\omega^{\frac{n-2}{2}}e^{-\omega}d\omega$. Here, for every
$\epsilon>0$, $\Omega_\epsilon$ denotes the parabolic region
$\Omega_\epsilon=\{(t,x)\in\DO\times\R: \;
\sqrt{t}+|x|>\epsilon\}$.

    By using Calder\'on-Zygmund theory in the parabolic setting in  \cite[Theorem 2.3, (B)]{PST}, weighted norm $L^p$-inequalities and  weighted weak $L^1$-estimates as \eqref{I1.2} were established in this context.

    Parabolic equation of \eqref{I1.1} type, where the Laplace operator $\Delta$ is replaced by the Hermite operator $\mathcal{H}=\Delta-|x|^2$, was also studied in \cite{PST}. In \cite[Theorems 1.3 and 1.4]{PST}, weighted parabolic Sobolev estimates (like \eqref{I1.2}) were established in the Hermite setting.

    It is well-known that partial differential equations and singular integrals are closely connected (see for instance \cite{CaZy1} and \cite{CaZy2}).
    The previous commented examples show this relationship. Although elliptic PDE's have received a preferential and continuous treatment over time,
    parabolic PDE's have been already studied since the sixties in the last century. Apart from Jones' papers (\cite{Jones1} and \cite{Jones2}),
    we can find relevant results about parabolic PDE's and "a priori" Sobolev estimates in \cite{Fabes}, \cite{FS}, and \cite{Sa}.
    In the last years, the study of parabolic equations by using harmonic analysis techniques has taken great interest (see, for instance, \cite{AEN}, \cite{CNS}, \cite{CRS} and \cite{Ny}).

    On the other hand, the use of Calder\'on-Zygmund theory in the context of parabolic PDE's  and parabolic singular integrals appears also in \cite{RT2}, where some of Jones' results are improved, and more recently in \cite{Li} where a singular integral approach to the maximal $L^p$-regularity is developed (see also \cite{KLK}, \cite{Kry1} and \cite{Kry2}).

    In this paper we consider the parabolic equations
    \begin{equation}\label{BP}
        \frac{\partial u(t,x)}{\partial t}= \Delta_\mu u(t,x)+f(t,x), \;\; (t,x)\in\R \times\DO \;\; or \;\; (t,x)\in\DO\times \DO,
    \end{equation}
    where, for every $\mu>-1$, $\Delta_\mu$ represents the Bessel operator defined by $\Delta_\mu= \partial_x^2+(\frac{1}{4}-\mu^2)x^{-2}$.

    Our study is motivated by \cite{PST}. The Bessel operator $\Delta_\mu$ can be seen as a one dimensional Schr\"odinger operator
    with the singular potential $V_\mu(x)=(\frac{1}{4}-\mu^2)x^{-2}$, $x\in\DO$. Singular integrals associated with parabolic Schr\"odinger
    operators $\partial_t-\Delta+V$ in $\R^{n+1}$ have been investigated in \cite{CMS1}, \cite{GJ}, \cite{LH} and \cite{OS}.
    Our potentials $V_\mu$, $\mu>-1$, are not included in the class of potentials considered in the above mentioned papers.
    There, the potentials $V$ are nonnegative and in $L^1_{\rm loc}(\R^{n+1})$ and they belong to the parabolic reverse H\"older classes.

    Let $\mu>-1$. For every $\phi\in C^\infty_c(\DO),$ the space of smooth functions with compact support on $\DO$, the Hankel transform $h_\mu(\phi)$ of $\phi$ is defined by
    $$
    h_\mu(\phi)(x)=\INT \sqrt{xy}J_\mu(xy)\phi(y)dy, \;\; x\in\DO,
    $$
    where $J_\mu$ denotes the Bessel function of the first kind and order $\mu$.
    $h_\mu$ can be extended to $L^2(\DO)$ as an isometry (see  \cite{BSt} and \cite{Tit}) and $h_{\mu}^{-1}=h_\mu$ on $L^2(\DO)$.
     For every $\phi\in  C^\infty_c(\DO),$ we have that (see \cite[Lemma 5.4-1]{Ze}),
    $$
    h_\mu(\Delta_\mu \phi)(x)=-x^2h_\mu(\phi)(x), \;\;x\in\DO.
    $$
    We extend the definition of the operator $\Delta_\mu$ as follows. We define the domain of $  {\bf\Delta}_\mu$, $D(\bf{\Delta}_\mu)$,
    by  $D({\bf \Delta_\mu})=\{ \phi\in L^2(\DO):\; x^2 h_\mu(\phi)\in L^2(\DO) \}$ and, for every $\phi\in D(\bf{\Delta_\mu})$, ${\bf\Delta}_\mu\phi=-h_\mu(x^2h_\mu(\phi))$. According to \cite[Theorem 5.4-1]{Ze}, $C^\infty_c(\DO)\subset D(\bf{\Delta}_\mu)$ and ${\bf\Delta_\mu}\phi=\Delta_\mu \phi$, $\phi\in C^\infty_c(\DO)$. Note that, for every $\mu\in (-1,1)$, $\Delta_\mu\phi=\Delta_{-\mu}\phi$, $\phi\in C^\infty_c(\DO)$, and $\bf{\Delta}_\mu\neq \bf{\Delta}_{-\mu}$.

    The operator $- \bf{\Delta}_\mu$ is positive and selfadjoint on $L^2(\DO)$.
    Moreover, $- \bf{\Delta}_\mu$ generates a semigroup $\{W_t^\mu\}_{t>0}$ of operators in $L^2(\DO)$ where, for every $t>0$ and $\phi\in L^2(\DO)$,
    \begin{equation}\label{I1.3}
    W_t^\mu (\phi)(x)=\INT W_t^\mu(x,y)\phi(y)dy, \;\; x\in\DO.
    \end{equation}
    Here, $W_t^\mu(x,y)=\frac{(xy)^{1/2}}{2t}I_\mu\left( \frac{xy}{2t}\right) e^{-\frac{x^2+y^2}{4t}}$, $x,y,t\in\DO$, where $I_\mu$ represents the modified Bessel
    function of the first kind and order $\mu$. $\{W_t^\mu\}_{t>0}$ is usually called the heat semigroup associated with the Bessel operator $\bf{\Delta}_\mu$.

    If, for every $t>0$, $W_t^\mu$ is given as in \eqref{I1.3}, $\{ W_t^\mu\}_{t>0}$ also defines a semigroup of operators on $L^p(\DO)$, for each $1<p<\infty$ when
    $\mu>-1/2$ and for each $1<p<\infty$ such that $-\mu-1/2<\frac{1}{p}<\mu+3/2$, when $-1<\mu\le -1/2$.

    Harmonic analysis associated with Bessel operator (Riesz transforms, maximal operators, Littlewood-Paley functions,
    fractional Bessel operators, Hardy spaces,..) has been developed in the last years (\cite{BCC}, \cite{BCS}, \cite{BDT}, \cite{BHNV}, \cite{BFMR}, \cite{BFMT}, \cite{DLMWY} and \cite{DLWY})
    although the first results about this topic had been obtain by Muckenhoupt and Stein (\cite{MS}) in the sixties of the last century as the paper about parabolic singular integrals mentioned at the beginning.

    Our results concerning to the solutions of \eqref{BP} in the whole space $\R\times\DO$ are the following.

    \begin{theorem}\label{teo1}
        Assume that $f\in L^\infty(\R\times(0,\infty))$ has compact support on $\R\times(0,\infty)$. Then, the function $u(t,x),$ $(t,x)\in\R\times\DO$, given by
            \begin{equation*}\label{eq0}
        u(t,x)=\INT\INT W_\tau^\mu(x,y)f(t-\tau,y) dy d\tau, \;\; (t,x)\in\R\times\DO,
        \end{equation*}
        is defined by an absolutely convergent integral, for every $(t,x)\in\R\times\DO$. Moreover, if  $f$ is also in $ C^2(\R\times\DO)$, then, for every $(t,x)\in\R\times\DO$,
        $
        \frac{\partial u (t,x)}{\partial t}=\Delta_\mu u(t,x)+f(t,x),
        $
         being
        \begin{align*}
        \frac{\partial u (t,x)}{\partial t}&=\displaystyle\lim_{\epsilon\to 0^+}\int_{\epsilon}^{\infty}\INT    \frac{\partial}{\partial \tau}W_\tau^\mu(x,y)f(t-\tau,y)dyd\tau+f(t,x)\\
        &=\displaystyle\lim_{\epsilon\to 0^+}\int_{\Omega_\epsilon(x)}  \frac{\partial}{\partial \tau}W_\tau^\mu(x,y)f(t-\tau,y)dyd\tau+Af(t,x),\;\; t,x\in\DO,
        \end{align*}
        and
        \begin{align*}
        \frac{\partial^2 u (t,x)}{\partial x^2}&=\displaystyle\lim_{\epsilon\to
        0^+}\int_{\epsilon}^{\infty} \INT\frac{\partial^2}{\partial
        x^2}W_\tau^\mu(x,y)f(t-\tau,y)dyd\tau\\
        &=\displaystyle\lim_{\epsilon\to 0^+}\int_{\Omega_\epsilon(x)}\frac{\partial^2}{\partial x^2}W_\tau^\mu(x,y)f(t-\tau,y)dyd\tau-(1-A)f(t,x)  
        \end{align*}
where, for every $\epsilon,x\in\DO$,
$\Omega_{\epsilon}(x)=\{(\tau,y)\in\DO^2:
\;\tau^{1/2}+|x-y|>\epsilon\}$, and $A=\frac{1}{\sqrt{\pi}}\int_0^1
e^{-\frac{w^2}{4}}dw.$
    \end{theorem}
    The Bessel operator can be written as $\Delta_\mu= \delta_\mu^* \delta_\mu$, where $\delta_\mu=x^{\mu+1/2}\frac{d}{dx}x^{-\mu-1/2}$,
    and $\delta_\mu^*=x^{-\mu-1/2}\frac{d}{dx}x^{\mu+1/2}$ represents the formal adjoint of $\delta_\mu$. This decomposition of $\Delta_\mu$ suggests,
    according to Stein's ideas (\cite{StLP}), defining the Riesz transform ${\mathfrak R}_\mu$ associated with $\Delta_\mu$ by ${\mathfrak R}_\mu=\delta_\mu \Delta_\mu^{-1/2}$. The main $L^p$-boundedness
    properties of ${\mathfrak R}_\mu$ can be found in \cite{BBFMT} and \cite{BFMR}.

    We now consider the operator $L_\mu$ defined by
    $$
    (L_\mu f)(t,x)=\INT\INT W_s^\mu (x,y)f(t-s,y)dyds,
    $$
    being $f$ a measurable complex function defined on $\R\times\DO$, provided that the last integral exists. In Theorem \ref{teo1} we have established that if $f\in C^2(\R\times\DO)$ and has compact support, then $(\partial_t-\Delta_\mu)L_\mu(f)=f$. In a similar way we can see that $L_\mu((\partial_t-\Delta_\mu)f)=f$, provided that $f\in C^2(\R\times\DO)$ with compact support. Thus, $L_\mu$ can be seen as an inverse of $\partial_t-\Delta_\mu$.  Keeping in mind Stein's ideas (\cite{StLP}), we define Riesz transforms associated with the parabolic operator $\partial_t-\Delta_\mu$ as follows: for every $f\in C^2(\R\times\DO)$ with compact support,
    $$
    R_\mu(f)=\delta_{\mu+1}\delta_\mu L_\mu(f) \;\; \text{and}  \;\;\widetilde{R_\mu}(f)=\partial_t L_\mu (f).
    $$
    Note that, according to Theorem \ref{teo1}, if $f\in C^2(\R\times\DO)$ with compact support, the above definitions of $R_\mu(f)$ and $\widetilde{R_\mu}(f)$ have sense because the derivatives of $L_\mu (f)$ do exist. Moreover, we can write, for every $f\in C^2(\R\times\DO)$ with compact support,

    \begin{equation}\label{I1.4}
    R_\mu (f)(t,x)=\displaystyle\lim_{\epsilon\to 0^+}\int_{\Omega_\epsilon(t,x)}K_\mu(t,x;\tau,y)f(\tau, y) d\tau dy+f(t,x)\frac{1}{\sqrt{\pi}}\int_1^\infty e^{-s^2/4}ds, \;\; (t,x)\in \R\times\DO,
    \end{equation}
    and
        \begin{equation}\label{I1.5}
    \widetilde{R_\mu}(f)(t,x)=\displaystyle\lim_{\epsilon\to 0^+}\int_{\Omega_\epsilon(t,x)}\widetilde{K_\mu}(t,x;\tau,y)f(\tau, y)d\tau  dy+f(t,x)\frac{1}{\sqrt{\pi}}\int_0^1 e^{-s^2/4}ds,\;\; (t,x)\in \R\times\DO,
        \end{equation}
where
    \begin{align*}
    &K_\mu(t,x;\tau,y)=\delta_{{\mu+1}}\delta_\mu W_{t-\tau}^\mu (x,y) \chi_{\DO}(t-\tau), \;\; x,y\in\DO, \;\; t,\tau\in\R,\\
    &\widetilde{K_\mu}(t,x;\tau,y)=-\partial_\tau W_{t-\tau}^\mu (x,y)\chi_{\DO}(t-\tau), \;\; x,y\in\DO, \;\; t,\tau\in\R,
    \end{align*}
    and $\Omega_\epsilon(t,x)=\{(\tau,y)\in\DO\times\DO: \; \max\{|t-\tau|^{1/2},|x-y|\}>\epsilon\}$, for $\epsilon, x\in\DO$ and $t\in\DO$.

    Next we establish $L^p$-boundedness properties of the Riesz transforms. If $m$ denotes the Lebesgue measure on $\R\times\DO$ and $d$ represents the parabolic metric defined
    by $d((t,x),(\tau,y))=|t-\tau|^{1/2}+|x-y|$, $t,\tau\in\R$ and $x,y\in\DO$, the triple $(\R\times\DO, m,d)$ is a space of homogeneous type in the sense of Coifman and Weiss (\cite{CW}).
    We represent, for every $1\le p<\infty$, by $A_p^*(\R\times\DO)$ the class of Muckenhoupt weigths in the space of homogeneous type $(\R\times\DO, m,d)$. In section 2
    we recall the main definitions and results related to Calder\'on-Zygmund singular integrals on spaces of homogeneous type.

\begin{theorem}\label{Iteo2}
    \begin{enumerate}
        \item If $\mu>-1$, the Riesz transformations $R_\mu$ and $\widetilde{R_\mu}$ are bounded from $L^2(\R\times\DO)$ into itself.
        \item Suppose that $\mu>1/2$ or $\mu=-1/2$. The Riesz transformations $R_\mu$ and $\widetilde{R_\mu}$ can be extended from $L^2(\R\times\DO)\cap L^p(\R\times\DO,\omega)$ to $ L^p(\R\times\DO,\omega)$ as bounded operators from $ L^p(\R\times\DO,\omega)$
        \begin{itemize}
            \item into $ L^p(\R\times\DO,\omega)$, for every $1<p<\infty$ and $\omega\in A_p^*(\R\times\DO)$.
                \item into $ L^{1,\infty}(\R\times\DO,\omega)$, for $p=1$ and $\omega\in A_1^*(\R\times\DO)$.
        \end{itemize}
        \item If $\mu>-1/2$, the Riesz transformations $R_\mu$ and $\widetilde{R_\mu}$ can be extended from $L^2(\R\times\DO)\cap L^p(\R\times\DO)$ to $L^p(\R\times\DO)$ as bounded operators from $ L^p(\R\times\DO)$
            \begin{itemize}
                \item into $ L^p(\R\times\DO)$, for every $1<p<\infty$.
                \item into $ L^{1,\infty}(\R\times\DO)$, for $p=1$.
            \end{itemize}
            \item If $-1<\mu\le -1/2$, then the Riesz transformation $\widetilde{R_\mu}$ can be extended  from $L^2(\R\times\DO)\cap L^p(\R\times\DO)$ to $L^p(\R\times\DO)$
            as a bounded operator from $ L^p(\R\times\DO)$ into itself, provided that $-\mu-1/2<{1}/{p}<\mu+3/2$ and $1<p<\infty$.
            \item If $-1<\mu\le -1/2$, then the Riesz transformation ${R_\mu}$ can be extended  from $L^2(\R\times\DO)\cap L^p(\R\times\DO)$ to $L^p(\R\times\DO)$
            as a bounded operator from $ L^p(\R\times\DO)$ into itself, provided that $p>\frac{1}{\mu+3/2}$ and $1<p<\infty$.
    \end{enumerate}
    Moreover, when $\mu>-1/2$ in all these cases the extensions of the operators $R_\mu$ and $\widetilde{R_\mu}$ are defined by \eqref{I1.4} and \eqref{I1.5}, respectively, where the limit exist a.e. $(t,x)\in\R\times\DO$ and the equalities are understood also in a.e. $(t,x)\in\R\times\DO$.
\end{theorem}

    $L^p$-boundedness properties for the Riesz transforms established in Theorem \ref{Iteo2} can be seen as Sobolev estimates in our parabolic Bessel setting.
    Note that the auxiliar operator $\delta_\mu$ plays the role of derivatives to define correct Sobolev spaces in the Bessel setting (see \cite{BFRTT}).
    On the other hand, $(4)$ and $(5)$ in Theorem \ref{Iteo2} remember the so called pencil phenomenon that appears related to the $L^p$-boundedness properties
    of harmonic analysis operators in Laguerre settings (see \cite{HSTV}, \cite{MST1}, \cite{MST2}, and \cite{NSj}).

    As an application of vector-valued Calder\'on-Zygmund theory (see \cite{RT2}), we establish the following mixed weighted norm inequalities for Riesz transforms $R_\mu$ and $\widetilde{R_\mu}$.
    For every $1\le p<\infty$, we denote the classical classes of Muckenhoupt weights by $A_p(\Omega)$, where $\Omega=\DO$ or $\Omega=\R$.

    \begin{theorem}\label{Iteo3}
        Assume that $\mu>1/2$ or $\mu=-1/2$. If $1<p<\infty$ and $v\in A_p(\DO)$, then the Riesz transforms $R_\mu$ and $\widetilde{R_\mu}$ can
        be extended from $L^2(\R\times\DO)\cap L^q(\R,u,L^p(\DO, v))$ to $L^q(\R,u,L^p(\DO, v))$ as a bounded operator from $L^q(\R,u,L^p(\DO, v))$ into itself,
        provided that $1<q<\infty$ and $u\in A_q(\R)$; and, for every $u\in A_1(\R)$, from $L^2(\R\times\DO)\cap L^1(\R,u,L^p(\DO, v))$ to $L^1(\R,u,L^p(\DO, v))$ as a bounded operator from $L^1(\R,u,L^p(\DO, v))$
        into $L^{1,\infty}(\R,u,L^p(\DO, v))$.
    \end{theorem}
    Note that from Theorem \ref{Iteo3} we can deduce that, if $\mu>1/2$ or $\mu=-1/2$, $R_\mu$ and $\widetilde{R_\mu}$
    define bounded operators from $L^p(\R\times\DO,uv)$ into itself, for every $1<p<\infty$, $u\in A_p(\R)$ and $v\in A_p(\DO)$. Moreover, $uv\in A_p^*(\R\times\DO)$ provided that $u\in A_p(\R)$ and $v\in A_p(\DO)$, but $A_p^*(\R\times\DO)\neq A_p(\R)\cdot A_p(\DO)$, when $1<p<\infty$. Hence, Theorem \ref{Iteo2} (2) is not a special case of strong type results in Theorem \ref{Iteo3}.

    We now consider the following Cauchy problem associated with \eqref{BP}:

    \begin{equation}\label{I1.6}
    \left\{ \begin{array}{c}
    \partial_t u(t,x)=\Delta_\mu u(t,x)+f(t,x), \;\; (t,x)\in\DO \times \DO, \\
    u(0,x)=g(x), \;\; x\in \DO.
        \end{array} \right.
    \end{equation}
%
%
%

\begin{theorem}\label{Iteo4}
    Let $\mu>-1$. Assume that $f\in L^\infty(\DO\times\DO)$ with compact support and $g\in L^\infty(\DO)$ with compact support. We define
        \begin{equation}\label{I1.7}
        u(t,x)=\int_0^t\INT W_\tau^\mu(x,y)f(t-\tau,y) dy d\tau+\INT W_t^\mu(x,y)g(y) dy , \;\; t,x\in\DO.
        \end{equation}
        Then, the last integrals are absolutely convergent for every $t,x\in\DO$. Moreover, if  $f$ is also in $C^2(\DO\times\DO)$, then the function $u$ defined by \eqref{I1.7} is a classical solution of \eqref{I1.6} and
            \begin{align*}
        \frac{\partial u (t,x)}{\partial t}&=\displaystyle\lim_{\epsilon\to 0^+}\int_0^{t-\epsilon}\INT\frac{\partial}{\partial \tau}W_\tau^\mu(x,y)f(t-\tau,y)dyd\tau+\INT\frac{\partial}{\partial \tau}W_t^\mu(x,y)g(y)dy, \;\;
        t,x\in\DO,
        \end{align*}
        and
        $$
        \frac{\partial^2 u (t,x)}{\partial x^2}=\displaystyle\lim_{\epsilon\to 0^+}\int_0^{t-\epsilon}\INT\frac{\partial^2}{\partial x^2}W_\tau^\mu(x,y)f(t-\tau,y)dyd\tau+\INT\frac{\partial^2}{\partial x^2}W_t^\mu(x,y) g(y)dy, t,x\in\DO.
        $$

\end{theorem}

For every $f\in C^2(\DO\times\DO)$  with compact support we define
    \begin{equation}\label{I1.8}
    {\bf R}_\mu (f)(t,x)=\displaystyle\lim_{\epsilon\to 0^+}\int_0^{t-\epsilon}\INT\delta_{{\mu+1}}\delta_\mu W_{\tau}^\mu (x,y)f(t-\tau, y) dy d\tau, \;\; t,x\in\DO,
    \end{equation}
    and
    \begin{equation}\label{I1.9}
    \widetilde{{\bf R}_\mu}(f)(t,x)=\displaystyle\lim_{\epsilon\to 0^+}\int_0^{t-\epsilon}\INT\partial_\tau W_{\tau}^\mu (x,y)f(t-\tau, y) dyd\tau, \;\; t,x\in\DO.
    \end{equation}
    Note that the above limits do exist.

\begin{theorem}\label{Iteo5}

    \begin{enumerate}
\item Suppose that $\mu>1/2$ or $\mu=-1/2$. The Riesz transformations
${\bf R}_\mu$ and $\widetilde{{\bf R}_\mu}$ can be extended to $
L^p(\DO\times\DO,\omega)$ as bounded operators from $
L^p(\DO\times\DO,\omega)$
        \begin{itemize}
            \item into $ L^p(\DO\times\DO,\omega)$, for every $1<p<\infty$ and $\omega\in A_p^*(\R\times\DO)$.
                \item into $ L^{1,\infty}(\DO\times\DO,\omega)$, for $p=1$ and $\omega\in A_1^*(\R\times\DO)$.
        \end{itemize}
        \item If $\mu>-1/2$, the Riesz transformations ${\bf R}_\mu$ and $\widetilde{{\bf R}_\mu}$ can be extended  to $L^p(\DO\times\DO)$ as bounded operators from $ L^p(\DO\times\DO)$
            \begin{itemize}
                \item into $ L^p(\DO\times\DO)$, for every $1<p<\infty$.
                \item into $ L^{1,\infty}(\DO\times\DO)$, for $p=1$.
            \end{itemize}
\end{enumerate}
    Moreover, the extensions of $    {\bf R}_\mu$ and $\widetilde{{\bf R}_\mu}$ to $L^p(\DO\times\DO, \omega)$ are defined as the principal value integral operators in \eqref{I1.8} and \eqref{I1.9},
    respectively, where the limits exist a.e. $(t,x)\in \DO\times\DO$.
\end{theorem}

Suppose that $X$ is a Banach space and $A: D(A)\subset X\to X$ is an operator. If $1<p<\infty$, we say that $A$ has maximal $L^p$- regularity when there exists a constant $C>0$ such that, for every $f\in L^p(\DO,X)$ there exists a unique $u_f\in L^p(\DO,D(A))$ solution of the Cauchy problem
    \begin{equation}\label{I1.10}
    \left\{ \begin{array}{c}
    \frac{\partial}{\partial t} u(t)+ A u(t)=f(t), \;\; t\in\DO , \\
    u(0)=0\hspace{4cm}
    \end{array} \right.
    \end{equation}
satisfying
$$
\left\|\frac{\partial}{\partial t}u_f \right\|_{ L^p(\DO,X)}+\|Au_f\|_{ L^p(\DO,X)}\le C\|f\|_{ L^p(\DO,X)}.
$$

If the operator $-A$ generates a semigroup $\{T_t \}_{t\ge 0}$ of operators on $X$, the solution of \eqref{I1.10} can be written as
$$
u(t)=\int_0^t T_{t-s}(f(s))ds, \;\; t\ge 0,
$$
and $A$ has maximal $L^p$-regularity when the operator
$$
R(f)(t)=\int_0^t\frac{\partial}{\partial t} (T_{t-s})(f(s))ds
$$
is bounded from $L^p(\DO,X)$  into itself. Note that
$\frac{\partial}{\partial t} T_t=-A T_t$, $t>0$. This fact leads,
from the point of view of harmonic analysis, to replace the property
of maximal $L^p$-regularity  by the $L^p$-boundedness of certain
Banach space valued singular integrals. If suitable Gaussian bounds
hold for the semigroup generated by $-A$, then $A$ has maximal
$L^p$-regularity (see \cite{CD} and \cite{HP}).

\begin{theorem}\label{Iteo6}
    Let $\mu>-1/2$. Assume that $1<p,q<\infty$. Then, the Bessel operator $\Delta_\mu$ has maximal $L^p$-regularity on $L^q(\DO)$.
\end{theorem}

Note that Theorem \ref{Iteo6} actually establishes mixed norm estimates for $\widetilde{{\bf R}_\mu}$.

In the next sections we will prove our Theorems. In order to show
our results we use two different ways. On the one hand, we employ
scalar and vectorial Calder\'on-Zygmund theory in the parabolic
context. Here we need to get estimates involving the kernels of the
integral operators. In order to do this, the properties of the
Bessel function $I_\mu$ plays a crucial role.  On the other hand, we
use a comparative approach. In this second way we take advantages
that Bessel operators $\Delta_\mu$ are nice (in a suitable sense)
perturbations of the Laplacian. Then, it is possible to deduce the
properties of our integral operators from the corresponding ones
associated to the Laplace operator established in \cite{PST}.

Throughout this paper we will denote by $C$ and $c$ positive
constants, not necessarily the same in each occurrence.

\vspace{3mm}

{\bf Acknowledgements.} The authors thank Professor Jos\'e Luis
Torrea (UAM, Madrid) for posing the problems studied in this paper
and for reading a first version of the manuscript. His comments have
allowed us to improve Theorem \ref{Iteo5} and its proof.

\section{Proof of Theorem \ref{teo1}.}

In this section and in the following ones we use some properties of
the modified Bessel function $I_\nu$ that can be found in the
Lebedev's monograph (\cite{Leb}) and we recall now. For every
$\nu>-1$, the modified Bessel function $I_\nu$ is defined by
$$
I_\nu(z)=\sum_{k=0}^\infty
\frac{z^{2k+\nu}}{2^{2k+\nu}k!\Gamma(\nu+1)}, \,\,\,z\in
\mathbb{C}\setminus (-\infty,0].
$$
The following properties hold
\begin{equation}\label{P1}
\lim_{z\to 0}\frac{I_\nu(z)}{z^\nu}=\frac{1}{2^\nu\Gamma(\nu+1)}.
\end{equation}
\begin{equation}\label{P2}
I_\nu(z)=\frac{e^z}{\sqrt{2\pi
z}}\Bigg(\sum_{k=0}^n(-1)^k[\nu,k](2z)^{-k}+O(|z|^{-n-1})\Bigg),
\,\,\,z\in \mathbb{C}\,\,\,\text{and}\,\,\,|Arg(z)|< \frac{\pi}{4}.
\end{equation}
where $[\nu,0]=1$ and
$$
[\nu,k]=\frac{(4\nu^2-1)(4\nu^2-9)...(4\nu^2-(2k-1)^2)}{2^{2k}\Gamma(k+1)},\,\,\,k\in
\mathbb{N}\,\,\,\text{and}\,\,\,k\ge 1,
$$
and
\begin{equation}\label{P3}
\frac{d}{dz}(z^{-\nu}I_\nu(z))=z^{-\nu}I_{\nu+1}(z),\,\,\,z\in
\mathbb{C}\setminus (-\infty,0].
\end{equation}

    Suppose that $f\in L^\infty(\R\times (0,\infty))$ is a complex function such that supp$f$ is compact on $\R\times\DO$. Since $\INT W_\tau^\mu(x,y)y^{\mu+1/2}dy=x^{\mu+1/2}$, $\tau,x\in\DO$,
    we can write
    $$
    \INT\INT W_\tau^\mu(x,y)|f(t-\tau,y)|dyd\tau\le C\|f\|_\infty x^{\mu+1/2}, \;\; x\in\DO\;\text{and} \;t\in\R.
    $$
    Here $C>0$ depends on the support of $f$.
    Hence, the integral defining
    $$
    u(t,x)=\INT\INT W_\tau^\mu(x,y)f(t-\tau,y) dy d\tau, \;\; x\in\DO \;\text{and} \; t\in\R,
    $$
     is absolutely convergent.

     Assume now that $f\in C^1( \R\times\DO)$ and it has compact support. By proceeding as above we can prove that
     \begin{align*}
        \frac{\partial u (t,x)}{\partial t}&=\INT\INT W_\tau^\mu(x,y)   \frac{\partial}{\partial t}f(t-\tau,y)dyd\tau\\
        &=-\INT\INT W_\tau^\mu(x,y) \frac{\partial}{\partial \tau}f(t-\tau,y)dyd\tau,\;\; x\in\DO, \;\text{and} \; t\in\R,
     \end{align*}
     where the integrals are absolutely convergent.

     We can write
     \begin{align}\label{A0}
        \frac{\partial u (t,x)}{\partial t}&=-\int_{x/2}^{2x}\INT( W_\tau^\mu(x,y)-W_\tau(x-y)) \frac{\partial}{\partial \tau}f(t-\tau,y)d\tau dy\\
        &-\int_0^{x/2}\INT W_\tau^\mu(x,y)  \frac{\partial}{\partial \tau}f(t-\tau,y)d\tau dy\nonumber\\
        &-\int_{2x}^\infty\INT W_\tau^\mu(x,y)  \frac{\partial}{\partial \tau}f(t-\tau,y)d\tau dy\nonumber\\
        &-\int_{x/2}^{2x}\INT W_\tau(x-y)   \frac{\partial}{\partial \tau}f(t-\tau,y)d\tau dy, \;\;x,t\in\DO.\nonumber
     \end{align}
     Here and in the sequel we denote by
$$
W_\tau(z)=
\frac{1}{\sqrt{4\pi\tau}}e^{-|z|^2/(4\tau)},\,\,\,\tau>0\,\,\,\text{and}\,\,\,z\in
\mathbb{R},
$$
the classical heat kernel.

     According to (\ref{P1}) and (\ref{P2}) (\cite[Lemma 3.1]{BHNV}), we have that
     \begin{equation}\label{V1}
     0\le W_\tau^\mu(x,y)\le  C\Bigg(1+\Bigg(\frac{xy}{\tau}\Bigg)^{\mu+1/2}\Bigg)\frac{e^{-\frac{(x-y)^2}{\tau}}}{\sqrt{\tau}},\,\,\, x,y,\tau\in\DO.
     \end{equation}

     Then, $$W_\tau^\mu(x,y)\le C \Bigg(1+\Bigg(\frac{xy}{\tau}\Bigg)^{\mu+1/2}\Bigg)\frac{e^{-cx^2/\tau}}{\sqrt{\tau}},\,\,\,0<y<x/2\,\,\,\text{and}\,\,\, \tau\in\DO.$$

     By partial integration we get
     \begin{equation}\label{A1}
    \int_0^{x/2}\INT W_\tau^\mu(x,y)    \frac{\partial}{\partial \tau}f(t-\tau,y)d\tau dy=-\int_0^{x/2}\INT\frac{\partial}{\partial \tau} W_\tau^\mu(x,y)f(t-\tau,y)d\tau dy,\;\; t,x\in\DO.
     \end{equation}
     Also, we have that
     \begin{equation}\label{A2}
        \int_{2x}^\infty\INT W_\tau^\mu(x,y)    \frac{\partial}{\partial \tau}f(t-\tau,y)d\tau dy=- \int_{2x}^\infty\INT\frac{\partial}{\partial \tau} W_\tau^\mu(x,y)f(t-\tau,y)d\tau dy,\;\; t,x\in\DO.
     \end{equation}
      By using (\ref{P2}), it follows that
     \begin{align*}
     |W_\tau^\mu(x,y)-W_\tau(x-y)|&=\left|\frac{(xy)^{1/2}}{2\tau}e^{-\frac{x^2+y^2}{4\tau}}I_\mu\bigg(\frac{xy}{2\tau} \bigg) -\frac{1}{2\sqrt{\pi}}\frac{e^{-(x-y)^2/4\tau}}{\sqrt{\tau}}\right|\\
     &\le  C\frac{\sqrt{\tau}}{xy}e^{-(x-y)^2/(4\tau)}, \;\; \tau,x,y\in\DO.
         \end{align*}
     Partial integration leads to
     \begin{align}\label{A3}
     \int_{x/2}^{2x}\INT( W_\tau^\mu(x,y)&-W_\tau(x-y)) \frac{\partial}{\partial \tau}f(t-\tau,y)d\tau dy\nonumber\\
     &=- \int_{x/2}^{2x}\INT\frac{\partial}{\partial \tau}( W_\tau^\mu(x,y)-W_\tau(x-y))f(t-\tau,y)d\tau dy, \;\; t,x\in\DO.
     \end{align}

     We are going to see that the integrals on the right hand side of \eqref{A1}, \eqref{A2} and \eqref{A3} are absolutely convergent.
     By (\ref{P1}), (\ref{P2}), and (\ref{P3}) (\cite[pages 128-131]{BHNV}) we have that, for every $0<y<x/2$,
     \begin{equation}\label{T3}
      \bigg|\frac{\partial}{\partial \tau} W_\tau^\mu(x,y)\bigg|\le C \left\{ \begin{array}{c}
    \frac{1}{\tau^{3/2}}e^{-c x^2/\tau},\;\; 0<\tau<xy,\\
    \\
    \frac{(xy)^{\mu+1/2}}{\tau^{\mu+2}}e^{-cx^2/\tau}, \;\; \tau\ge xy.
     \end{array} \right.
     \end{equation}
     Let $x\in (0,\infty)$ and $t\in \mathbb{R}$. Since supp$f$ is compact, there exist $0<a<x/2$, $2x<b$ and $c>0$, such that $f(t-\tau,y)=0$, $(\tau,y)\notin (-\infty,c)\times (a,b)$. Then,
     $$
     \int_0^{x/2}\INT\bigg|\frac{\partial}{\partial \tau} W_\tau^\mu(x,y) \bigg||f(t-\tau,y)|d\tau dy\le C  \int_a^{x/2}\bigg(\int_0^{xy}\frac{d\tau}{x^2\sqrt{\tau}}d\tau+\int_{xy}^\infty\frac{(xy)^{\mu+1/2}}{\tau^{\mu+2}}d\tau\bigg)dy<\infty.
$$

     In a similar way we can see that
$$
     \int_{2x}^\infty\INT \bigg|\frac{\partial}{\partial \tau} W_\tau^\mu(x,y) \bigg||f(t-\tau,y)|d\tau dy<\infty.
         $$

    Again, according to (\ref{P1}), (\ref{P2}) and (\ref{P3}) (\cite[pages 128-130]{BHNV}) we have that
    \begin{equation}\label{T2}
     \bigg|\frac{\partial}{\partial \tau} W_\tau^\mu(x,y)-\frac{\partial}{\partial \tau} W_\tau(x-y)\bigg|\le C\frac{e^{-c\frac{(x-y)^2}{\tau}}}{xy\sqrt{\tau}} \le \frac{C}{xy\sqrt{\tau}},\;\;\tau,x,y\in\DO.
    \end{equation}
     Then

      \begin{align*}
      \int_{x/2}^{2x}\INT\bigg|\frac{\partial}{\partial \tau} W_\tau^\mu(x,y)-\frac{\partial}{\partial \tau} W_\tau(x-y)\bigg||f(t-\tau,y)|d\tau dy\le C \int_{a}^{b}\int_0^{c}\frac{d\tau}{xy\sqrt{\tau}}dy<\infty.
     \end{align*}

     On the other hand,
     \begin{align}\label{A4}
     &\int_{x/2}^{2x}\int_0^\infty W_\tau(x-y) \frac{\partial f}{\partial \tau}(t-\tau,y)d\tau dy=\displaystyle\lim_{\epsilon\to 0^+}\int_{x/2}^{2x}\int_{\epsilon}^{\infty}W_\tau(x-y)\frac{\partial f}{\partial \tau}(t-\tau,y)d\tau dy\nonumber\\
     &=-\lim_{\epsilon\to 0^+}\left(\int_{x/2}^{2x}W_\epsilon(x-y)f(t-\epsilon,y) dy+\int_{x/2}^{2x}\int_\epsilon^\infty  \frac{\partial }{\partial \tau}W_\tau(x-y)f(t-\tau,y)d\tau dy\right)\nonumber\\
     &=-\lim_{\epsilon\to 0^+}\int_{\epsilon}^\infty\int_{x/2}^{2x}  \frac{\partial }{\partial \tau}W_\tau(x-y)f(t-\tau,y)d\tau dy-f(t,x),\,\,\,t\in \mathbb{R}\,\,\,\text{and}\,\,\,x\in (0,\infty).
     \end{align}
      In the last equality we have taken into account that
$$
\lim_{s\to 0^+}\int_{x/2}^{2x}W_s(x-y)f(t-s,y)dy=f(t,x), \,\,\,t\in
\mathbb{R}\,\,\,\text{and}\,\,\,x\in (0,\infty).
$$
Indeed, let $t\in \mathbb{R}$ and $x\in (0,\infty)$. Since $f\in
C^1(\mathbb{R}\times (0,\infty))$ with compact support, by using
mean value theorem we deduce that $|f(t-s,y)-f(t,y)|\le Cs$, $s,y\in
(0,\infty)$. Then, we can write
$$
\Bigg|\int_{x/2}^{2x}W_s(x-y)f(t-s,y)dy-\int_{x/2}^{2x}W_s(x-y)f(t,y)dy\Bigg|\le
Cs\int_{x/2}^{2x}W_s(x-y)dy\le Cs.
$$
On the other hand, for a certain $a>0$ such that $2/a<x<a/2$ and
$f(t,y)=0$, $y\notin (1/a,a)$. It follows, with the obvious
extension of $f$, that
$$
\Bigg|\int_{-\infty}^{x/2}W_s(x-y)f(t,y)dy\Bigg|\le
C\int_{1/a}^{x/2}W_s(x-y)dy\le Cs^{-1/2}e^{-cx^2/s},\,\,\,s>0,
$$
and
$$
\Bigg|\int_{2x}^{\infty}W_s(x-y)f(t,y)dy\Bigg|\le
C\int_{2x}^{a}W_s(x-y)dy\le Cs^{-1/2}e^{-cx^2/s},\,\,\,s>0.
$$
Moreover, it is well known that
$$
\lim_{s\to 0^+}\int_{-\infty}^\infty W_s(x-y)f(t,y)dy=f(t,x).
$$
Putting together all the above estimates we obtain
\begin{align*}
\Bigg|\int_{x/2}^{2x}W_s(x-y)f(t-s,y)dy&-f(t,x)\Bigg|\le \Bigg|\int_{x/2}^{2x}W_s(x-y)f(t-s,y)dy-\int_{x/2}^{2x}W_s(x-y)f(t,y)dy\Bigg|\\
&+ \Bigg|\int_{x/2}^{2x}W_s(x-y)f(t,y)dy-\int_{-\infty}^\infty W_s(x-y)f(t,y)dy\Bigg|\\
&+ \Bigg|\int_{-\infty}^\infty W_s(x-y)f(t,y)dy-f(t,x)\Bigg|\\
&\le C\Big(s+s^{-1/2}e^{-cx^2/s}\Big)+\Bigg|\int_{-\infty}^\infty
W_s(x-y)f(t,y)dy-f(t,x)\Bigg|,\,\,\,s>0.
\end{align*}
We conclude that
$$
\lim_{s\to 0^+}\int_{x/2}^{2x}W_s(x-y)f(t-s,y)dy=f(t,x).
$$

      From \eqref{A0}, \eqref{A1}, \eqref{A2}, \eqref{A3} and \eqref{A4} we deduce that

      \begin{align*}
      \frac{\partial}{\partial t}u(t,x)&=\lim_{\epsilon\to 0^+}\bigg(\int_{\epsilon}^\infty\int_{x/2}^{2x}\frac{\partial}{\partial \tau} (W_\tau^\mu(x,y)- W_\tau(x-y))f(t-\tau,y)dy d\tau\\
    &+\int_{\epsilon}^\infty\int_0^{x/2}\frac{\partial}{\partial \tau}W_\tau^\mu(x,y)f(t-\tau,y)dy d\tau+\int_{\epsilon}^\infty\int_{2x}^\infty\frac{\partial}{\partial \tau}W_\tau^\mu(x,y)f(t-\tau,y)dy d\tau\\
&+\int_{\epsilon}^\infty\int_{x/2}^{2x}\frac{\partial}{\partial
\tau}W_\tau(x-y)f(t-\tau,y)dy d\tau\bigg)+f(t,x), \;\; t,x>0.
      \end{align*}
      We conclude that
      \begin{equation}\label{A5}
\partial_t u(t,x)=\displaystyle\lim_{\epsilon\to 0^+}\int_{\epsilon}^\infty\INT\frac{\partial}{\partial \tau}W_\tau^\mu(x-y)f(t-\tau,y)dy d\tau+f(t,x), \;\; t,x\in\DO.
      \end{equation}

      Assume now that $f\in C^2(\R\times\DO)$ and it has compact support. We consider the function
      $$
      H(t,x)=\INT\INT (W_\tau^\mu(x,y)- W_\tau(x-y))f(t-\tau,y)dy d\tau, \;\; t,x\in\DO.
      $$
      Note that there exists $0<a<b<\infty$ such that
      $$
      H(t,x)=\INT\int_a^b (W_\tau^\mu(x,y)- W_\tau(x-y))f(t-\tau,y)dy d\tau, \;\; t,x\in\DO.
      $$
      Let $t\in\R$. There exists $\tau_0\in\DO$ for which
      $$
          H(t,x)=\int_0^{\tau_0}\int_a^b (W_\tau^\mu(x,y)- W_\tau(x-y))f(t-\tau,y)dy d\tau, \;\; x\in\DO.
      $$
      By (\ref{P3}) we have that
      \begin{align}\label{T1}
      \frac{\partial}{\partial x}W_\tau^\mu(x,y)&=\frac{e^{-\frac{x^2+y^2}{4\tau}}}{(2\tau)^{1/2}}\bigg(x\left( \frac{y}{2\tau}\right)^2\left(\frac{xy}{2\tau}\right)^{-1/2}I_{\mu+1}\left(\frac{xy}{2\tau}\right)-\frac{x}{2\tau}\left(\frac{xy}{2\tau}\right)^{1/2}I_\mu\left(\frac{xy}{2\tau}\right)\nonumber\\
      &+(\mu+1/2)y(xy)^{-1/2}\frac{1}{\sqrt{2\tau}}I_\mu\left(\frac{xy}{2\tau}\right)\bigg), \;\; \tau,x,y\in\DO.
      \end{align}
       Then, from (\ref{P1}) it follows that
       \begin{align}\label{A5bis}
       &\left|\frac{\partial}{\partial x}W_\tau^\mu(x,y)- \frac{\partial}{\partial x}W_\tau(x-y) \right|\nonumber\\
&\le C\bigg[ \frac{e^{-\frac{x^2+y^2}{4\tau}}}{\tau^{3/2}}\left(\frac{xy}{2\tau}\right)^\mu\left(\left(\frac{xy}{2\tau}\right)^{3/2} y+\left(\frac{xy}{2\tau}\right)^{1/2}x+\left(\frac{xy}{2\tau}\right)^{-1/2} y \right)+\frac{|x-y|e^{-\frac{(x-y)^2}{4\tau}}}{\tau^{3/2}}\bigg]\nonumber\\
       &\le C \frac{e^{-\frac{x^2+y^2}{4\tau}}}{\tau^{3/2}}\left(x+y+\left(\frac{xy}{2\tau}\right)^{-1/2}x+\left(\frac{xy}{2\tau}\right)^{-3/2}y\right),
       \;\; \tau,x,y\in\DO,\; xy\le \tau,
       \end{align}
       and (\ref{P2}) implies that
       \begin{align}\label{A6bis}
        &\left|\frac{\partial}{\partial x}W_\tau^\mu(x,y)- \frac{\partial}{\partial x}W_\tau(x-y) \right|\nonumber\\
&=\bigg|\frac{e^{-\frac{(x-y)^2}{4\tau}}}{\sqrt{4\pi\tau}}\bigg\{\Bigg(x\Bigg(\frac{y}{2\tau}\Bigg)^2\Bigg(\frac{xy}{2\tau}\Bigg)^{-1}-\frac{x}{2\tau}+(\mu+1/2)y(xy)^{-1}
        \bigg)\left(1+O\left(\frac{\tau}{xy}\right)\right)\nonumber\\
&+\frac{x-y}{2\tau}\bigg\}\bigg|\nonumber\\
       &\le Ce^{-\frac{(x-y)^2}{4\tau}}\left(\frac{1}{x\sqrt{\tau}}+\frac{1}{y\sqrt{\tau}}\right),\;\;\tau,x,y\in\DO, \; xy\geq \tau.
       \end{align}
      From \eqref{A5bis} and \eqref{A6bis} it follows that
      \begin{align*}
     \int_0^{\tau_0}\int_a^b&\left|\frac{\partial}{\partial x}W_\tau^\mu(x,y) - \frac{\partial}{\partial x}W_\tau(x-y) \right||f(t-\tau,y)|dy d\tau\\
&\le C\int_a^b\int_0^{xy}{e^{-\frac{(x-y)^2}{4\tau}}}\left(\frac{1}{x\sqrt{\tau}}+\frac{1}{y\sqrt{\tau}}\right)d\tau dy\\
      &+\int_a^b\int_{xy}^{\max\{xy,\tau_0\}}\frac{e^{-\frac{x^2+y^2}{4\tau}}}{\tau^{3/2}}\left(x+y+\left(\frac{xy}{\tau}\right)^{-1/2}x+\left(\frac{xy}{\tau}\right)^{-3/2}y\right)d\tau dy<\infty, \;\; x\in\DO.
      \end{align*}


Hence,
$$
\frac{\partial}{\partial x} H(t,x)=\INT\INT
\left(\frac{\partial}{\partial x}W_\tau^\mu(x,y)-
\frac{\partial}{\partial x}W_\tau(x-y) \right)f(t-\tau,y) dy d\tau,
\;\; x\in\DO,
$$
and the last integral is absolutely convergent.

On the other hand, for every $x,y\in (0,\infty)$,
\begin{equation}\label{V3}
\frac{\partial^2}{\partial
x^2}[W_\tau^\mu(x,y)-W_\tau(x-y)]=\frac{\mu^2-1/4}{x^2}W_\tau^\mu(x,y)+\frac{\partial}{\partial
\tau}[W_\tau^\mu(x,y)-W_\tau(x-y)]
\end{equation}

By proceeding as above we get that
$$
\frac{\partial^2}{\partial x^2}{H}(t,x)=\INT\INT
\frac{\partial^2}{\partial
x^2}[W_\tau^\mu(x,y)-W_\tau(x-y)]f(t-\tau,y)dy d\tau, \;\;x\in\DO,
$$
and the last integral is absolutely convergent.

We now consider the function
$$
\mathcal{H}(t,x)=\INT\INT W_\tau(x-y)f(t-\tau, y)dy d\tau, \;\;
t,x\in\DO.
$$

Note that, by extending $f$ in the obvious way,
$$
\mathcal{H}(t,x)=\INT\int_{-\infty}^\infty W_\tau(x-y)f(t-\tau, y)dy
d\tau=\INT\int_{-\infty}^\infty W_\tau(y)f(t-\tau,x-y)dy d\tau, \;\;
t,x\in\DO.
$$
Then,
\begin{align*}
\frac{\partial}{\partial x}\mathcal{H}(t,x)&=\INT\int_{-\infty}^\infty W_\tau(y)\frac{\partial}{\partial x}f(t-\tau,x-y)dy d\tau\\
&=-\INT\int_{-\infty}^\infty W_\tau(y)\frac{\partial}{\partial
y}f(t-\tau,x-y)dy d\tau, \;\; t,x\in\DO,
\end{align*}
 and the last integral is
absolutely convergent.

Partial integration leads to

\begin{align*}
\frac{\partial}{\partial x}\mathcal{H}(t,x)&=-\lim_{\epsilon\to 0^+}\int_\epsilon^\infty\int_{-\infty}^\infty W_\tau(y)\frac{\partial}{\partial y}f(t-\tau,x-y)dy d\tau\\
&=\lim_{\epsilon\to 0^+}\int_\epsilon^\infty\int_{-\infty}^\infty
\frac{\partial}{\partial y}W_\tau(y)f(t-\tau,x-y)dy d\tau,
t,x\in\DO,
\end{align*}

In a similar way we can see that
$$
\frac{\partial^2}{\partial x^2}\mathcal{H}(t,x)=\lim_{\epsilon\to
0^+}\int_\epsilon^\infty\int_{-\infty}^\infty
\frac{\partial^2}{\partial y^2}W_\tau(y)f(t-\tau,x-y)dy d\tau,
\;\;t,x\in\DO.
$$

We conclude that, for $i=1,2$,
\begin{equation}\label{A6}
\frac{\partial^i}{\partial x^i} u(t,x)=\lim_{\epsilon\to
0^+}\int_\epsilon^\infty\int_{0}^\infty \frac{\partial^i}{\partial
x^i}W_\tau^\mu(x,y)f(t-\tau,y)dy d\tau, \;\;t,x\in\DO.
\end{equation}
By combining \eqref{A5} and \eqref{A6} we obtain
\begin{align*}
&\frac{\partial}{\partial t} u(t,x)-\frac{\partial^2}{\partial x^2} u(t,x)+\frac{\mu^2-1/4}{x^2}u(t,x)\\
&=\lim_{\epsilon\to 0^+}\int_\epsilon^\infty\int_{0}^\infty \Bigg(\partial_\tau-\partial_x^2+\frac{\mu^2-1/4}{x^2}\Bigg)W_\tau^\mu(x,y)f(t-\tau,y)dy d\tau+f(t,x)\\
&=f(t,x), \;\;t,x\in\DO.
\end{align*}
The other representations of the derivatives of $u$ as principal
values can be proved by proceeding  as above and by taking into
account \cite[Theorem 1.3,(A)]{PST}.

\section{Proof of Theorem \ref{Iteo2}, (1) and (2).}

Assume that $\mu>-1$. If $f$ is a measurable complex function
defined on $\R\times\DO$, we define $L_\mu f$ as follows
$$
(L_\mu f)(t,x)=\INT\INT W_\tau^\mu(x,y)f(t-\tau,y)dy d\tau,
$$
provided that the last integral exists with $(t,x)\in\R\times\DO$.

In Theorem \ref{teo1} we established that if $f\in C^2(\R\times\DO)$
and it has compact support, then $(\partial_t-\triangle_\mu)L_\mu
f=f$. Also, in a similar way we can see that if $f\in
C^2(\R\times\DO)$ and it has compact support, then
$L_\mu(\partial_t-\triangle_\mu) f=f$. In other words, we can see
that for good enough functions,
$L_\mu=(\partial_t-\triangle_\mu)^{-1}$.

Suppose that $f\in C^2(\R\times\DO)$ and it has compact support.
According to \cite[page 134]{Leb} we have that
$|\sqrt{z}J_\nu(z)|\le C$, $z\in (1, \infty)$, and
$|\sqrt{z}J_\nu(z)|\le Cz^{\nu+1/2}$, $z\in (0,1)$, when $\nu>-1$.
Let $z\in (0,\infty)$ and $t\in \mathbb{R}$. There exist
$0<a<b<\infty$ and $c>0$ such that
\begin{align*}
\INT\INT\INT|\sqrt{xz}J_\mu(xz)|&|W_\tau^\mu(x,y)||f(t-\tau,y)|dyd\tau dx\\
&\le C\int_a^b\int_0^c\int_0^\infty
|\sqrt{xz}J_\mu(xz)||W_\tau^\mu(x,y)|dxd\tau dy.
\end{align*}

Let $\mu>-1/2$. From (\ref{P1}) and (\ref{P2}) we deduce that
$$
W_\tau^\mu(x,y)\le C\frac{1}{\sqrt{\tau}}e^{-\frac{(x-y)^2}{4\tau}},
\,\,\,x,\,y,\,\tau\in (0,\infty).
$$
We also have that $|\sqrt{xz}J_\mu(xz)|\le C$, $x,z\in (0,\infty)$.
Then, we get
$$
\int_0^\infty |\sqrt{xz}J_\mu(xz)||W_\tau^\mu(x,y)|dx\le
C\int_{\mathbb{R}}\frac{1}{\sqrt{\tau}}e^{-\frac{(x-y)^2}{4\tau}}dx\le
C,\,\,\,\tau>0.
$$
Hence,
$$
\INT\INT\INT|\sqrt{xz}J_\mu(xz)||W_\tau^\mu(x,y)||f(t-\tau,y)|dyd\tau
dx<\infty.
$$

Assume that now $-1<\mu\le -1/2$. By using again (\ref{P1}) and
(\ref{P2}) we obtain that
$$
 W_\tau^\mu(x,y)\le C \left\{ \begin{array}{c}
    \frac{1}{\sqrt{\tau}}e^{-c (x-y)^2/\tau},\;\; 0<\tau<xy,\\
    \,\\
    \frac{(xy)^{\mu+1/2}}{\tau^{\mu+1}}e^{-c(x^2+y^2)/\tau}, \;\; \tau\ge xy.
     \end{array} \right.
     $$
Then, it follows that
\begin{align*}
\int_0^\infty |\sqrt{xz}&J_\mu(xz)|W_\tau^\mu(x,y)dx\le
C\Bigg(\int_0^1(xz)^{\mu+1/2}W_\tau^\mu(x,y)dx+\int_1^\infty
W_\tau^\mu(x,y)dx\Bigg)\\
&\le
C\Bigg(\int_0^{min\{1,\tau/y\}}\frac{(xz)^{\mu+1/2}(xy)^{\mu+1/2}}{(x^2+y^2)^{\mu+1}}dx+\int_{min\{1,\tau/y\}}^1
(xz)^{\mu+1/2}\frac{e^{-c(x-y)^2/\tau}}{\sqrt{\tau}}dx\\
&+\int_1^{max\{1,\tau/y\}}\frac{(xy)^{\mu+1/2}}{(x^2+y^2)^{\mu+1}}dx+\int_{max\{1,\tau/y\}}^\infty
\frac{e^{-c(x-y)^2/\tau}}{\sqrt{\tau}}dx\Bigg)\\
&\le
C\Bigg(\int_0^1x^{2\mu+1}dx+\frac{1}{\sqrt{\tau}}\int_0^1x^{\mu+1/2}dx+1+\int_{\mathbb{R}}\frac{e^{-c(x-y)^2/\tau}}{\sqrt{\tau}}dx\Bigg)\\
&\le C\Bigg(1+\frac{1}{\sqrt{\tau}}\Bigg), \,\,\,y\in
(a,b)\,\,\,\text{and}\,\,\,\tau\in (0,c).
\end{align*}
Hence,
$$
\INT\INT\INT|\sqrt{xz}J_\mu(xz)||W_\tau^\mu(x,y)||f(t-\tau,y)|dyd\tau
dx<\infty.
$$

This fact justifies the interchanges in the orders of integration to
get
\begin{align*}
h_\mu((L_\mu f)(t,x);x\to z)&=\INT h_\mu\left( \INT W_t^\mu(x,y)f(t-\tau,y)dy; x\to z\right)d\tau\\
&=\INT e^{-z^2\tau}h_\mu(f(t-\tau,y);y\to z)d\tau,
\end{align*}
because (see \cite[p. 195]{Wat})
$$
W_\tau^\mu(x,y)=\int_0^\infty
e^{-z^2\tau}\sqrt{xz}J_\mu(xz)\sqrt{yz}J_\mu(yz)dz,\,\,\,x,y,\tau\in
(0,\infty).
$$

Also, we have that, for certain $0<a<b<+\infty$ and
$-\infty<c<d<+\infty$,
\begin{align*}
&\int_{\R}\INT\INT e^{-z^2\tau}|\sqrt{yz}J_\mu(yz)||e^{-it\rho}||f(t-\tau,y)|dyd\tau dt\\
&\le C\INT e^{-z^2\tau}\int_{c+\tau}^{d+\tau}\int_a^b |f(t-\tau, y)|
dy dt d\tau <\infty,\,\,\, z\in
(0,\infty)\,\,\,\text{and}\,\,\,\rho\in\R.
\end{align*}
Note that, fixed $z\in (0,\infty)$,  $|\sqrt{yz}J_\mu(yz)|\le C$,
$y\in (a,b)$.

We denote, as usual, by $\mathcal{F}$ the Fourier transformation
defined by, for every $\phi\in L^1(\mathbb{R})$, by
$$
\mathcal{F}(\phi)(\rho)=\int_\mathbb{R}e^{-i\rho
t}\phi(t)dt,\,\,\,\rho\in \mathbb{R}.
$$

Then,
\begin{align*}
\mathcal{F}(h_\mu&((L_\mu f)(t,x);x\to z), t\to \rho)=\INT e^{-(z^2+i\rho)\tau}d\tau \mathcal{F}(h_\mu(f(t,x), x\to z);t\to\rho)\\
&=\frac{1}{z^2+i\rho}\mathcal{F}(h_\mu(f)(t,x); x\to z); t\to \rho),
\;\;z\in\DO\;\;\text{and}\;\; \rho\in\R.
\end{align*}
We define the space of functions $S_\mu$ as follows. A smooth
function $f$ on $\R\times \DO$ is in $S_\mu$ if and only if, for
every $m,k,l\in\N$,
$$
\sup_{t\in\R,
x\in\DO}(1+x^2)^m(1+t^2)^m\left|\frac{\partial^k}{\partial
t^k}\left(\frac{1}{x}\frac{\partial}{\partial x}\right)^{l}
(x^{-\mu-1/2}f(t,x)) \right|<\infty.
$$

By proceeding as above we can see that if $f\in S_\mu$ then the
integral defining $L_\mu(f)(t,x)$ is absolutely convergent, for
every $x\in\DO$ and $t\in\R$, and
$$
\mathcal{F}(h_\mu((L_\mu f)(t,x); x\to z);
t\to\rho)=\frac{1}{z^2+i\rho}\mathcal{F}(h_\mu(f(t,x), x\to z); t\to
\rho), \;\; z\in\DO,\; \rho\in\R.
$$

We consider the function space $C_{c,0}^\infty(\R)$ that consists of
all those $C^\infty(\R)$-functions $\phi$ such that supp$\,\phi$ is
compact and $\phi(t)=0$, $t\in (-r,r)$, for some $r>0$.
$C_{c,0}^\infty(\R)$ is a dense subspace of $L^2(\R)$. We define
$Z=\mathcal{F}(C_{c,0}^\infty(\R))$ and
$$
Z\otimes C_c^\infty(0,\infty)=\Bigg\{ \sum_{i=1}^n
\alpha_i\beta_i,\;\;\alpha_i\in Z, \beta_i\in
C_c^\infty(0,\infty),\; i=1,\dots, n, \; n\in\N\Bigg\}.
$$
 Since the Fourier transform $\mathcal{F}$ is an isometry on $L^2(\R)$, $Z$ is a dense subspace of $L^2(\R)$.
 Then $Z\otimes C_c^\infty(0,\infty)$ is a dense subset of $L^2(\R\times (0,\infty))$. If $\alpha\in Z$ and $\beta\in C_c^\infty(0,\infty)$, for a certain $r>0$,
 $$
 \Bigg|\frac{1}{z^2+i\rho}\mathcal{F}(\alpha)(\rho)h_\mu(\beta)(z)\Bigg|\le \frac{1}{r}|\mathcal{F}(\alpha)(\rho)h_\mu(\beta)(z)|,\,\,\,\rho\in \R\,\,\,\text{and}\,\,\,z\in (0,\infty),
 $$
 and hence
 $$
 \frac{1}{z^2+i\rho}\mathcal{F}(\alpha)(\rho)h_\mu(\beta)(z)\in L^2(\R\times (0,\infty))\cap L^1(\R\times (0,\infty)).
 $$
 It follows that, for every $f\in Z\otimes C_c^\infty(0,\infty)$,
 $$
 (L_\mu f)(t,x)=\mathcal{F}^{-1}(h_\mu(\frac{1}{z^2+i\rho}\mathcal{F}(h_\mu(f(s,y); y\to z); s\to \rho); z\to x); \rho\to t ).
 $$

According to \cite{Ze}, we have that, for every $\beta\in H_\mu$,
$\delta_\mu h_\mu (\beta)=-h_{\mu+1}(z\beta),$ where
$\delta_\mu=x^{\mu+1/2}\frac{d}{dx}x^{-\mu-1/2}$. Here $H_\mu$
denotes the space introduced by Zemanian \cite[Chapter 5]{Ze}
consisting of all those $\phi\in C^\infty\DO$ such that, for every
$m,k\in\N$,
$$
\sup_{x\in\DO}\bigg|(1+x^2)^m\left(\frac{1}{x}\frac{d}{d x}\right)^k
(x^{-\mu-1/2}\phi(x))\bigg|<\infty.
$$
 Since $z\beta\in H_{\mu+1}$, for every $\beta\in H_\mu$, we can write
 \begin{align*}
 \delta_{\mu+1}&\delta_\mu L_\mu (f)(t,x)\\
 &=\mathcal{F}^{-1}\left(h_{\mu+2}\left( \frac{z^2}{z^2+i\rho} \mathcal{F}(h_\mu(f(s,y);y\to z); s\to \rho); z\to x\right); \rho\to t\right),\;\;
 t\in\R,\,\,\,x\in \DO,
 \end{align*}
for each $f\in Z\otimes C_c^\infty(0,\infty)$.

We define the Riesz transformation $R_\mu$ by $R_\mu
f=\mathcal{F}^{-1}h_{\mu+2}\left( \frac{z^2}{z^2+i\rho}
\mathcal{F}h_\mu(f)\right), \;\; f\in L^2(\R\times\DO).$

Thus, $R_\mu f= \delta_{\mu+1}\delta_\mu L_\mu f$, $f\in Z\otimes
C_c^\infty(0,\infty)$, and $R_\mu$ is bounded from
$L^2(\R\times\DO)$ into itself.

Also, for every $f\in Z\otimes C_c^\infty(0,\infty)$, we have that
\begin{align*}
\partial_tL_\mu& (f)(t,x)\\
&=\mathcal{F}^{-1}\left(h_{\mu}\left( \frac{-i\rho}{z^2+i\rho}\mathcal{F}(h_\mu(f(s,y);y\to z); s\to \rho); z\to x\right); \rho\to t\right), \;\; t\in\R, \; x\in\DO.
\end{align*}

We define the Riesz transformation $\widetilde{R_\mu} $ by
$\widetilde{R_\mu} f=\mathcal{F}^{-1}h_{\mu}\left(
\frac{-i\rho}{z^2+i\rho}\mathcal{F}h_\mu(f)\right)$, $f\in
L^2(\R\times\DO)$. Thus, $\widetilde{R_\mu} f=\partial_tL_\mu f$,
$f\in Z\otimes C_c^\infty(0,\infty)$, and $\widetilde{R_\mu}$ is
bounded from $L^2(\R\times\DO)$ into itself.

Suppose that $f(t,x)=\alpha(t)\beta(x)$, $t\in \R$ and $x\in
(0,\infty)$, where $\alpha\in Z$ and $\beta\in
C_c^\infty(0,\infty)$. We have that
\begin{align*}
\frac{\partial^2}{\partial x^2}&\int_0^\infty\int_0^\infty W_\tau(x-y)\alpha(t-\tau)\beta(y)dyd\tau=\frac{\partial^2}{\partial x^2}\int_{-\infty}^x\int_0^\infty W_\tau(y)\alpha(t-\tau)d\tau\beta(x-y)dy\nonumber\\
&=\int_{-\infty}^x\int_0^\infty
W_\tau(y)\alpha(t-\tau)d\tau\frac{\partial^2}{\partial
x^2}\beta(x-y)dy,\,\,\,t\in \R\,\,\,\text{and}\,\,\,x\in (0,\infty),
\end{align*}
and the last integral is absolutely convergent. Then, we can write,
for every $t\in \mathbb{R}$ and $x\in \DO$,
\begin{align*}
\frac{\partial^2}{\partial x^2}&\int_0^\infty\int_0^\infty
W_\tau(x-y)\alpha(t-\tau)\beta(y)dyd\tau=\lim_{\varepsilon\to
0^+}\int_{\Omega_\varepsilon}
W_\tau(y)\alpha(t-\tau)\frac{\partial^2}{\partial
x^2}\beta(x-y)dyd\tau,
\end{align*}
where $\Omega_\varepsilon=\{(\tau,y)\in (0,\infty)\times\mathbb{R}:|y|+\sqrt{\tau}>\varepsilon\}$. By partial
integration as in the proof of \cite[Theorem 2.3, (B)]{PST} we
obtain
\begin{align*}
\frac{\partial^2}{\partial x^2}&\int_0^\infty\int_0^\infty W_\tau(x-y)\alpha(t-\tau)\beta(y)dyd\tau=\lim_{\varepsilon\to 0^+}\int_{\Omega_\varepsilon} \frac{\partial^2}{\partial y^2}W_\tau(y)\alpha(t-\tau)\beta(x-y)dyd\tau\nonumber\\
&+f(t,x)\frac{1}{\sqrt{\pi}}\int_1^\infty e^{-w^2/4}dw, \,\,\,t\in
\R\,\,\,\text{and}\,\,\,x\in (0,\infty).
\end{align*}
By proceeding as in the proof of Theorem \ref{teo1} we can see that,
for every $f\in Z\otimes C_c^\infty(0,\infty)$,
\begin{align*}
R_\mu(f)(t,x)&=\lim_{\varepsilon\to 0^+}\int_{\Omega_\varepsilon(x)} \delta_{\mu+1}\delta_\mu W_\tau(x,y)\alpha(t-\tau)\beta(y)dyd\tau\nonumber\\
&+f(t,x)\frac{1}{\sqrt{\pi}}\int_1^\infty e^{-w^2/4}dw, \,\,\,t\in
\R\,\,\,\text{and}\,\,\,x\in (0,\infty),
\end{align*}
where $\Omega_\varepsilon(x)=\{(\tau,y)\in (0,\infty)\times
(0,\infty):\,|y-x|+\sqrt{\tau}>\varepsilon\}$, for every
$x\in (0,\infty)$.

In a similar way we can show that, for every $f\in Z\otimes
C_c^\infty(0,\infty)$,
\begin{align*}
\widetilde{R_\mu}(f)(t,x)&=\lim_{\varepsilon\to 0^+}\int_{\Omega_\varepsilon(x)} \frac{\partial}{\partial \tau} W_\tau^\mu(x,y)\alpha(t-\tau)\beta(y)dyd\tau\nonumber\\
&+f(t,x)\frac{1}{\sqrt{\pi}}\int_0^1 e^{-w^2/4}dw, \,\,\,t\in
\R\,\,\,\text{and}\,\,\,x\in (0,\infty).
\end{align*}

In order to prove Theorem \ref{Iteo2}, (ii), we use
Calder\'on-Zygmund theory on the space of homogeneous type
$(\R\times (0,\infty),m,d)$, where $m$ and $d$ denote the Lebesgue
measure and the parabolic distance, respectively, on $\R\times
(0,\infty)$. We now recall the definitions and results that will be
useful in the sequel. We describe now Calder\'on-Zygmund theory in
the more general vectorial setting because we will use it in the
proof of Theorem \ref{Iteo3}.

Suppose that $X$ and $Y$ are Banach spaces. By $\mathcal{L}(X,Y)$ we
denote the space of bounded operators from $X$ to $Y$. If $1\le
p<\infty$ we represent by $L^p(\R\times (0,\infty),X)$ and
$L^{p,\infty}(\R\times (0,\infty),X)$ the Bochner Lebesgue
$L^p$-space and weak Bochner Lebesgue $L^{p,\infty}$-space. Assume
that $T$ is a bounded operator from $L^p(\R\times (0,\infty),X)$
into $L^p(\R\times (0,\infty),Y)$, for some $1<p<\infty$, satisfying
that
\begin{equation}\label{N1}
T(f)(t,x)=\int_{\R\times
(0,\infty)}K(t,x;s,y)(f(s,y))dsdy,\,\,\,(t,x)\notin \text{supp}\,f,
\end{equation}
for every $f\in S$, where $S$ represents a linear space that is
dense in $L^q(\R\times (0,\infty),X)$, for every $1\le q<\infty$.
Here
$$
K:[(\R\times (0,\infty)\times(\R\times (0,\infty)]\setminus
D\longrightarrow \mathcal{L}(X,Y),
$$
is a strongly measurable function, being
$$
D=\{(t,x;s,y)\in (\R\times (0,\infty))\times(\R\times (0,\infty)):
(t,x)= (s,y)\}.
$$

We say that $K$ is a standard Calder\'on-Zygmund kernel in
$(\R\times (0,\infty),m,d)$ when the following properties hold

(a)$\|K(t,x;s,y)\|_{\mathcal{L}(X,Y)}\le
\frac{C}{d((t,x),(s,y))^{3}}$, $(t,x)\neq (s,y)$.

\vspace{2mm}

 (b) provided that
$d((t,x),(s_0,y_0))>d((s,y),(s_0,y_0))$,
$$
\|K(t,x;s,y)-K(t,x;s_0,y_0)\|_{\mathcal{L}(X,Y)}+\|K(s,y;t,x)-K(s_0,y_0;t,x)\|_{\mathcal{L}(X,Y)}\le
C\frac{d((s,y),(s_0,y_0))}{d((t,x),(s_0,y_0))^{4}}.
$$

If $1<p<\infty$, a weight $w$ on $\R\times (0,\infty)$ is in the
Muckenhoupt class $A_p^{*}(\R\times (0,\infty))$ when there exists
$C>0$ such that
$$
\frac{1}{|B|}\int_Bw(t,x)dtdx\Bigg(\frac{1}{|B|}\int_Bw(t,x)^{1/(1-p)}dtdx\Bigg)^{p-1}\le
C,
$$
for every ball (with respect to $d$) in $\R\times (0,\infty)$.

A weight $w$ is in $A_1^*(\R\times (0,\infty))$ when there exists
$C>0$ such that, for a.e. $(t,x)\in \R\times (0,\infty)$,
$$
\frac{1}{|B|}\int_Bw(s,y)dsdy\le Cw(t,x),
$$
for every ball $B$ (with respect to $d$) containing $(t,x)$.

The Calder\'on-Zygmund Theorem says that if $T$ satisfies the above
properties where $K$ in (\ref{N1}) is a standard Calder\'on-Zygmund
kernel, then the operator $T$ can be extended,

(a) for every $1<q<\infty$ and $w\in A_q^*(\R\times (0,\infty))$,
from $L^p(\R\times (0,\infty),X)\cap L^q(\R\times (0,\infty),w,X)$
to $L^q(\R\times (0,\infty),w,X)$ as a bounded operator from
$L^q(\R\times (0,\infty),w,X)$ into $L^q(\R\times (0,\infty),w,Y)$;

(b) for every $w\in A_1^*(\R\times (0,\infty))$, from $L^p(\R\times
(0,\infty),X)\cap L^1(\R\times (0,\infty),w,X)$ to $L^1(\R\times
(0,\infty),w,X)$ as a bounded operator from $L^1(\R\times
(0,\infty),w,X)$ into $L^{1,\infty}(\R\times (0,\infty),w,Y)$.

Moreover, the maximal operator given by
$$
T^*(f)(t,x)=\sup_{\varepsilon>0}\Bigg\|\int_{d((t,x),(s,y))>\varepsilon}K(t,x;s,y)(f(s,y))dsdy\Bigg\|_Y,
$$
defines a bounded operator from

(a) $L^q(\R\times (0,\infty),w,X)$ into $L^q(\R\times
(0,\infty),w)$, for every $1<q<\infty$ and $w\in A^*_q(\R\times
(0,\infty))$;

(b) $L^1(\R\times (0,\infty),w,X)$ into $L^{1,\infty}(\R\times
(0,\infty),w)$, for every $w\in A^*_1(\R\times (0,\infty))$.

A complete study about vector valued Calder\'on-Zygmund theory on
spaces of homogeneous type can be encountered in \cite{RRT},
\cite{RT1} and \cite{RT2}.

We have that, for every $f\in Z\otimes C_c^\infty(0,\infty)$,
$$
R_\mu f(t,x)=\INT\INT\delta_{\mu+1}\delta_{\mu} W_\tau^\mu(x,y)
f(t-\tau, y) d\tau dy, \;\; (t,x)\not\in \text{supp}{ f},$$ and
$$
\widetilde{R_\mu}f(t,x)=\INT\INT\partial_\tau W_\tau^\mu(x,y)
f(t-\tau, y)d\tau  dy\;\; (t,x)\not\in \text{supp}{ f}.
$$
We consider the kernel functions defined as follows

$$
K_\mu(t,x;\tau,y)=\delta_{{\mu+1}}\delta_\mu W_{t-\tau}^\mu (x,y)
\chi_{\DO}(t-\tau), \;\; x,y\in\DO, \;\; t,\tau\in\R,
$$
and

$$
\widetilde{K_\mu}(t,x;\tau,y)=-\partial_\tau W_{t-\tau}^\mu
(x,y)\chi_{\DO}(t-\tau), \;\; x,y\in\DO, \;\; t,\tau\in\R.
$$

It is clear that, for every $f\in Z\otimes C^\infty_c(0,\infty)$,
$R_\mu f(t,x)=\int_{\R}\INT K_\mu(x,t;y,\tau)f(\tau,y)dyd\tau,$ and
$\widetilde{R_\mu}f(t,x)=\int_{\R}\INT
\widetilde{K_\mu}(x,t;y,\tau)f(\tau,y)dyd\tau,$
$(t,x)\not\in\text{supp} f$.

We remark that $d((t,x), (s,y))=|x-y|+\sqrt{|t-s|}, $ $t,s\in\R$ and
$x,y\in\DO$.

\begin{proposition}
    Let $\mu>1/2$ or $\mu=-1/2$. The kernels $K_\mu$ and $\widetilde{K_\mu}$ are standard Calder\'on-Zygmund with respect to the homogeneous type space $(\R\times\DO,m,d)$.

\end{proposition}

\begin{proof}
Firstly we analyze $K_\mu$. We consider the function,
 $$\mathbb{K}_\mu(x,y,s)=\delta_{\mu+1}\delta_\mu W_s^\mu(x,y)\chi_{(0,\infty)}(s),\;\; x,y\in\DO, \text{and} \;\; s\in\R.$$
 According to (\ref{P3}) we have that
\begin{align}\label{eq2.1}
\mathbb{K}&_\mu(x,y,s)=x^{\mu+5/2}\left(\frac{1}{x}\frac{\partial}{\partial x}\right)^{2}\left(\bigg(\frac{xy}{2s}\bigg)^{-\mu} I_\mu\bigg(\frac{xy}{2s}\bigg)e^{-\frac{x^2+y^2}{4s}}\right)\frac{y^{\mu+1/2}}{(2s)^{\mu+1}}\nonumber\\
&=\frac{x^{\mu+5/2}y^{\mu+1/2}}{(2s)^{\mu+1}}\left(\frac{1}{x}\frac{\partial}{\partial x}\right)\left[ \left(\frac{y}{2sx}\bigg(\frac{xy}{2s}\bigg)^{-\mu} I_{\mu+1}\bigg(\frac{xy}{2s}\bigg)-\frac{1}{2s}\bigg(\frac{xy}{2s}\bigg)^{-\mu}I_\mu\bigg(\frac{xy}{2s}\bigg)\right)e^{-\frac{x^2+y^2}{4s}} \right]\nonumber\\
&=\frac{x^{\mu+5/2}y^{\mu+1/2}}{(2s)^{\mu+1}}\bigg[ \left(\frac{y}{2s}\right)^3\bigg(\frac{xy}{2s}\bigg)^{-\mu-1}I_{\mu+2}\bigg(\frac{xy}{2s}\bigg)\frac{1}{x}-\frac{y}{4s^2x}\bigg(\frac{xy}{2s}\bigg)^{-\mu}I_{\mu+1}\bigg(\frac{xy}{2s}\bigg)\nonumber\\
&-\frac{1}{2s}\bigg(\frac{y}{2sx}\bigg(\frac{xy}{2s}\bigg)^{-\mu}I_{\mu+1}\bigg(\frac{xy}{2s}\bigg)
-\frac{1}{2s}\bigg(\frac{xy}{2s}\bigg)^{-\mu}I_{\mu}\bigg(\frac{xy}{2s}\bigg) \bigg)\bigg]e^{-\frac{x^2+y^2}{4s}}\nonumber\\
&=\frac{x^{\mu+5/2}y^{\mu+1/2}}{(2s)^{\mu+3}}\bigg(\frac{y^3}{2s x}\bigg(\frac{xy}{2s}\bigg)^{-\mu-1}I_{\mu+2}\bigg(\frac{xy}{2s}\bigg)-\frac{2y}{x}\bigg(\frac{xy}{2s}\bigg)^{-\mu}I_{\mu+1}\bigg(\frac{xy}{2s}\bigg)\nonumber\\
&\hspace{5mm}+\bigg(\frac{xy}{2s}\bigg)^{-\mu}I_{\mu}\bigg(\frac{xy}{2s}\bigg)
\bigg)e^{-\frac{x^2+y^2}{4s}}, \;\; x,y,s\in\DO.
\end{align}
By taking into account (\ref{P1}) we obtain
\begin{align}\label{X1}
|\mathbb{K}_\mu(x,y,s)|&\le C  \left( \frac{xy^3}{s^{7/2 }}+\frac{xy}{s^{5/2}}+\frac{x^2}{s^{5/2}}\right)e^{-\frac{x^2+y^2}{4s}}\nonumber\\
&\le   \frac{C}{(\sqrt{s}+|x-y|)^3}, \;\; x,y,s\in\DO, \;\; xy\le s.
\end{align}
On the other hand, we can write
\begin{align*}
\mathbb{K}_\mu&(x,y,s)=\frac{x^{\mu+5/2}y^{\mu+1/2}}{(2s)^{\mu+2}}\bigg(\frac{y^3}{4s^2 x}\bigg(\frac{xy}{2s}\bigg)^{-\mu-3/2}\sqrt{\frac{xy}{2s}}I_{\mu+2}\bigg(\frac{xy}{2s}\bigg)\\
&-\frac{y}{sx}\bigg(\frac{xy}{2s}\bigg)^{-\mu-1/2}\sqrt{\frac{xy}{2s}}I_{\mu+1}\bigg(\frac{xy}{2s}\bigg)+\frac{1}{2s}\bigg(\frac{xy}{2s}\bigg)^{-\mu-1/2}\sqrt{\frac{xy}{2s}}I_{\mu}\bigg(\frac{xy}{2s}\bigg)
\bigg)e^{-\frac{x^2+y^2}{4s}}, \;\; x,y,s\in\DO.
\end{align*}
By (\ref{P2}) we deduce that
 \begin{align}\label{eq2.2}
 \mathbb{K}&_\mu(x,y,s)=\frac{x^{\mu+5/2}y^{\mu+1/2}}{\sqrt{2\pi}(2s)^{\mu+2}}\bigg(\frac{xy}{2s}\bigg)^{-\mu}\bigg( \frac{y^3}{4s^2x}\bigg(\frac{xy}{2s}\bigg)^{-3/2}-\frac{y}{sx}\bigg(\frac{xy}{2s}\bigg)^{-1/2}\nonumber\\
 &+\frac{1}{2s}\bigg(\frac{xy}{2s}\bigg)^{-1/2}\bigg)\left(1+O\left(\frac{s}{xy}\right) \right) e^{-\frac{(x-y)^2}{4s}}\nonumber\\
 &=\frac{1}{\sqrt{2\pi}(2s)^2}\left(\frac{y^2-2yx+x^2}{(2s)^{1/2}}+O\left(\frac{y\sqrt{s}}{x} \right)+O(\sqrt{s})+O\left( \frac{x\sqrt{s}}{y}\right) \right)e^{-\frac{(x-y)^2}{4s}}\nonumber\\
&=\frac{1}{\sqrt{2\pi}(2s)^{3/2}}\left( \frac{(y-x)^2}{2s}+O\left(
\frac{y}{x}\right)+O(1)+O\left(\frac{x}{y}\right)\right)e^{-\frac{(x-y)^2}{4s}},
\;\; s,x,y\in\DO.
 \end{align}
 Hence, if $s,x,y\in\DO,$ and $xy\ge s$, then
 \begin{align}\label{AMI1}
 |\mathbb{K}_\mu(x,y,s)|&\le \frac{Ce^{-\frac{c(x-y)^2}{s}}}{s^{3/2}}\le \frac{C}{(\sqrt{s}+|x-y|)^3},\;\; x/2<y<2x,
 \end{align}
 and
 \begin{align}\label{AMI2}
 | \mathbb{K}_\mu(x,y,s)|&\le \frac{C}{s^{3/2}}\left( 1+\frac{x^2+y^2}{xy}\right) e^{-c\frac{(x-y)^2}{s}}\nonumber\\
 &\le \frac{C}{s^{3/2}}\left(1+\frac{\max\{x,y\}^2}{s}\right)e^{-c\frac{\max\{x,y\}^2}{s}}e^{-c\frac{(x-y)^2}{s}}\nonumber\\
&\le \frac{C}{(\sqrt{s}+|x-y|)^3},\;\; 0<y<x/2,\;\;\text{or}\;\;
2x<y<\infty.
 \end{align}

 We conclude that
 $$
  |K_\mu(x,t;y,\tau)|\le \frac{C}{(\sqrt{|t-\tau|}+|x-y|)^3},\;\; x,y\in\DO, \text{and} \;\; t, \tau\in\R.
 $$

 According to \eqref{eq2.2} we have that $\displaystyle\lim_{s\to 0^+}\mathbb{K}_\mu(x,y,s)=0$, $x,y\in \DO$, $x\neq y$. Hence, $K_\mu$ is a continuous function on
$[(\DO\times\R)\times(\DO\times\R)]\setminus D$.

 By using that $\frac{\partial}{\partial x}=\delta_{\mu+2}+\frac{\mu+5/2}{x}$ and (\ref{P3}), from \eqref{eq2.1} it follows that

 \begin{align}\label{X2}
  \partial&_x \mathbb{K}_\mu(x,y,s)=\frac{x^{\mu+7/2}y^{\mu+1/2}}{(2s)^{\mu+3}}\bigg(-\frac{y^3}{4s^2 x}\bigg( \frac{xy}{2s}\bigg)^{-\mu-1}I_{\mu+2}\bigg( \frac{xy}{2s}\bigg)
  +\frac{y}{sx}\bigg( \frac{xy}{2s}\bigg)^{-\mu } I_{\mu+1}\bigg(
  \frac{xy}{2s}\bigg)\nonumber\\
  &-\frac{1}{2s}\bigg( \frac{xy}{2s}\bigg)^{-\mu}I_\mu\bigg( \frac{xy}{2s}\bigg)
  +\frac{y^5}{(2s)^3x}\bigg( \frac{xy}{2s}\bigg)^{-\mu-2}I_{\mu+3}\bigg(
  \frac{xy}{2s}\bigg)-\frac{y^3}{2s^2x}\bigg( \frac{xy}{2s}\bigg)^{-\mu-1}I_{\mu+2}\bigg(
  \frac{xy}{2s}\bigg)\nonumber\\
 &
  +\frac{y}{2sx}\bigg( \frac{xy}{2s}\bigg)^{-\mu} I_{\mu+1}\bigg( \frac{xy}{2s}\bigg)\bigg)e^{-\frac{x^2+y^2}{4s}}+\frac{\mu+\frac{5}{2}}{x}\frac{x^{\mu+5/2}y^{\mu+1/2}}{(2s)^{\mu+3}}\bigg(\frac{y^3}{2s x}\bigg(\frac{xy}{2s}\bigg)^{-\mu-1}I_{\mu+2}\bigg(\frac{xy}{2s}\bigg)\nonumber\\
  &-\frac{2y}{x}\bigg(\frac{xy}{2s}\bigg)^{-\mu}I_{\mu+1}\bigg(\frac{xy}{2s}\bigg)+\bigg(\frac{xy}{2s}\bigg)^{-\mu}I_{\mu}\bigg(\frac{xy}{2s}\bigg) \bigg)e^{-\frac{x^2+y^2}{4s}}\nonumber\\
 &\,\nonumber\\
  &=\frac{x^{\mu+5/2}y^{\mu+1/2}}{(2s)^{\mu+3}} e^{-\frac{x^2+y^2}{4s}}\bigg[ \frac{y^5}{(2s)^3}\left(\frac{x}{2s} \right)^{-\mu-2}I_{\mu+3}\bigg(\frac{xy}{2s}\bigg)+\bigg(\frac{xy}{2s}\bigg)^{-\mu-1}I_{\mu+2}\bigg(\frac{xy}{2s}\bigg)\bigg( -\frac{3y^3}{4s^2}
  +\frac{(\mu+\frac{5}{2}) y^3}{2sx^2}\bigg) \nonumber\\
  &+\bigg(\frac{xy}{2s}\bigg)^{-\mu}
  I_{\mu+1}\bigg(\frac{xy}{2s}\bigg)\bigg(\frac{3y}{2s}-2\frac{(\mu+\frac{5}{2})y}{x^2}\bigg)+\bigg(\frac{xy}{2s}\bigg)^{-\mu} I_\mu\bigg(\frac{xy}{2s}\bigg)\bigg(-\frac{x}{2s}+\frac{\mu+\frac{5}{2}}{x}\bigg) \bigg],\;\; x,y,s\in\DO,
  \end{align}
 and since $\partial_y=\delta_\mu+\frac{\mu+1/2}{y}$,

 \begin{align}\label{eq2.3}
  \partial_y \mathbb{K}_\mu(x,y,s)&=\frac{x^{\mu+5/2}y^{\mu+1/2}}{(2s)^{\mu+2}}\bigg[\frac{4y^3}{(2s)^3}\bigg(\frac{xy}{2s}\bigg)^{-\mu-2}I_{\mu+2}\bigg(\frac{xy}{2s}\bigg)+\frac{y^4x}{(2s)^4}\bigg(\frac{xy}{2s}\bigg)^{-\mu-2}I_{\mu+3}\bigg(\frac{xy}{2s}\bigg)\nonumber\\
  &-\frac{y}{s^2}\bigg(\frac{xy}{2s}\bigg)^{-\mu-1}I_{\mu+1}\bigg(\frac{xy}{2s}\bigg)
  -\frac{y^2x}{4s^3}\bigg(\frac{xy}{2s}\bigg)^{-\mu-1}I_{\mu+2}\bigg(\frac{xy}{2s}\bigg)+\frac{x}{(2s)^2}\bigg(\frac{xy}{2s}\bigg)^{-\mu}I_{\mu+1}\bigg(\frac{xy}{2s}\bigg)\nonumber\\
  &+\bigg(\frac{y^4}{(2s)^3}\bigg(\frac{xy}{2s}\bigg)^{-\mu-2}I_{\mu+2}\bigg(\frac{xy}{2s}\bigg)-\frac{y^2}{2s^2}\bigg(\frac{xy}{2s}\bigg)^{-\mu-1}I_{\mu+1}\bigg(\frac{xy}{2s}\bigg)+\frac{1}{2s}\bigg(\frac{xy}{2s}\bigg)^{-\mu}I_{\mu}\bigg(\frac{xy}{2s}\bigg)\bigg)\nonumber\\
  &\cdot\left(-\frac{y}{2s}+\frac{\mu+1/2}{y} \right) \bigg] e^{-\frac{x^2+y^2}{4s}}\nonumber\\
&\,\nonumber\\
  &=\frac{x^{\mu+5/2}y^{\mu+1/2}}{(2s)^{\mu+3}}e^{-\frac{x^2+y^2}{4s}}\bigg[\frac{xy^4}{(2s)^3}\bigg(\frac{xy}{2s}\bigg)^{-\mu-2}I_{\mu+3}\bigg(\frac{xy}{2s}\bigg)\nonumber\\
  &+\bigg(\frac{xy}{2s}\bigg)^{-\mu-1}I_{\mu+2}\bigg(\frac{xy}{2s}\bigg)\left(\frac{4y^2}{2sx}-\frac{xy^2}{2s^2}+\frac{y^3}{2sx}\left(-\frac{y}{2s}+\frac{\mu+1/2}{y}\right)\right)\nonumber\\&+\bigg(\frac{xy}{2s}\bigg)^{-\mu}I_{\mu+1}\bigg(\frac{xy}{2s}\bigg)\left(-\frac{4}{x}+\frac{x}{2s}-\frac{2y}{x}\left(-\frac{y}{2s}+\frac{\mu+1/2}{y}\right)\right)\nonumber\\
  &+\bigg(\frac{xy}{2s}\bigg)^{-\mu}I_{\mu}\bigg(\frac{xy}{2s}\bigg)\left(-\frac{y}{2s}+\frac{\mu+1/2}{y}\right)\bigg],\;\; x,y,s\in\DO.
 \end{align}
 It is clear that \begin{equation} \label{eq2.4}
 \partial_x \mathbb{K}_\mu(x,y,s)=\partial_y \mathbb{K}_\mu(x,y,s)=0,\;\; x,y\in\DO,\;\; s\in(-\infty,0].
 \end{equation}

 Now, we estimate $\partial_x\mathbb{K}_\mu(x,y,s)$. By (\ref{P1}), \eqref{X2} leads that
 \begin{align*}
 |\partial&_x\mathbb{K}_\mu(x,y,s)|\le C \frac{x^{\mu+5/2}y^{\mu+1/2}}{s^{\mu+3}} e^{-\frac{x^2+y^2}{4s}}\bigg( \frac{y^5}{(2s)^3}\frac{xy}{2s}+\left( \frac{y^3}{s^2}+\frac{y^3}{sx^2}\right)\frac{xy}{2s}+\left(\frac{y}{s}+\frac{y}{x^2}\right)\frac{xy}{2s}+\frac{x}{s}+\frac{1}{x}\bigg)\\
 &\le C \bigg(\frac{y^{\mu+13/2}x^{\mu+7/2}}{\left(s+ x^2+y^2\right)^{\mu+7}}+ \frac{y^{\mu+9/2}x^{\mu+7/2}}{\left(s+ x^2+y^2\right)^{\mu+6}}+ \frac{y^{\mu+9/2}x^{\mu+3/2}}{\left(s+ x^2+y^2\right)^{\mu+5}}\\
 &+ \frac{y^{\mu+5/2}x^{\mu+7/2}}{\left(s+ x^2+y^2\right)^{\mu+5}}
 + \frac{y^{\mu+5/2}x^{\mu+3/2}+y^{\mu+1/2}x^{\mu+7/2}}{\left(s+ x^2+y^2\right)^{\mu+4}}+ \frac{y^{\mu+1/2}x^{\mu+3/2}}{\left(s+ x^2+y^2\right)^{\mu+3}}\bigg) \\
 &\le \frac{C}{(\sqrt{s}+x+y)^4}\le\frac{C}{(\sqrt{s}+|x-y|)^4},\;\; s,x,y\in\DO \;\;\text{and}\;\; \frac{xy}{s}\le 1.
 \end{align*}

 On the other hand, (\ref{P2}) allows us to deduce from \eqref{X2}
 that

 \begin{align*}
 \partial_x&\mathbb{K}_\mu (x,y,s)=\frac{x^2}{(2s)^{5/2}}\frac{e^{-\frac{(x-y)^2}{4s}}}{\sqrt{2\pi}}\bigg[ \bigg\{ \frac{y^5}{(2s)^3}\left(\frac{xy}{2s}\right)^{-2}+\left(\frac{xy}{2s}\right)^{-1}\left(-\frac{3y^3}{4s^2}+\frac{\mu+5/2}{2s}\frac{y^3}{x^2}\right)\\
    &+\frac{3y}{2s}-\frac{2(\mu+5/2)y}{x^2}-\frac{x}{2s}+\frac{\mu+5/2}{x}\bigg\}   \left(1+ O\left(\frac{s}{xy}\right)\right)\bigg]\\
     &=\frac{x^2}{(2s)^{5/2}}\frac{e^{-\frac{(x-y)^2}{4s}}}{\sqrt{2\pi}}\left[\frac{y^3}{2sx^2}-\frac{3y^2}{2sx}+\frac{3y}{2s}-\frac{x}{2s}+(\mu+5/2)\left(\frac{y^2}{x^3}-\frac{2y}{x^2}+\frac{1}{x} \right)\right]O(1)\\
    &=\frac{x^2}{(2s)^{5/2}}\frac{e^{-\frac{(x-y)^2}{4s}}}{\sqrt{2\pi}}\left[(\mu+5/2)\frac{(x-y)^2}{x^3}+\frac{(y-x)^3}{2sx^2}\right] O(1), \;\; x,y,s\in\DO,\;\text{and}\;\frac{xy}{s}\ge 1.
 \end{align*}
 Then,  $ | \partial_x\mathbb{K}_\mu (x,y,s)|\le C\frac{e^{-c\frac{(x-y)^2}{s}}}{s^{5/2}}\left(\frac{(x-y)^2}{x}+\frac{|x-y|^3}{s}\right)$, $x,y,s\in\DO$ and $\frac{xy}{s}\geq 1$.
 We have that, for every $x,y,s\in\DO$, and $xy\ge s$,
 \begin{equation}\label{eq2.6}
  \frac{1}{x}\le C \left\{ \begin{array}{c}
  \frac{1}{\sqrt{xy}}\le \frac{C}{\sqrt{s}},\;\;\;\;\;\;\quad \quad\;\; x>y/2,\\
  \,\\
  \frac{y}{s}\le \frac{|y-x|+x}{s}\le\frac{2|y-x|}{s}, \;\;0<x<\frac{y}{2}.
  \end{array} \right.
 \end{equation}
 Hence, we get
 \begin{align*}
 |\partial_x\mathbb{K}_\mu(x,y,s)|\le C e^{-\frac{(x-y)^2}{4s}}\left( \frac{(x-y)^3}{s^{7/2}}+\frac{(x-y)^2}{s^3}\right)\le \frac{C}{(\sqrt{s}+|x-y|)^4}, \;\;x,y,s\in\DO, \; xy\ge s.
 \end{align*}
We conclude that  \begin{equation}\label{eq2.7}
|\partial_x\mathbb{K}_\mu(x,y,s)|\le  \frac{C}{(\sqrt{s}+|x-y|)^4},
\;\;x,y,s\in\DO.
\end{equation}

We now estimate $\partial_y \mathbb{K}_\mu(x,y,s)$. By (\ref{P1}),
from \eqref{eq2.3} we deduce that
\begin{align*}
 |\partial_y\mathbb{K}_\mu(x,y,s)|&\le  Ce^{-\frac{x^2+y^2}{4s}}\bigg(  \frac{x^{\mu+9/2}y^{\mu+4/2}}{s^{\mu+7}}+ \frac{x^{\mu+5/2}y^{\mu+7/2}}{s^{\mu+5}}+ \frac{x^{\mu+5/2}y^{\mu+11/2}}{s^{\mu+6}}
+\frac{x^{\mu+5/2}y^{\mu+3/2}}{s^{\mu+4}}\\&+ \frac{x^{\mu+9/2}y^{\mu+3/2}}{s^{\mu+5}}+ \frac{x^{\mu+9/2}y^{\mu+7/2}}{s^{\mu+6}}+ (\mu+\frac{1}{2})\frac{x^{\mu+5/2}y^{\mu-1/2}}{s^{\mu+3}}\bigg)\\
&\le\frac{C}{(\sqrt{s}+|x-y|)^4}, \;\; s,x,y\in\DO, \; xy\le s.
\end{align*}
Note that in the last estimate we use that $\mu>1/2$ or $\mu=-1/2$.

By \eqref{eq2.3} we can write
\begin{align*}
\partial_y\mathbb{K}_\mu(x,y,s)&= e^{-\frac{x^2+y^2}{4s}}\frac{x^2}{(2s)^{5/2}}\bigg(\sqrt{\frac{xy}{2s}}I_{\mu+3}\left(\frac{xy}{2s}\right)\frac{y^2}{2sx}\\
&+\sqrt{\frac{xy}{2s}}I_{\mu+2}\left(\frac{xy}{2s}\right)\left(\frac{4y}{x^2}-\frac{y}{s}+\frac{y^2}{x^2}\left(-\frac{y}{2s}+\frac{\mu+1/2}{y}\right)\right)\\
&+\sqrt{\frac{xy}{2s}}I_{\mu+1}\left(\frac{xy}{2s}\right)\left(-\frac{4}{x}+\frac{x}{2s}-\frac{2y}{x}\left(-\frac{y}{2s}+\frac{\mu+1/2}{y}\right)\right)\\
&+\sqrt{\frac{xy}{2s}}I_{\mu}\left(\frac{xy}{2s}\right)\left(-\frac{y}{2s}+\frac{\mu+1/2}{y}\right)
\bigg), \;\;x,y,s\in\DO.
\end{align*}
Then, (\ref{P2}) leads to
\begin{align*}
\partial_y&\mathbb{K}_\mu(x,y,s)=\frac{e^{-\frac{(x-y)^2}{4s}}}{\sqrt{2\pi}}\frac{x^2}{(2s)^{5/2}}\left(1+O\left(\frac{s}{xy}\right)\right)\bigg(\frac{y^2}{2sx}+\frac{4y}{x^2}-\frac{y}{s}-\frac{y^3}{x^22s}-\frac{4}{x}+\frac{x}{2s}+\frac{y^2}{xs}-\frac{y}{2s}\\
&+(\mu+1/2)\left(\frac{y}{x^2}-\frac{2}{x}+\frac{1}{y}\right)\bigg)\\
&={e^{-\frac{(x-y)^2}{4s}}}O(1)\frac{x^2}{s^{5/2}}\bigg[
(\mu+1/2)\frac{(y-x)^2}{x^2 y}+\frac{(x-y)^3}{x^2
2s}+\frac{4(y-x)}{x^2}\bigg],\;\; s,x,y\in\DO,\; xy\ge s.
\end{align*}
Hence, by \eqref{eq2.6} we obtain
\begin{align*}
 |\partial_y\mathbb{K}_\mu(x,y,s)|&\le C{e^{-\frac{(x-y)^2}{4s}}} \bigg(\frac{|y-x|^2}{s^{5/2} y}+\frac{|x-y|^3}{s^{7/2}}+\frac{|y-x|}{s^{5/2}}\bigg)\\
 &\le C{e^{-\frac{(x-y)^2}{4s}}} \bigg(\frac{|y-x|^2}{s^{3} }+\frac{|x-y|^3}{s^{7/2}}+\frac{|y-x|}{s^{5/2}}\bigg)\\
 &\le\frac{C}{(\sqrt{s}+|x-y|)^4}, \;\;x,y,s\in\DO, \;\; xy\ge s.
\end{align*}
We conclude that
\begin{align}\label{eq2.8}
|\partial_y\mathbb{K}_\mu(x,y,s)|\le\frac{C}{(\sqrt{s}+|x-y|)^4},
\;\;&x,y,s\in\DO.
\end{align}

Now we estimate $\partial_s\mathbb{K}_\mu(x,y,s)$. Let $x,y\in
(0,\infty)$. We define the function
$\varphi_{x,y}(z)=\mathbb{K}(x,y,z),\,\,\,z\in
\mathbb{C},\,\,\,\text{Re}\,z>0 $. Thus, $\varphi_{x,y}$ is an
holomorphic function in $\{z\in \mathbb{C}:Re\,z>0\}$. Note that, if
$a>0$, $Arg\,\frac{a}{z}=-Arg(z)$ and
$Re\Bigg(\frac{a}{z}\Bigg)=\frac{a}{|z|^2}Re\,z$, $z\in \mathbb{C}$.
Note that $Re z\ge\frac{\sqrt{2}}{2}|z|$ provided that   $|Arg
z|\le\frac{\pi}{4}.$ Hence, $\left|e^{-\frac{(x-y)^2}{4z}}
\right|=e^{\frac{-Re z(x-y)^2}{4|z|^2}}\le
e^{-\frac{\sqrt{2}}{8}\frac{(x-y)^2}{|z|}},$ and
$\left|e^{-\frac{x^2+y^2}{4z}} \right|\le
e^{-\frac{\sqrt{2}}{8}\frac{x^2+y^2}{|z|}}$, when $|Arg
z|\le\frac{\pi}{4}.$

According (\ref{P1}) and (\ref{P2}) as in (\ref{X1}) and
(\ref{eq2.2}), we can obtain
\begin{align*}
|\varphi_{x,y}(z)|&\le C |z|^{-3/2}e^{-cRe\frac{|x-y|^2}{z}}\le C |z|^{-3/2}e^{-c\frac{|x-y|^2Re\,z}{|z|^2}}\le C |z|^{-3/2}e^{-c\frac{|x-y|^2}{|z|}} ,\,\,\,z\in
\mathbb{C},\,\,\,|Arg(z)|<\frac{\pi}{4}.
\end{align*}
By using Cauchy integral formula we get
$$
|\partial_s\mathbb{K}_\mu(x,y,s)|\le
C\frac{1}{s^{5/2}}e^{-c\frac{(x-y)^2}{s}},\,\,\,s\in (0,\infty).
$$
Here $C$ and $c$ do not depend on $x,y\in (0,\infty)$. Then, we
obtain
\begin{equation}\label{eq2.9}
|\partial_s\mathbb{K}_\mu(x,y,s)| \le\frac{C}{(\sqrt{s}+|x-y|)^5},
\;\;x,y,s\in\DO.
\end{equation}

Hence, for every $x,y\in\DO$, $x\neq y$, $\displaystyle\lim_{s\to
0^+}\partial_s  \mathbb{K}_\mu(x,y,s)=0$.

Then, we deduce that $\mathbb{K}_\mu$ is in $C^1(
\DO\times\DO\times\R\setminus\{(x,x,0): \; x\in\DO\})$, and $K_\mu$
is in $C^1((\R\times\DO\times\R\times\DO)\setminus D)$.

    According to \eqref{eq2.7}, \eqref{eq2.8} and \eqref{eq2.9}, we have that
    $$
    |\partial_t{K}_\mu(x,t;y,\tau)|=|\partial_\tau K_\mu(x,t;y,\tau)|\le\frac{C}{(\sqrt{|t-\tau|}+|x-y|)^5},
    $$
    $$
        |\partial_x{K}_\mu(x,t;y,\tau)|\le\frac{C}{(\sqrt{|t-\tau|}+|x-y|)^4}, \;\;\text{and} \;\;|\partial_y{K}_\mu(x,t;y,\tau)|\le\frac{C}{(\sqrt{|t-\tau|}+|x-y|)^4},
    $$
    for every $(x,t;y,\tau)\in[(\DO\times\R)\times(\DO\times \R)]\setminus D$, where $D=\{(x,t;x,t):\; x\in\DO\;\; \text{and}\;\; t\in\R\}.$

    Let now $x,y,y_0\in\DO$ and $t,\tau,\tau_0\in\R$ such that $d((x,t);(y_0,\tau_0))=\sqrt{|t-\tau_0|}+|x-y_0|>2(\sqrt{|\tau-\tau_0|}+|y-y_0|)=2d((y,\tau); (y_0,\tau_0)).$
    Then, $s(x,t;y_0,\tau_0)+(1-s)(x,t;y,\tau)\not\in D$, for every $s\in(0,1)$. Indeed, suppose that $s\in(0,1)$ and that $s(x,t;y_0,\tau_0)+(1-s)(x,t;y,\tau)\in D$. We have that $x=sy_0+(1-s)y$ and $t=s\tau_0+(1-s)\tau$. It follows that
    $$
\sqrt{|t-\tau_0|}+|x-y_0|=\sqrt{(1-s)|\tau_0-\tau|}+|1-s||y-y_0|\le
\sqrt{|\tau_0-\tau|}+|y-y_0|,
    $$
    and this is not possible.

    By using the mean value theorem, we can write

    \begin{align*}
    K_\mu(x,&t;y,\tau)-K_\mu(x,t;y_0,\tau_0)\\
    &=\partial_zK_\mu(x,t;z,\theta)_{\big| \substack{z=sy+(1-s)y_0\\\theta=s\tau+(1-s)\tau_0}}(y-y_0)+\partial_\theta K_\mu(x,t;z,\theta)_{\big| \substack{z=sy+(1-s)y_0\\\theta=s\tau+(1-s)\tau_0}}(\tau-\tau_0),
    \end{align*}
    for certain $s\in(0,1)$.
    Then, \begin{align*}
    |K_\mu(x,t;y,\tau)-K_\mu(x,t;y_0,\tau_0)|&\le C\bigg(\frac{|y-y_0|}{\left(\sqrt{|t-s\tau-(1-s)\tau_0|}+|x-sy-(1-s)y_0|\right)^4}\\
    & +\frac{|\tau-\tau_0|}{\left(\sqrt{|t-s\tau-(1-s)\tau_0|}+|x-sy-(1-s)y_0|\right)^5} \bigg).
    \end{align*}
    Note that $|\tau-\tau_0|\le
    \sqrt{|\tau-\tau_0|}(\sqrt{|\tau-\tau_0|}+|y-y_0|)<\frac{1}{2}\sqrt{|\tau-\tau_0|}(\sqrt{|t-\tau_0|}+|x-y_0|),$
    and
    \begin{align*}
    d((x,t),(z,\theta))&=|x-z|+\sqrt{|t-\theta|}\geq |x-y_0|-|z-y_0|+\sqrt{|t-\tau_0|}-\sqrt{|\theta-\tau_0|}\\
    &=|x-y_0|+\sqrt{|t-\tau_0|}-(s|y-y_0|+\sqrt{s}\sqrt{|\tau-\tau_0|})\\
    &\geq |x-y_0|+\sqrt{|t-\tau_0|}-(|y-y_0|+\sqrt{|\tau-\tau_0|})\\
    &>\frac{1}{2}( |x-y_0|+\sqrt{|t-\tau_0|})=\frac1{2}d((x,t);(y_0,\tau_0)).
    \end{align*}
 We get $|   K_\mu(x,t;y,\tau)-K_\mu(x,t;y_0,\tau_0)|\le C\frac{|y-y_0|+\sqrt{|\tau-\tau_0|}}{(\sqrt{|t-\tau_0|}+|x-x_0|)^4}.$

 By proceeding in a similar way we also obtain
    $   |   K_\mu(y,\tau;x,t)-K_\mu(y_0,\tau_0;x,t)|\le C\frac{|y-y_0|+\sqrt{|\tau-\tau_0|}}{(\sqrt{|t-\tau_0|}+|x-x_0|)^4}.$
    We have just proved that $K_\mu$ is a standard Calder\'on-Zygmund kernel with respect to the homogeneous type space $(\R\times (0,\infty),m,d)$.

\vspace{5mm}

Now we prove that $\widetilde{K_\mu}$ is a standard
Calder\'on-Zygmund kernel with respect to the homogeneous type space
$(\R\times (0,\infty),m,d)$.

We define
$$
\mathcal{K}_\mu(x,y,s)=\frac{\partial}{\partial
s}(W_s^\mu(x,y))\chi_{(0,\infty)}(s), \,\,\,s\in
\R\,\,\,and\,\,\,x,y\in (0,\infty).
$$

By (\ref{P3}) we get
\begin{align*}
\mathcal{K}_\mu(x,y,s)&=(xy)^{\mu+1/2}e^{-\frac{x^2+y^2}{4s}}\bigg(-\bigg(\frac{xy}{2s}\bigg)^{-\mu}I_{\mu+1}\bigg(\frac{xy}{2s}\bigg)\frac{2xy}{(2s)^{\mu+3}}\\
&+\bigg(\frac{xy}{2s}\bigg)^{-\mu}I_{\mu}\bigg(\frac{xy}{2s}\bigg)\frac{2(\mu+1)}{(2s)^{\mu+2}}+\bigg(\frac{xy}{2s}\bigg)^{-\mu}I_{\mu}\bigg(\frac{xy}{2s}\bigg)\frac{x^2+y^2}{(2s)^{\mu+3}}\bigg),\,\,\,x,y,s\in
(0,\infty).
\end{align*}

According to (\ref{P1}), we can write
        \begin{align*}
        \mathcal{K}_\mu(x,y,s)&\le C   \left( \frac{(xy)^2}{s^{\mu+4}}+\frac{1}{s^{\mu+2}}+\frac{x^2+y^2}{s^{\mu+3}}\right)(xy)^{\mu+1/2}e^{-\frac{x^2+y^2}{4s}}\\
        &\le C\left( \frac{xy}{s}\right)^{\mu+1/2}\left( \frac{x^2+y^2}{(s+x^2+y^2)^{5/2}}+\frac{1}{(s+x^2+y^2)^{3/2}}\right)\\
        &\le \frac{C}{(\sqrt{s}+|x-y|)^3}, \;\; s,x,y\in\DO \;\; \text{and} \;\; \frac{xy}{s}\le 1.
        \end{align*}
        Note that $\mu\ge -1/2$.

        Also, by (\ref{P2}), we have that
        \begin{align}\label{eq2.11}
        \mathcal{K}_\mu(x,y,s)&=\bigg[ -\frac{2xy}{(2s)^{5/2}} \left(\frac{xy}{2s} \right)^{1/2}I_{\mu+1}\left(\frac{xy}{2s}\right)-\frac{2( \mu+1)}{(2s)^{3/2}}\left(\frac{xy}{2s} \right)^{1/2}I_{\mu}\left(\frac{xy}{2s}\right)\nonumber\\
        &+\left(\frac{xy}{2s} \right)^{1/2}I_{\mu}\left(\frac{xy}{2s}\right) \frac{x^2+y^2}{(2s)^{5/2}}\bigg]e^{-\frac{x^2+y^2}{4s}}\nonumber\\
        &=\frac{e^{-\frac{(x-y)^2}{4s}}}{\sqrt{2\pi}}\left[\frac{(x-y)^2}{(2s)^{5/2}}-\frac{2(\mu+1)}{2s^{3/2}}+ O\left(\frac{1}{s^{3/2}}\right) \right], \;\; s,x,y\in\DO, \;\; \text{and}\;\; xy\ge s.
        \end{align}
        Then, $$
            |\mathcal{K}_\mu(x,y,s)|\le \frac{C}{({s}+|x-y|^2)^{3/2}}\le \frac{C}{(\sqrt{s}+|x-y|)^3}, \;\;s,x,y\in\DO \;\; \text{and} \;\; {xy}\ge s.
        $$
        We conclude that $$  |\mathcal{K}_\mu(x,y,s)|\le\frac{C}{(\sqrt{s}+|x-y|)^3}, \;\;s,x,y\in\DO.$$

        By proceeding in a similar way, (\ref{P1}) and (\ref{P2}) lead to
        \begin{equation}\label{F1}
        |\mathcal{K}_\mu(x,y,z)|\le C\frac{e^{-c\frac{(x-y)^2}{|z|}}}{|z|^{3/2}},\;\; x,y\in\DO, \; |Arg z|\le \frac{\pi}{4}.
    \end{equation}
    By using Cauchy integral formula we deduce that
    $$
    \left|\frac{\partial}{\partial s}\mathcal{K}_\mu(x,y,s)\right|\le C\frac{e^{-c\frac{(x-y)^2}{s}}}{s^{5/2}},\;\; s,x,y\in\DO.
    $$
    Then, we obtain
    $$
      \left|\frac{\partial}{\partial s}\mathcal{K}_\mu(x,y,s)\right|\le\frac{C}{(\sqrt{s}+|x-y|)^5}, \;\;s,x,y\in\DO.
      $$

On the other hand, by using (\ref{P3}) (see (\ref{T1})), we have
that
    \begin{align*}
    &\frac{\partial}{\partial x}W_s^\mu(x,y)=  \frac{\partial}{\partial x}\left(\left( \frac{xy}{2s}\right)^{-\mu}I_\mu\left( \frac{xy}{2s}\right)\left( \frac{xy}{2s}\right)^{\mu+1/2}\frac1{\sqrt{2s}}e^{-\frac{x^2+y^2}{4s}} \right)\\
    &=\bigg[\frac{y}{2s}\left( \frac{xy}{2s}\right)^{-\mu} I_{\mu+1}\left( \frac{xy}{2s}\right)\left( \frac{xy}{2s}\right)^{\mu+1/2}\frac{1}{\sqrt{2s}}+\left( \frac{xy}{2s}\right)^{-\mu} I_{\mu}\left( \frac{xy}{2s}\right)\frac{1}{(2s)^{\mu+1}}y^{\mu+1/2}(\mu+1/2)x^{\mu-1/2}\\&-\frac{x}{2s}\left( \frac{xy}{2s}\right)^{-\mu} I_{\mu}\left( \frac{xy}{2s}\right) \left( \frac{xy}{2s}\right)^{\mu+1/2}\frac{1}{\sqrt{2s}}\bigg]e^{-\frac{x^2+y^2}{4s}}, \;\; x,y,s\in\DO.
    \end{align*}
    Then, (\ref{P1}) implies that
    $$
    \left| \frac{\partial}{\partial x}W_s^\mu(x,y) \right|\le C \left(\frac{y}{s^{3/2}}+\frac{x}{s^{3/2}} \right) e^{-\frac{x^2+y^2}{4s}}\le \frac{C}{s} e^{-c\frac{(x-y)^2}{s}}, \;\; s,x,y\in\DO, \; xy\le s,
    $$
    provided that $\mu>1/2$ or $\mu=-1/2$. Also, (\ref{P2}) leads to
    \begin{align*}
    &\frac{\partial}{\partial x}W_s^\mu(x,y)=\left(\frac{y}{(2s)^{3/2}}+ \frac{1}{x\sqrt{2s}}-\frac{x}{(2s)^{3/2}}\right) \left(1+ O\left(\frac{s}{xy}\right)\right)  e^{-\frac{(x-y)^2}{4s}}\\
    &=\frac{x-y}{(2s)^{3/2}}e^{-\frac{(x-y)^2}{4s}}+ \left( \frac{1}{x\sqrt{s}}+\frac{1}{y\sqrt{s}}\right)e^{-\frac{(x-y)^2}{4s}}O(1), \;\; s,x,y\in \DO, \; xy\ge s.
    \end{align*}
    We obtain,
    $$
    \left|\frac{\partial}{\partial x}W_s^\mu(x,y)\right|\le C\left(\frac{1}{s}+\frac{1}{\sqrt{xys}}\right)e^{-\frac{(x-y)^2}{4s}}\le \frac{C}{s}e^{-c\frac{(x-y)^2}{s}}, \;\; xy\ge s, \;\; \text{and} \;\; x/2\le y\le 2x.
    $$
    and, by taking $z=\max\{x,y\}$,
        \begin{align*}
        \left|\frac{\partial}{\partial x}W_s^\mu(x,y)\right|&\le C\frac{y+x}{s^{3/2}}e^{-\frac{(x-y)^2}{4s}}\le C\frac{z}{s^{3/2}}e^{-\frac{cz^2}{s}}\le\frac{C}{s}e^{-c\frac{z^2}{s}}\\
        &\le \frac{C}{s}e^{-c\frac{(x-y)^2}{s}}, \;\; xy\ge s, \;\; \text{and} \;\; 0<y<x/2 \;\; \text{or}\;\; 2x<y.
        \end{align*}
    We conclude that
    $$
    \left|\frac{\partial}{\partial x}W_s^\mu(x,y)\right|\le C \frac{e^{-c\frac{(x-y)^2}{s}}}{s},\;\; x,y,s\in\DO.
    $$
    The same arguments, by using again (\ref{P1}), (\ref{P2}) and (\ref{P3}), allows us to obtain
    $$
    \left|\frac{\partial}{\partial x}W_s^\mu(x,y)\right|\le C \frac{e^{-c\frac{(x-y)^2}{|z|}}}{|z|},\;\; x,y\in\DO, \; |Arg z|\le \frac{\pi}{4},
    $$
    and Cauchy integral formula leads to
        \begin{equation}\label{F2}
    \left|\frac{\partial}{\partial s}\frac{\partial}{\partial x}W_s^\mu(x,y)\right|\le C \frac{e^{-c\frac{(x-y)^2}{s}}}{s^2},\;\; x,y,s\in\DO.
    \end{equation}
    Symmetries imply that
        $$
        \left|\frac{\partial}{\partial s}\frac{\partial}{\partial y}W_s^\mu(x,y)\right|\le C \frac{e^{-c\frac{(x-y)^2}{s}}}{s^2},\;\; x,y,s\in\DO.
        $$
    We get
    $$
     \left|\frac{\partial}{\partial x}\mathcal{K}_\mu(x,y,s)\right|+    \left|\frac{\partial}{\partial y}\mathcal{K}_\mu(x,y,s)\right|\le \frac{C}{(|x-y|+\sqrt{s})^4}, \;\; x,y,s\in\DO.
    $$
    Putting together the above estimates and proceeding as in the $K_\mu$-case we can prove that $\widetilde{K_\mu}$ is a standard Calder\'on-Zygmund kernel with respect to the homogeneous type space $(\R\times\DO,m,d)$.

    Thus, the proof of this proposition is finished.

\end{proof}

Now the statements in Theorem \ref{Iteo2}, (ii), follows from Calder\'on-Zygmund theorem.

\section{Proof of Theorem \ref{Iteo2}, (3), (4) and (5)}

In order to prove the parts (3), (4) and (5) in Theorem \ref{Iteo2}
we use a procedure which is different from the one employed in
Section 3 to prove Theorem \ref{Iteo2}, (2). As it was mentioned in
the introduction, Ping, Stinga and Torrea \cite{PST} investigated
$L^p$-boundedness properties of the Riesz transformations associated
with the parabolic equation (\ref{I1.1}). They studied, when a one
dimensional spatial variable is considered, the following two
operators
\begin{equation}\label{4.1}
R(f)(t,x)=\lim_{\varepsilon\to
0}\int_{\Omega_\varepsilon}\partial^2_{yy}W_s(y)f(t-s,x-y)dsdy
\end{equation}
and
\begin{equation}\label{4.2}
\widetilde{R}(f)(t,x)=\lim_{\varepsilon\to
0}\int_{\Omega_\varepsilon}\partial_{s}W_s(y)f(t-s,x-y)dsdy,
\end{equation}
where, for every $\varepsilon>0$, $\Omega_\varepsilon=\{(s,y)\in
(0,\infty)\times \mathbb{R},\sqrt{s}+y>\varepsilon\}$. Our
procedure consists, roughly speaking, in studying the
$L^p$-boundedness properties of the difference operators $R_\mu-R$
and $\widetilde{R}_\mu-\widetilde{R}$. Then, $L^p$-boundedness properties of
$R_\mu$ and $\widetilde{R}_\mu$ are deduced from the corresponding ones
of $R$ and $\widetilde{R}$, respectively, established in \cite[Theorem
2.3, (B)]{PST}.

Calder\'on-Zygmund theorem employed in the proof of Theorem
\ref{Iteo2}, (ii), in the previous section allows us to consider
weighted $L^p$-spaces but the parameter $\mu$ is restricted to
$\mu=-1/2$ or $\mu>1/2$. This comparative approach applies to the
full range of values of $\mu>-1$.

\subsection{Riesz transformation $\widetilde{R_\mu}$}
 We consider the operator
$$
\widetilde{R_\mu}(f)(t,x)=\displaystyle\lim_{\epsilon\to
0^+}\int_{\Omega_\epsilon(x)}\mathcal{K}_\mu(s,x,y)f(t-s,y)ds dy,
\;\; f\in C^\infty_c(\R\times\DO),
$$
where we name $\mathcal{K}_\mu(s,x,y)=\partial_sW_s^\mu(x,y),$
$s,x,y\in\DO$.

We shall also fix our attention in the operator
$$
\widetilde{R}(f)(t,x)=\lim_{\epsilon\to
0^+}\int_{\Omega_\epsilon(x)}\mathcal{K}(s,x,y)f(t-s,y)ds dy, \;\;
f\in C^\infty_c(\R^2),
$$
where $\mathcal{K}(s,x,y)=\partial_s W_s(x-y)$, $s\in\DO$ and
$x,y\in\R$.

Ping, Stinga and Torrea in \cite{PST} studied $L^p$-boundedness
properties for the operator $\widetilde{R}$. Our objective is to
prove $L^p$-boundedness properties for the operator
$\widetilde{R_\mu}$ by using the corresponding ones for
$\widetilde{R}$.


According to (\ref{T2}) we have that

%
\begin{align}\label{eq3.1}
&\mathcal{K}_\mu(s,x,y)-\mathcal{K}(s,x,y)
=e^{-\frac{(x-y)^2}{16s}}O\left(\frac{1}{\sqrt{s}xy} \right), \;\;
s,x,y\in\DO\;\; \text{and} \;\; s\le xy.
\end{align}


By (\ref{T3}) we obtain
\begin{align}\label{eq3.2}
&|\mathcal{K}_\mu(s,x,y)-\mathcal{K}(s,x,y) |\le|\mathcal{K}_\mu(s,x,y)|+|\mathcal{K}(s,x,y)|\nonumber\\
&\le C\left(\frac{(xy)^{\mu+1/2}}{s^{\mu+2}}
+\frac{1}{s^{3/2}}\right)e^{-\frac{x^2+y^2}{8s}}, \;\; s,x,y\in\DO,
\;\; xy\le s.
\end{align}
From \eqref{eq3.1} and \eqref{eq3.2} we deduce that
\begin{align*}
\int_0^\infty&\int_{x/2}^{3x/2}|\mathcal{K}_\mu(s,x,y)-\mathcal{K}(s,x,y) |dyds\nonumber\\
&\le C\bigg(\int_{x/2}^{3x/2}\int_0^{xy}\frac{e^{-c\frac{(x-y)^2}{s}}}{\sqrt{s}xy}dsdy+\int_{x/2}^{3x/2}\int_{xy}^\infty e^{-c\frac{x^2+y^2}{s}}\left(\frac{1}{s^{3/2}}+\frac{(xy)^{\mu+1/2}}{s^{\mu+2}}\right)dsdy \\
&\le
C\bigg(\int_{x/2}^{3x/2}\frac{\sqrt{xy}}{xy}dy+\int_{x/2}^{3x/2}\left(\frac{1}{(xy)^{1/2}}+\frac{(xy)^{\mu+1/2}}{(xy)^{\mu+1}}\right)dy
\bigg)\le C, \;\; x\in\DO.
\end{align*}

By proceeding as in \eqref{T3} we get
$$
|\mathcal{K}_\mu(s,x,y)|\le
C\frac{(xy)^{\mu+1/2}}{s^{\mu+2}}e^{-c\frac{x^2+y^2}{s}}, \;\;
s,x,y\in\DO, \; s\ge xy.
$$
Also, from (\ref{F1}), it follows that
$$
|\mathcal{K}_\mu(s,x,y)|\le
C\frac{e^{-c\frac{(x-y)^2}{s}}}{s^{3/2}}, \;\; s,x,y\in\DO, \;\;
s\le xy.
$$
Then,
\begin{align}\label{eq3.3}
\INT\int_0^{x/2}|\mathcal{K}_\mu(s,x,y)|dyds&\le C\bigg(\int_0^{x/2}\int_0^{xy}\frac{e^{-c\frac{(x-y)^2}{s}}}{s^{3/2}}dsdy+\int_0^{x/2}\int_{xy}^\infty\frac{(xy)^{\mu+1/2}}{s^{\mu+2}}e^{-c\frac{x^2+y^2}{s}}dsdy \bigg)\nonumber\\
&\le C\left(\int_0^{x/2}\int_0^{xy}\frac{e^{-c\frac{x^2}{s}}}{s^{3/2}}dsdy+\int_0^{x/2} \frac{(xy)^{\mu+1/2}}{(xy)^{\mu+1}}dy\right) \nonumber\\
&\le C \left( \int_0^{x/2}\frac{(xy)^{1/2}}{x^2}dy+1\right)\le C,
\;\; x\in\DO,
\end{align}
and when $\mu>-1/2$,

\begin{align}\label{eq3.4}
\INT \int_{3x/2}^\infty|&\mathcal{K}_\mu(s,x,y)|dyds\le C\left( \int_{3x/2}^\infty\int_0^{xy}\frac{e^{-c\frac{(x-y)^2}{s}}}{s^{3/2}}dsdy+\int_{3x/2}^\infty\int_{xy}^\infty\frac{(xy)^{\mu+1/2}}{s^{\mu+2}}e^{-c\frac{x^2+y^2}{s}}dsdy \right) \nonumber\\
&\le C\left( \int_{3x/2}^\infty\int_0^{xy}\frac{e^{-c\frac{y^2}{s}}}{s^{3/2}}dsdy+\int_{3x/2}^\infty\int_{xy}^\infty\frac{(xy)^{\mu+1/2}}{s^{\mu+2}}e^{-c\frac{x^2+y^2}{s}}dsdy \right) \nonumber\\
&\le C\left( \int_{3x/2}^\infty\int_0^{xy}\frac{y^{-2}}{s^{1/2}}dsdy+\int_{3x/2}^\infty\int_{xy}^\infty\frac{(xy)^{\mu+1/2}}{s^{\mu+2}}\frac{s^{\frac{2\mu+3}{4}}}{(x^2+y^2)^{\frac{2\mu+3}{4}}}dsdy \right) \nonumber \\
&\le
C\left(x^{1/2}\int_{3x/2}^\infty\frac{dy}{y^{3/2}}+x^{\frac{2\mu+1}{4}}\int_{3x/2}^\infty\frac{dy}{y^{\frac{2\mu+5}{4}}}\right)\le
C, \;\; x\in\DO.
\end{align}
Observe that this estimate can not be improved for $-1<\mu\le -1/2$.

We now suppose that $g$ is a complex valued continuous function with
compact support in $\DO$. We define $g_0$ as the odd extension of
$g$ to $\R$. We can write

\begin{align*}
\int_{\R}&\partial_tW_t(x-y)g_0(y)dy=\INT\partial_tW_t(x-y)g(y)dy+\int_{-\infty}^0\partial_tW_t(x-y)g(-y)dy\\
&=\INT\partial_t(W_t(x-y)-W_t(x+y))g(y)dy,\;\; x\in\R,\,\,\,
\text{and}\;\; t\in\DO.
\end{align*}
This fact and the following estimate will be useful in the sequel.

Note that
\begin{align*}
\partial_t(W_t(x-y)-W_t(x+y))&=\frac{\partial}{\partial t}\bigg[ \frac{1}{\sqrt{4\pi t}}\left(e^{-\frac{(x-y)^2}{4t}}-e^{-\frac{(x+y)^2}{4t}}\right)\bigg]=\frac{\partial}{\partial t}\left[W_t(x-y)\left(1-e^{-\frac{xy}{t}}\right)\right]\\
&=\frac{\partial}{\partial
t}\left(W_t(x-y)\right)\left(1-e^{-\frac{xy}{t}}\right)+W_t(x-y)\frac{xy}{t^2}e^{-\frac{xy}{t}},
\;\; t,x,y\in\DO.
\end{align*}
Then,
\begin{align}\label{eq3.5}
|\partial_t(W_t(x-y)-W_t(x+y))|&\le C\left(\left| \partial_tW_t(x-y)\right|\frac{xy}{t}+|W_t(x-y)|\frac{xy}{t^2}\right)\nonumber\\
&\le C \frac{e^{-c\frac{(x-y)^2}{t}}}{t^{5/2}}xy\le
C\frac{e^{-c\frac{x^2+y^2}{t}}}{t^{5/2}}xy, \;\; t,x,y \in\DO, \;\;
xy\le t.
\end{align}

Let $f\in C_c^\infty(\R\times\DO)$. We define
$$
 f_0(t,x)= \left\{ \begin{array}{c}
    f(t,x),\;\;t\in\R, x\ge 0,\\
    \,\\
 f(t,-x), \;\;t\in\R, x< 0.
 \end{array} \right.
 $$

 Thus, $f_0\in C_c^\infty(\R^2)$. We can write
 \begin{align}\label{eq3.6}
 \INT\int_{\R}
 \partial_s&W_s(x-y)f_0(t-s,y)dyds\nonumber\\&=\INT\INT(\partial_sW_s(x-y)-\partial_sW_s(x+y))f(t-s,y)dyds,\,\,\,
(t,x)\not\in \text{supp}f_0
 \end{align}
 This fact suggests the following analysis. From \eqref{eq3.5} we deduce that
 \begin{align}\label{eq3.7}
 \INT&\int_0^{x/2}|\partial_sW_s(x-y)-\partial_sW_s(x+y)|dyds\nonumber\\
  &\le C\left(\int_0^{x/2}\int_{0}^{xy}\Bigg(|\partial_sW_s(x-y)|+|\partial_sW_s(x+y)|\Bigg)dsdy
  +\int_0^{x/2}\int_{xy}^\infty \frac{xye^{-c\frac{x^2+y^2}{s}} }{s^{5/2}}dsdy \right) \nonumber\\
  &\le C\left(\int_0^{x/2}\int_{0}^{xy} \frac{e^{-c\frac{(x-y)^2}{s}} }{s^{3/2}}dsdy +\int_0^{x/2}\int_{xy}^\infty \frac{xy}{s^{5/2}}dsdy \right) \nonumber\\
   &\le C\left(\int_0^{x/2}\int_{0}^{xy} \frac{e^{-c\frac{x^2}{s}} }{s^{3/2}}dsdy +\int_0^{x/2} \frac{1}{(xy)^{1/2}}dy \right)\le C, \;\; x\in\DO,
 \end{align}
 and
 \begin{align}\label{eq3.8}
  \INT\int_{3x/2}^\infty&|\partial_sW_s(x-y)-\partial_sW_s(x+y)|dyds\nonumber\\
  &\le C\left(\int_{3x/2}^\infty\int_{0}^{xy} \frac{e^{-c\frac{(x-y)^2}{s}} }{s^{3/2}}dsdy+\int_{3x/2}^\infty\int_{xy}^\infty \frac{xye^{-c\frac{x^2+y^2}{s}} }{s^{5/2}}dsdy \right) \nonumber\\
     &\le C\left(\int_{3x/2}^\infty\int_{0}^{xy} \frac{e^{-c\frac{y^2}{s}} }{s^{3/2}}dsdy+\int_{3x/2}^\infty\int_{xy}^\infty \frac{xys}{s^{5/2}(x^2+y^2)}dsdy \right)\nonumber\\
     & \le \left( \int_{3x/2}^\infty \int_0^{xy}\frac{1}{y^3}dsdy+\int_{3x/2}^\infty\frac{\sqrt{xy}}{y^2}dy\right)\le C, \;\; x\in\DO.
 \end{align}

  Finally, we have that
  \begin{align}  \label{eq3.9}
  \INT\int_{x/2}^{3x/2}|\partial_sW_s(x+y)|dyds&\le \int_{x/2}^{3x/2} \INT\frac{e^{-c\frac{(x+y)^2}{s}} }{s^{3/2}}dsdy\le C\int_{x/2}^{3x/2}\frac{dy}{x+y}\nonumber\\
  &\le C(\log(x+3x/2)-log(x+x/2))\le C, \;\;x\in\DO.
  \end{align}

 Fix $\mu>-1/2$. Let $f\in C_c^\infty(\R\times\DO)$. We define $f_0$ as above. According to \eqref{eq3.6} and \cite[Theorem 2.3, (B)]{PST}, we can write, for every $t\in \mathbb{R}$ and $x\in (0,\infty)$,
$$
\widetilde{R_\mu}(f)(t,x)-\widetilde{R}(f_0)(t,x)=\displaystyle\lim_{\epsilon\to
0^+}\Big(\int_{\Omega_\epsilon(x)}\mathcal{K}_\mu(\tau,x,y)f(t-\tau)d\tau
dy-\int_{W_\epsilon(x)}\mathcal{K}(\tau,x,y))f(t-\tau,y)d\tau
dy\Big),
$$
where $W_\varepsilon(x)=\{(\tau,y)\in (0,\infty)\times
\mathbb{R}:\,\sqrt{\tau}+|x-y|>\varepsilon\}$, for every $x\in
\mathbb{R}$ and $\varepsilon>0$.

Then,
 \begin{small}
 \begin{align*}
 &\widetilde{R_\mu}(f)(t,x)-\widetilde{R}(f_0)(t,x)
 =\lim_{\epsilon\to 0^+}\bigg[\INT\int_{x/2, \; (\tau,y)\in\Omega_\epsilon(x)}^{2x} (\mathcal{K}_\mu(\tau,x,y)-\mathcal{K}(\tau,x,y))f(t-\tau,y) dyd\tau\\
 &+ \INT \int_{0, \; (\tau,y)\in\Omega_\epsilon(x)}^{x/2}\mathcal{K}_\mu(\tau,x,y)f(t-\tau,y) dyd\tau+\INT \int_{3x/2, \; (\tau,y)\in\Omega_\epsilon(x)}^\infty\mathcal{K}_\mu(\tau,x,y)f(t-\tau,y) dyd\tau\\
 &-\INT\int_{0, \; (\tau,y)\in\Omega_\epsilon(x)}^{x/2}(\mathcal{K}(\tau,x,y)-\mathcal{K}(\tau,x,-y))f(t-\tau,y)dy d\tau+\INT\int_{x/2, \; (\tau,y)\in\Omega_\epsilon(x)}^{2x}\mathcal{K}(\tau,x,-y)f(t-\tau,y)dy d\tau\\
 &-\INT\int_{3x/2, \; (\tau,y)\in\Omega_\epsilon(x)}^{\infty}(\mathcal{K}(\tau,x,y)-\mathcal{K}(\tau,x,-y))f(t-\tau,y)dy d\tau \bigg] ,\;\; t\in \mathbb{R},\,\,\,x\in\DO.
 \end{align*}
 \end{small}
 Note that from \eqref{eq3.3}, \eqref{eq3.4}, \eqref{eq3.5}, \eqref{eq3.7}, \eqref{eq3.8} and \eqref{eq3.9} we deduce that the last six integrals are absolutely convergent. We get
 $$
\widetilde{R_\mu}(f)(t,x)-\widetilde{R}(f_0)(t,x)=\INT\INT
T_\mu(\tau,x,y)f(t-\tau,y)dyd\tau, \;\; t\in
\mathbb{R},\,\,\,x\in\DO,
 $$
where $T_\mu(t,x,y)=
\mathcal{K}_\mu(t,x,y)-\mathcal{K}(t,x,y)+\mathcal{K}(t,x,-y)$,
$t\in \mathbb{R}$ and $x,y\in\DO$.

Note that \begin{equation}\label{eq3.10} \sup_{x\in\DO}\INT\INT
|T_\mu(\tau,x,y)|d\tau dy<\infty.
\end{equation}

We consider the  operator $\widetilde{T_\mu}(f)(t,x)=\INT\INT
T_\mu(\tau,x,y)f(t-\tau,y)d\tau dy $, $t\in\R$, $x\in\DO$.

 According to \eqref{eq3.10} we obtain
  \begin{itemize}
    \item There exists  a constant $C>0$ such that, for every $f\in L^1(\R\times\DO)$,
    \begin{align*}
    \|\widetilde{T_\mu}(f)\|_{L^1(\R\times\DO)}&=\int_{\R}\INT\left| \INT\INT T_\mu(\tau,x,y)f(t-\tau,y)d\tau dy\right|dxdt\\
    &\le \int_{\R}\INT\int_{\R}\INT \left|T_\mu(t-\tau,x,y)\right||f(\tau,y)|\chi_{\DO}(t-\tau)dyd\tau dxdt\\
            &= \int_{\R}\INT|f(\tau,y)|\INT\INT \left|T_\mu(u,x,y)\right|dxdud\tau dy\le C\|f\|_{L^1(\R\times\DO)}.
    \end{align*}
    We have used that $T_\mu(u,x,y)=T_\mu(u,y,x)$, $u,x,y\in\DO$.
    \item There exists  a constant $C>0$ such that, for every $f\in L^\infty(\R\times\DO)$,
    $$
        \|\widetilde{T_\mu}(f)\|_{L^\infty(\R\times\DO)}\le C\|f\|_{L^\infty(\R\times\DO)}.
    $$
  \end{itemize}

 Marcinkiewicz interpolation theorem allows us to conclude that $\widetilde{T_\mu}$ is a bounded operator from $L^p(\R\times \DO)$ into itself, for every $1\le p\le\infty$.

 By \cite[Theorem 2.3, (B)]{PST} we deduce that, for every $1\le p<\infty$, the operator $\widetilde{R_\mu}$ can be extended to $L^p(\R\times\DO)$ as a bounded operator from  $L^p(\R\times \DO)$ into itself when $1< p<\infty$ and from $L^1(\R\times \DO)$ into $L^{1,\infty}(\R\times \DO)$.

 As it was established in \cite[Theorem 2.3, (B)]{PST}, the maximal opertor
 $$
 \tilde{R}_*(f)(t,x)=\sup_{\epsilon>0}\left| \int_{\Omega_{\epsilon}(x)}\tilde{R}(\tau,x,y)f(t-\tau,y)d\tau dy\right|, \;\; t,x\in\R,
 $$
 is bounded from $L^p(\R^2)$ into itself, for every $1<p<\infty,$ and from $L^1(\R^2)$ into $L^{1,\infty}(\R^2)$.
 Then, the above results imply that the maximal operator
 $$
  \widetilde{R_{\mu,*}}(f)(t,x)=\sup_{\epsilon>0}\left| \int_{\Omega_{\epsilon}(x)}\widetilde{R_\mu}(\tau,x,y)f(t-\tau,y)d\tau dy\right|, \;\; t\in\R,\,\,\,\text{and} \; \; x\in\DO,
 $$
 is bounded from $L^p(\R\times \DO)$ into itself, for every $1<p<\infty$, and from $L^1(\R\times \DO)$ into $L^{1,\infty}(\R\times \DO)$.

 Since the principal value $\displaystyle\lim_{\epsilon\to 0^+} \int_{\Omega_{\epsilon}(x)}\widetilde{R_\mu}(\tau,x,y)f(t-\tau,y)d\tau dy$ exists, for every $f\in C_c^\infty(\R\times\DO)$ and $(t,x)\in\R\times\DO$, and as $C_c^\infty(\R\times\DO)$ is a dense subspace of  $L^p(\R\times \DO)$, $1\le p<\infty$, we have that, for every $f\in  L^p(\R\times \DO)$, $1\le p<\infty$, there exists the principal value $\displaystyle\lim_{\epsilon\to 0^+} \int_{\Omega_{\epsilon}(x)}\widetilde{R_\mu}(\tau,x,y)f(t-\tau,y)d\tau dy$ for almost all $(t,x)\in\R\times\DO$. Also, the operator $\widehat{R_\mu}$ defined on  $L^p(\R\times \DO)$, $1\le p<\infty$, as follows
 $$
 \widehat{R_\mu}(f)(t,x)=\displaystyle\lim_{\epsilon\to 0^+} \int_{\Omega_{\epsilon}(x)}\widetilde{R_\mu}(\tau,x,y)f(t-\tau,y)d\tau dy,\;\; \text{a.e.}\; (t,x)\in\R\times\DO,
 $$
 is bounded from  $L^p(\R\times \DO)$ into itself, for every $1<p<\infty$, and from  $L^1(\R\times \DO)$ into $L^{1,\infty}(\R\times \DO)$.

Our objective is to study the $L^p$-boundedness properties for $\widetilde{R_\mu}$ when $-1<\mu\le -1/2$. We have that
    \begin{align*}
    \widetilde{R_\mu}(f)&=\mathcal{F}^{-1}h_\mu\left( \frac{-i\rho}{z^2+i\rho}\mathcal{F}h_\mu(f)\right)=h_\mu h_{\mu+2}\mathcal{F}^{-1}h_{\mu+2}\left( \frac{-i\rho}{z^2+i\rho}\mathcal{F}h_{\mu+2}h_{\mu+2}h_\mu(f)\right)\\
    &=S_\mu \widetilde{R_{\mu+2}}S_\mu^*f, \;\; f\in L^2(\R\times\DO),
    \end{align*}
    where $S_\mu=h_\mu h_{\mu+2}$ and $S_\mu^*=h_{\mu+2}h_\mu$.

    The composition operators $h_\mu h_\nu$ are named transplantation operators associated with Hankel transforms.

 According to \cite[Theorem 2.1]{NS}, if $\mu>-1$, $1<p<\infty$ and $v$ is  a nonnegative measurable function such that
 $$
 \sup_{r>0}\left(\int_0^rv(x)^px^{p(\mu+1/2)} dx\right)^{1/p}\left( \int_r^\infty v(x)^{-p'}x^{-p'(\mu+3/2)}dx\right)^{1/p'}<\infty
 $$
 then, the operator $S_\mu$ can be extended to $L^p(v)$ as a bounded operator from $L^p(v)$ into
 itself. Here, as usual, $p'$ denotes the conjugated of $p$, that is,
 $p'=\frac{p}{p-1}$.

 Since
 \begin{align*}
 \left(\int_0^rx^{p(\mu+1/2)} dx\right)^{1/p}\left( \int_r^\infty x^{-p'(\mu+3/2)}dx\right)^{1/p'}&=\frac{r^{\frac{1}{p}(p(\mu+1/2)+1)+\frac{1}{p'}(1-p'(\mu+3/2))}}{(p(\mu+1/2)+1)(-1+p'(\mu+3/2))}\\
 &=\frac{1}{(p(\mu+1/2)+1)(-1+p'(\mu+3/2))}, \;\; r>0,
 \end{align*}
 provided that $p( \mu+1/2)+1>0$ and $1-p'(\mu+3/2)<0$, $S_\mu$ defines a bounded operator from $L^p(\DO)$ into itself when $1<p<\infty$ and $\mu>-1/2-1/p$ and $\mu>-1$. Then $S_\mu^*$ is bounded from $L^p(\DO)$ into itself when $1<p<\infty$, $\mu>-1$ and $\mu>-1/2-(1-1/p)=-3/2+1/p$. Hence, since $\widetilde{R_{\mu+2}}$ is bounded from  $L^p(\DO)$ into itself when $1<p<\infty$ and $\mu>-1$, $\widetilde{R_\mu}$ defines a bounded operator from $L^p(\R\times\DO)$ into itself provided that $-1<\mu\le 1/2$ and $-1/2-\mu<1/p<3/2+\mu$.

 Our results about $\widetilde{R_\mu}$ can be summarized as follows:
 \begin{itemize}
    \item $\widetilde{R_\mu}$ is bounded from $L^p(\R\times\DO)$ into itself when
    \begin{enumerate}
        \item $\mu>-1/2$ and $1<p<\infty$.
        \item $-1<\mu\le -1/2$ and $-1/2-\mu<1/p<3/2+\mu$.
            \end{enumerate}
    \item $\widetilde{R_\mu}$ is bounded from $L^1(\R\times\DO)$ into $L^{1,\infty}(\R\times\DO)$, when $\mu>-1/2$.
     \end{itemize}

\begin{center}

\begin{tikzpicture}

    \draw[gray,->](-2.2,0)--(3.5,0);
        \node [right] at (3.5,0) {$\mu$};

    \draw[gray,->](0,-0.5)--(0,3);
    \node [left] at (0,3) {$\frac{1}{p}$};

    \draw[dashed,black](-2.4,0.6)--(-0.8,2.2);

    \draw[dashed,black](-2,1)--(-0.1,-0.9);
    \draw[dashed,black](-1,2)--(3,2);
    \node[above] at (1.7,2) {weak};
    \node at (1.7,1) {strong};
\node[below right] at (-0.2,-0.8) {$\mu=-\frac{1}{2}-\frac{1}{p}$};
\node[below right] at (-4.65,1) {$\mu=-\frac{3}{2}+\frac{1}{p}$};
\draw [thick] (-.1,1) node[right]{$\frac{1}{2}$} -- (0.3,1); \draw
[thick] (-.1,2) node[right]{${1}$} -- (0.3,2); \draw [thick]
(-1,-0.1) node[below]{$\frac{-1}{2}$} -- (-1,0.1); \draw [thick]
(-2,-0.1) node[below]{${-1}$} -- (-2,0.1);
    \end{tikzpicture}

\end{center}

\subsection{Riesz transformation $R_\mu$}

We consider the operator $R_\mu$ defined by
$$
R_\mu(f)(t,x)=\lim_{\varepsilon\to
0}\int_{\Omega_\varepsilon(x)}\mathbb{K}_\mu(x,y,s)f(t-s,y)dsdy,
$$
for every $f\in C_c^\infty((\R\times (0,\infty))$.  We recall that
$$
\mathbb{K}_\mu
(x,y,s)=\delta_{\mu+1}\delta_{\mu}W_s^\mu(x,y),\,\,\,s,x,y\in
(0,\infty).
$$
We also write
$$
\mathbb{K}(x,y,s)=\partial^2_{xx}W_s(x-y),\,\,\,x,y\in
\R\,\,\,\text{and}\,\,\,s\in (0,\infty).
$$

According to \eqref{eq2.1} we have that
\begin{align*}
\mathbb{K}_\mu(x,y,s)&=\frac{x^{\mu+5/2}y^{\mu+1/2}}{(2s)^{\mu+3}}\bigg(\frac{y^3}{2s x}\bigg(\frac{xy}{2s}\bigg)^{-\mu-1}I_{\mu+2}\bigg(\frac{xy}{2s}\bigg)-\frac{2y}{x}\bigg(\frac{xy}{2s}\bigg)^{-\mu}I_{\mu+1}\bigg(\frac{xy}{2s}\bigg)\\
&+\bigg(\frac{xy}{2s}\bigg)^{-\mu}I_{\mu}\bigg(\frac{xy}{2s}\bigg)
\bigg)e^{-\frac{x^2+y^2}{4s}}, \;\; x,y,s\in\DO.
\end{align*}
 Also, we get
  \begin{align*}
 \mathbb{K}(x,y,s)=\frac{\partial^2}{\partial x^2}\left(\frac{ e^{-\frac{(x-y)^2}{4s}}}{\sqrt{4\pi s}}\right)=\frac{ e^{-\frac{(x-y)^2}{4s}}}{\sqrt{4\pi s}}\left( \frac{(x-y)^2}{(2s)^2}-\frac{1}{2s}\right), \;\; x,y\in\R, \; s\in\DO.
 \end{align*}
 We define $T_\mu(x,y,s)=\mathbb{K}_\mu(x,y)-\mathbb{K}(x,y,s)+\mathbb{K}(x,-y,s), \;\; s,x,y\in\DO.$

 Our first objective is to study when
 \begin{equation}\label{eq4.1}
 \displaystyle\sup_{x\in\DO}\INT\INT |T_\mu(x,y,s)|dyds<\infty,
 \end{equation}
 is true.

 We decompose the proof of \eqref{eq4.1} in several steps.   According to (\ref{P1}), we  can write
 \begin{align}\label{eq4.2}
 |\mathbb{K}&_\mu(x,y,s)|\le C\frac{x^{\mu+5/2}y^{\mu+1/2}}{s^{\mu+2}}\left(\frac{y^4}{s^3}+\frac{y^2}{s^2}+\frac{1}{s} \right)e^{-\frac{x^2+y^2}{4s}}\nonumber\\
 &\le  C\left(\frac{xy}{s}\right)^{\mu+1/2}\frac{x^2}{s^{5/2}}\left(\frac{y^4}{s^{2}}+\frac{y^2}{s}+1 \right)e^{-\frac{x^2+y^2}{4s}}, \;\; s,x,y\in\DO\,\,\,\text{and}\,\,\,xy\le s.
 \end{align}
  By (\ref{P2})  we get
  \begin{align}\label{eq4.3}
  |\mathbb{K}_\mu(x,y,s)|&\le C\frac{x^{\mu+5/2}y^{\mu+1/2}}{s^{\mu+2}}\left(\frac{y^2}{sx^2}+\frac{y}{sx}+\frac{1}{s} \right)\left(\frac{xy}{s}\right)^{-\mu-1/2}e^{-\frac{(x-y)^2}{4s}}\nonumber\\
  &\le C\frac{x^2}{s^{5/2}}e^{-\frac{(x-y)^2}{4s}}\left(\frac{y^2}{x^2}+\frac{y}{x}+1 \right), \;\; s,x,y\in\DO, \; xy\ge s.
  \end{align}

 From \eqref{eq4.3} it follows that, for every $s,x,y\in\DO$ such that $xy\ge s$,
 \begin{equation}\label{eq4.4}
  |\mathbb{K}_\mu(x,y,s)|\le C\frac{x^2}{s^{5/2}}\left(\frac{y^2}{x^2}+\frac{y}{x}+1 \right)e^{-c\frac{\max\{x,y\}^2}{s}}, \;\; 0<y<x/2\,\,\,\text{or}\,\,\, \;y>2x.
 \end{equation}
 By using \eqref{eq4.2} and \eqref{eq4.4} we obtain
 \begin{small}

 \begin{align}\label{eq4.5}
 \int_0^{x/2}&\INT|\mathbb{K}_\mu(x,y,s)|dsdy\le C\bigg(\int_0^{x/2}\int_0^{xy}\frac{x^2}{s^{5/2}}\left( \frac{y^2}{x^2}+\frac{y}{x}+1 \right)e^{-c\frac{x^2}{s}}dsdy\nonumber\\
 &+\int_0^{x/2}\int_{xy}^\infty\frac{(xy)^{\mu+1/2}}{s^{\mu+3}}x^2\left( \frac{y^4}{s^2}+\frac{y^2}{s}+1 \right)e^{-\frac{x^2+y^2}{4s}}dsdy\bigg)\nonumber\\
 &\le C x^2\bigg( \int_0^{x/2}\int_0^{xy}\frac{e^{-c\frac{x^2}{s}}}{s^{5/2}}dsdy\nonumber\\
 &+\int_0^{x/2}(xy)^{\mu+1/2}\int_{xy}^\infty\left(\frac{y^4}{(x^2+y^2)^2}+\frac{y^2}{x^2+y^2}+1 \right)\frac{1}{s^{\mu+3}}\left( \frac{s}{x^2+y^2}\right)^{\frac{2\mu+5}{4}}dsdy\bigg) \nonumber\\
 &\le Cx^2\left( \int_0^{x/2}\int_{x/y}^\infty\frac{e^{-cu}\sqrt{u}}{x^3}dudy+x^{\mu+1/2}\int_0^{x/2}\frac{y^{\mu+1/2}}{(x^2+y^2)^{\frac{2\mu+5}{4}}(xy)^{\mu/2+3/4}}dy\right) \nonumber\\
 &\le C\left(\frac{1}{x} \int_0^{x/2}dy\INT e^{-cu}\sqrt{u} du+x^{2+\mu+1/2-\frac{2\mu+5}{2}-\frac{\mu}{2}-3/4}\int_0^{x/2}y^{\mu+1/2-\mu/2-3/4}dy \right)\nonumber \\
 &\le C\left( 1+x^{-\mu/2-3/4}\int_0^{x/2}y^{\mu/2-1/4}dy\right) \le C, \;\; x\in\DO,
 \end{align}

  \end{small}
 and

 \begin{align}\label{eq4.6}
 \int_{3x/2}^\infty&\INT |\mathbb{K}_\mu(x,y,s)|dsdy\le C\bigg(\int_{3x/2}^\infty\int_0^{xy}\frac{x^2}{s^{5/2}}\left( \frac{y^2}{x^2}+\frac{y}{x}+1\right)e^{-c\frac{y^2}{s}}dsdy\nonumber \\&+\int_{3x/2}^\infty\int_{xy}^\infty\frac{(xy)^{\mu+1/2}}{s^{\mu+3}}x^2\left(\frac{y^4}{s^2}+\frac{y^2}{s}+1 \right) e^{-\frac{x^2+y^2}{4s}}dsdy\bigg)\nonumber \\
 &\le C\left(\int_{3x/2}^\infty y^2\int_0^{xy}\frac{e^{-c\frac{y^2}{s}}}{s^{5/2}}dsdy+\int_{3x/2}^\infty(xy)^{\mu+1/2}x^2 \int_{xy}^\infty\left(\frac{y^4}{s^2}+\frac{y^2}{s}+1 \right)\frac{e^{-\frac{y^2}{4s}}}{s^{\mu+3}}dsdy\right)\nonumber \\
 &\le C\left( \int_{3x/2}^\infty y^2\int_{y/x}^\infty\frac{e^{-cu}\sqrt{u}}{y^3}dudy+\int_{3x/2}^\infty(xy)^{\mu+1/2}\frac{x^2}{(xy)^{\mu+2}}dy \right)\nonumber \\
 &\le C\left( x\int_{3x/2}^\infty\frac{y}{2x}\frac{e^{-c\frac{y}{2x}}}{y^2}dy+1\right)\le C, \;\; x\in\DO.
 \end{align}

We have that
\begin{align*}
\mathbb{K}&(x,y,s)-\mathbb{K}(x,-y,s)=\frac{e^{-\frac{(x-y)^2}{4s}}}{\sqrt{4\pi}(2s)^{3/2}}\left(1-e^{-\frac{(x+y)^2-(x-y)^2}{4s}} \right)\left(\frac{(x-y)^2}{2s}-1 \right)-\frac{4xye^{-\frac{(x+y)^2}{4s}}}{\sqrt{4\pi}(2s)^{5/2}}\\
&=\frac{1}{\sqrt{4\pi}(2s)^{3/2}}e^{-\frac{(x-y)^2}{4s}}\left(1-e^{-\frac{xy}{s}}\right)\left(\frac{(x-y)^2}{2s}-1
\right)-\frac{4xye^{-\frac{(x+y)^2}{4s}}}{\sqrt{4\pi}(2s)^{5/2}},
\;\; s,x,y\in\DO.
\end{align*}

 Then,
 \begin{equation}\label{eq4.7}
 \big|\mathbb{K}(x,y,s)-\mathbb{K}(x,-y,s)\big|\le C  \frac{xye^{-c\frac{(x-y)^2}{s}}}{s^{5/2}}, \;\; s,x,y\in\DO, \; xy\le s,
 \end{equation}
 and
  \begin{equation}\label{eq4.8}
  \big|\mathbb{K}(x,y,s)-\mathbb{K}(x,-y,s)\big|\le C  \frac{e^{-c\frac{(x-y)^2}{s}}}{s^{3/2}}, \;\; s,x,y\in\DO.
  \end{equation}
 From \eqref{eq4.7} and \eqref{eq4.8} we deduce that

 \begin{align}\label{eq4.9}
 \int_0^{x/2}\INT \big|\mathbb{K}(x,y,s)-\mathbb{K}(x,-y,s)\big|dsdy&\le C \int_0^{x/2}\INT\frac{e^{-c\frac{(x-y)^2}{s}}}{s^{3/2}}dsdy\le C, \;\; x\in\DO;
 \end{align}
 and

 \begin{align}\label{eq4.10}
  \int_{3x/2}^\infty\INT \big|\mathbb{K}(x,y,s)-\mathbb{K}(x,-y,s)\big|dsdy&\le C\bigg( \int_{3x/2}^\infty\int_0^{xy}\frac{e^{-c\frac{(x-y)^2}{s}}}{s^{3/2}}dsdy
  +\int_{3x/2}^\infty\int_{xy}^\infty\frac{xye^{-c\frac{(x-y)^2}{s}}}{s^{5/2}}dsdy \bigg)\nonumber\\
    &\le C\left( \int_{3x/2}^\infty\frac{1}{y}\int_{y/x}^\infty \frac{e^{-cu}}{\sqrt{u}}du+ \int_{3x/2}^\infty xy\int_0^{y/x}\frac{e^{-cu}\sqrt{u}}{y^3}dudy\right)\nonumber\\
  &\le C\left( \int_{3x/2}^\infty \frac{e^{-c\frac{y}{x}}}{y}\INT \frac{e^{-cu}}{\sqrt{u}}du+x \int_{3x/2}^\infty \frac{1}{y^2}\INT e^{-cu}\sqrt{u}dudy\right)\nonumber\\
  &\le C x \int_{3x/2}^\infty\frac{dy}{y^2}\le C, \;\; x\in\DO.
 \end{align}
 On the other hand, we have that
$$\big|\mathbb{K}(x,-y,s)\big|\le C \frac{e^{-c\frac{(x+y)^2}{s}}}{s^{3/2}}, \;\; s,x,y\in\DO.
$$
We can write
\begin{align}\label{eq4.11}
\int_{x/2}^{3x/2}\INT|\mathbb{K}(x,-y,s)| ds dy&\le C
\int_{x/2}^{3x/2}\INT\frac{e^{-c\frac{(x+y)^2}{s}}}{s^{3/2}}dsdy\nonumber\\
&\le C
\int_{x/2}^{3x/2}\frac{1}{x+y}\INT\frac{e^{-u}}{\sqrt{u}}dudy\le C
\;\; x\in\DO.
\end{align}
 Finally we are going to estimate $\displaystyle\int_{x/2}^{3x/2}\INT\big|\mathbb{K}_\mu(x,y,s)-\mathbb{K}(x,y,s)\big|dsdy$.

 According to \eqref{eq4.2} we obtain
 \begin{align}\label{eq4.12}
 &\int_{x/2}^{3x/2}\int_{xy}^\infty\big|\mathbb{K}_\mu(x,y,s)-\mathbb{K}(x,y,s)\big|dsdy\le\int_{x/2}^{3x/2}\int_{xy}^\infty\bigg|\mathbb{K}_\mu(x,y,s)\bigg|dsdy+ \int_{x/2}^{3x/2}\int_{xy}^\infty\bigg|\mathbb{K}(x,y,s)\bigg|dsdy \nonumber\\
 &\le C \bigg( \int_{x/2}^{3x/2}\int_{xy}^\infty x^2(xy)^{\mu+1/2}\left( \frac{y^4}{(x^2+y^2)^2}+\frac{y^2}{x^2+y^2}+1\right) \frac{e^{-c\frac{x^2+y^2}{s}}}{s^{\mu+3}}dsdy+ \int_{x/2}^{3x/2}\int_{xy}^\infty\frac{e^{-c\frac{x^2+y^2}{s}}}{s^{3/2}}dsdy \bigg)\nonumber\\
  &\le C \left( x^2\int_{x/2}^{3x/2}\frac{(xy)^{\mu+1/2}}{(xy)^{\mu+2}}dy+\int_{x/2}^{3x/2}\frac{dy}{\sqrt{xy}}\right) \le C, \;\; x\in\DO.
 \end{align}
 By using \eqref{eq2.2} we get
 $$
 \;\bigg|\mathbb{K}_\mu(x,y,s)-\mathbb{K}(x,y,s)\bigg|\le C\frac{e^{-\frac{(x-y)^2}{8s}}}{s^{3/2}}, \;\; 0<s\le xy\,\,\,\text{and}\,\,\,x/2<y<3x/2,
 $$
 but at this moment it is not sufficient. We need to improve the last estimates.

Since  $I_\nu(z)=\frac{e^z}{\sqrt{2\pi z}}\left(1+
\frac{1-4\nu^2}{8z}+O\left(\frac{1}{z^2}\right)\right)$, for $z>0$
and $\nu>-1$ (see (\ref{P2})), we can write

 \begin{align*}
 &\mathbb{K}_\mu(x,y,s)=\frac{x^2}{\sqrt{2\pi}(2s)^{3/2}}\bigg( \frac{y^3}{(2s)^2x}\frac{2s}{xy}\left( 1+\frac{1-4(\mu+2)^2}{4}\frac{s}{xy}+ O\left( \left(\frac{s}{xy}\right)^2\right)\right)\\
 &-\frac{y}{sx}\left( 1+\frac{1-4(\mu+1)^2}{4}\frac{s}{xy}+ O\left( \left(\frac{s}{xy}\right)^2\right)\right)+\frac{1}{2s}\left( 1+\frac{1-4\mu^2}{4}\frac{s}{xy}+ O\left( \left(\frac{s}{xy}\right)^2\right)\right)\bigg)e^{-\frac{(x-y)^2}{4s}}\\
 &=\frac{x^2}{\sqrt{2\pi}(2s)^{3/2}}\bigg( \frac{y^2}{2sx^2}-\frac{y}{sx}+\frac{1}{2s}+\frac{1-4\mu^2}{4}\left(\frac{y}{2x^3}-\frac{1}{x^2}+\frac{1}{2xy} \right)-2\mu\left( \frac{y}{x^3}-\frac{1}{x^2}\right)-\frac{2y}{x^3}+\frac{1}{x^2}\\
 &+O\left(\frac{s}{x^4} \right)+O\left(\frac{s}{x^3y} \right)+O\left(\frac{s}{x^2y^2} \right)\bigg)e^{-\frac{(x-y)^2}{4s}}\\
 &=\frac{x^2}{\sqrt{2\pi}(2s)^{3/2}}e^{-\frac{(x-y)^2}{4s}}\bigg(  \frac{(y-x)^2}{2sx^2}+\frac{1-4\mu^2}{4}\frac{(y-x)^2}{2x^3y}-2\mu\frac{y-x}{x^3}-\frac{2y-x}{x^3}+O\left( \frac{s}{x^4}\right)\bigg),\\
 & \hspace{10cm}\;\; 0<x/2<y<3x/2 \;\;\text{and}\;\; 0<s\le xy.
 \end{align*}

 We get
 \begin{align*}
 \mathbb{K}_\mu(x,y,s)-\mathbb{K}(x,y,s)&=\frac{e^{-\frac{(x-y)^2}{4s}}}{\sqrt{2\pi}}\bigg[ \frac{(y-x)^2}{(2s)^{5/2}}+\frac{1-4\mu^2}{8}\frac{(y-x)^2}{(2s)^{3/2}xy}-2\mu\frac{y-x}{(2s)^{3/2}x}\\
 &-\frac{2y-x}{(2s)^{3/2}x}-\frac{(x-y)^2}{(2s)^{5/2}}+\frac{1}{(2s)^{3/2}}+ O\left(\frac{1}{x^2\sqrt{s}} \right)\bigg]\\
 &=\frac{e^{-\frac{(x-y)^2}{4s}}}{\sqrt{2\pi}}\bigg[\frac{1-4\mu^2}{8}\frac{(y-x)^2}{(2s)^{3/2}xy}-2(\mu+1)\frac{y-x}{(2s)^{3/2}x}\\
 &+ O\left(\frac{1}{x^2\sqrt{s}} \right) \bigg], \;\; x/2<y<3x/2, \;\;  \text{and}\;\; 0<s\le xy.
 \end{align*}

 By using the last equality we deduce that
 \begin{align} \label{eq4.13}
 &\int_{x/2}^{3x/2}\int_0^{xy}\bigg|\mathbb{K}_\mu(x,y,s)-\mathbb{K}(x,y,s)\bigg|dsdy\nonumber\\
 &\le C \bigg(\int_{x/2}^{3x/2}\int_0^{xy}\frac{e^{-c\frac{(x-y)^2}{s}}}{xy\sqrt{s}}dsdy+ \int_{x/2}^{3x/2}\int_0^{xy}\frac{e^{-c\frac{(x-y)^2}{s}}\sqrt{|x-y|}}{s^{5/4}x}dsdy+\int_{x/2}^{3x/2}\int_0^{xy}\frac{e^{-c\frac{(x-y)^2}{s}}}{x^2\sqrt{s}}dsdy\bigg)\nonumber\\
 &\le \left( \int_{x/2}^{3x/2}\frac{\sqrt{xy}}{xy}dy+\int_{x/2}^{3x/2}\frac{\sqrt{|x-y|}}{x}\INT\frac{e^{-u}}{u^{3/4}\sqrt{|x-y|}}dudy+\int_{x/2}^{3x/2}\frac{\sqrt{xy}}{x^2}dy\right)\nonumber\\
 &\le C\int_{x/2}^{3x/2}\frac{dy}{x}\le C, \;\; x\in\DO.
 \end{align}
 By combining \eqref{eq4.5}, \eqref{eq4.6}, \eqref{eq4.9}, \eqref{eq4.10}, \eqref{eq4.11},  \eqref{eq4.12} and  \eqref{eq4.13} we obtain
 $$
 \sup_{x\in\DO}\INT\INT |T_\mu(x,y,s)|dsdy<\infty.
 $$
 This property implies that the operator
 $$
 T_\mu (f)(t,x)=\INT\INT T_\mu(x,y,s)f(t-s,y)dsdy
 $$
 is bounded from $L^\infty(\R\times\DO)$ into itself, for every $\mu>-1$.

 The operator $R_\mu$ is bounded from $L^2(\R\times\DO)$ into itself and the operator $R$ is bounded from $L^2(\R^2)$ into itself, see \cite[Theorem 2.3, (B)]{PST}.
 From these facts we deduce that $T_\mu$ is bounded from $L^2(\R\times\DO)$ into itself. Hence, interpolation theorem implies that $T_\mu$ is bounded from $L^p(\R\times\DO)$
 into itself, for every $2\le p\le \infty$. By using again \cite[Theorem 2.3, (B)]{PST} we conclude that $R_\mu$ defines a bounded operator from $L^p(\R\times\DO)$ into itself for every $2\le p<\infty$ and  $\mu>-1$.

It is remarkable that the operator $R_\mu$ is not selfadjoint in
$L^2(\R\times\DO)$.

 To simplify we now consider the function
\begin{align*}
M_\mu(x,y,s)&=\frac{y^{\mu+5/2}x^{\mu+1/2}}{(2s)^{\mu+2}}\bigg(\frac{x^3}{4s^2 y}\bigg(\frac{xy}{2s}\bigg)^{-\mu-1}I_{\mu+2}\bigg(\frac{xy}{2s}\bigg)-\frac{x}{sy}\bigg(\frac{xy}{2s}\bigg)^{-\mu}I_{\mu+1}\bigg(\frac{xy}{2s}\bigg)\\
&+\frac{1}{2s}\bigg(\frac{xy}{2s}\bigg)^{-\mu}I_{\mu}\bigg(\frac{xy}{2s}\bigg)
\bigg)e^{-\frac{x^2+y^2}{4s}}-\frac{\partial^2}{\partial x^2}\left(
\frac{e^{-\frac{(x-y)^2}{4s}}}{\sqrt{4\pi s}}\right), \;\;
x,y,s\in\DO.
\end{align*}
 We are going to see that $\displaystyle\sup_{x\in\DO}\INT\INT |M_\mu(x,y,s)|dyds<\infty.$

 We define
 \begin{align*}
 K_\mu(x,y,s)&=\frac{y^{\mu+5/2}x^{\mu+1/2}}{(2s)^{\mu+2}}\bigg(\frac{x^3}{4s^2 y}\bigg(\frac{xy}{2s}\bigg)^{-\mu-1}I_{\mu+2}\bigg(\frac{xy}{2s}\bigg)-\frac{x}{sy}\bigg(\frac{xy}{2s}\bigg)^{-\mu}I_{\mu+1}\bigg(\frac{xy}{2s}\bigg)\\
 &+\frac{1}{2s}\bigg(\frac{xy}{2s}\bigg)^{-\mu}I_{\mu}\bigg(\frac{xy}{2s}\bigg) \bigg)e^{-\frac{x^2+y^2}{4s}}, \;\; x,y,s\in\DO.
 \end{align*}

 From \eqref{eq4.2}, \eqref{eq4.3} and \eqref{eq4.4} we deduce that
 \begin{equation}\label{eq4.14}
| K_\mu(x,y,s)|\le C
\bigg(\frac{xy}{s}\bigg)^{\mu+1/2}\frac{y^2}{s^{5/2}}\left(\frac{x^4}{s^2}+\frac{x^2}{s}+1
\right)e^{-\frac{x^2+y^2}{4s}}, \;\; s,x,y\in\DO,\;\; \text{and} \;
xy\le s;
 \end{equation}

  \begin{equation}\label{eq4.15}
  | K_\mu(x,y,s)|\le C \frac{y^2}{s^{5/2}}e^{-\frac{(x-y)^2}{4s}}\left(\frac{x^2}{y^2}+\frac{x}{y}+1 \right), \;\; s,x,y\in\DO,\;\; \text{and} \; xy\ge s;
  \end{equation}
 and, for every $x,y,s\in (0,\infty)$ and $xy\ge s$,

   \begin{equation}\label{eq4.16}
   | K_\mu(x,y,s)|\le C \frac{y^2}{s^{5/2}}e^{-c\frac{\max\{x,y\}^2}{s}}\left(\frac{x^2}{y^2}+\frac{x}{y}+1 \right), \;\;  0<y<x/2,\;\; \text{or} \;\; y>3x/2.
   \end{equation}

 By \eqref{eq4.14} and \eqref{eq4.16} it follows that
 \begin{align}\label{eq4.17}
 &\int_0^{x/2}\INT | K_\mu(x,y,s)|dsdy\le C \bigg( \int_0^{x/2} \int_0^{xy}\frac{y^2}{s^{5/2}}\left(\frac{x^2}{y^2}+\frac{x}{y}+1 \right)e^{-c\frac{x^2}{s}}dsdy\nonumber \\
 &+  \int_0^{x/2}\int_{xy}^\infty \frac{(xy)^{\mu+1/2}}{s^{\mu+3}}y^2\left(\frac{x^4}{s^2}+\frac{x^2}{s}+1 \right)e^{-\frac{x^2+y^2}{4s}}dsdy\bigg)\nonumber \\
 &\,\nonumber\\
 &\le C \bigg(  \int_0^{x/2}(x^2+xy+y^2) \int_0^{xy}\frac{e^{-c\frac{x^2}{s}}}{s^{5/2}}dsdy \nonumber\\
 &+\int_0^{x/2} \left( \frac{x^4}{(x^2+y^2)^2}+ \frac{x^2}{x^2+y^2}+1\right)\frac{(xy)^{\mu+1/2} y^2}{(x^2+y^2)^{\mu+2}}dy\INT e^{-u}u^{\mu+1}du\bigg)\nonumber \\
 &\,\nonumber\\
 &\le C \bigg(\int_0^{x/2}\frac{x^2+xy+y^2}{x^3}\INT e^{-u}\sqrt{u}du+x^{\mu+1/2}\int_0^{x/2}\frac{y^{\mu+5/2}}{(x+y)^{2\mu+4}}dy\INT e^{-u}u^{\mu+1}du \bigg)\nonumber \\
 &\hspace{4mm}\le C \left( \frac{1}{x}\int_0^{x/2} dy+\frac{1}{x^{\mu+7/2}}\int_0^{x/2}y^{\mu+5/2}dy\right) \le C, \;\; x\in\DO;
 \end{align}
 and
 \begin{align}\label{eq4.18}
  &\int_{3x/2}^\infty\INT | K_\mu(x,y,s)|dsdy\le C \bigg( \int_{3x/2}^\infty \int_0^{xy}\frac{y^2}{s^{5/2}}\left(\frac{x^2}{y^2}+\frac{x}{y}+1 \right)e^{-c\frac{y^2}{s}}dsdy\nonumber \\
  &+  \int_{3x/2}^\infty\int_{xy}^\infty \frac{(xy)^{\mu+1/2}}{s^{\mu+3}}y^2\left(\frac{x^4}{s^2}+\frac{x^2}{s}+1 \right)e^{-\frac{x^2+y^2}{4s}}dsdy\bigg)\nonumber \\
  &\le C \bigg(  \int_{3x/2}^\infty\frac{x^2+xy+y^2}{y^3} \int_{y/x}^\infty e^{-cu}\sqrt{u} dudy+ \int_{3x/2}^\infty \int_{xy}^\infty\frac{x^{\mu+1/2} y^{\mu+5/2}}{s^{\mu+3}} e^{-c\frac{x^2+y^2}{s}}dsdy\bigg) \nonumber \\
      &\le C \left(x\int_{3x/2}^\infty\frac{e^{-c\frac{y}{x}}}{y^2}dy +x^{\mu+1/2}\int_{3x/2}^\infty\frac{dy}{y^{\mu+3/2}} \right) \le C, \;\; x\in\DO,
 \end{align}
 provided that $\mu>-1/2$.

 Also, for $ x\in\DO$, we have
 \begin{align}\label{eq4.19}
 \int_{x/2}^{3x/2}\int_{xy}^\infty | K_\mu(x,y,s)|dsdy&\le  C  \int_{x/2}^{3x/2}\int_{xy}^\infty y^2(xy)^{\mu+1/2}\frac{e^{-c\frac{x^2+y^2}{s}}}{s^{\mu+3}}dsdy\nonumber\\
 &\le C  \int_{x/2}^{3x/2}y^2\frac{(xy)^{\mu+1/2}}{(xy)^{\mu+2}}dy\le C,
 \end{align}
and, as in \eqref{eq4.13},
  \begin{align}\label{eq4.20}
   &\int_{x/2}^{3x/2}\int_0^{xy}
   \bigg|K_\mu(x,y,s)-\mathbb{K}(x,y,s)
   \bigg|dsdy \le C \bigg(\int_{x/2}^{3x/2}\int_0^{xy} \frac{e^{-\frac{(x-y)^2}{8s}}}{\sqrt{s}xy}dsdy\nonumber \\ &+\int_{x/2}^{3x/2}\int_0^{xy} \frac{e^{-\frac{(x-y)^2}{8s}}\sqrt{|x-y|}}{s^{5/4} y}dsdy+ \int_{x/2}^{3x/2}\int_0^{xy} \frac{e^{-\frac{(x-y)^2}{8s}}}{y^2\sqrt{s} }dsdy\bigg) \le C.
  \end{align}

 From  \eqref{eq4.17}, \eqref{eq4.18},  \eqref{eq4.19} and  \eqref{eq4.20} and by taking into account the other above estimates, we deduce that
$$
 \sup_{x\in\DO}\INT\INT |M_\mu(x,y,s)|dsdy<\infty.
 $$
Then, the operator $T_\mu$ is bounded from $L^1(\R\times
(0,\infty))$ into itself, provided that $\mu>-1/2$.

By invoking interpolation theorem we infer that $T_\mu$ is bounded
from $L^p(\R\times (0,\infty)$ into itself, for every $1\le p\le
\infty$ when $\mu>-1/2$. By using again \cite[Theorem 2.3, (B)]{PST}
we conclude that $R_\mu$ defines,  for every $1< p<\infty$, a
bounded operator from $L^p(\R\times\DO)$ into itself and from
$L^1(\R\times (0,\infty))$ into  $L^{1,\infty}(\R\times (0,\infty))$
provided that $\mu>-1/2$.

By following the same argument as in the previous section, the use
of the maximal operator associated to the singular integral $R_\mu$
and \cite[Theorem 2.3, (B)]{PST} allow us to conclude that, for
every $f\in L^p(\R\times (0,\infty))$, $1\le p<\infty$, the limit
$$
\lim_{\varepsilon\to
0^+}\int_{\Omega_\varepsilon(x)}\mathbb{K}(x,y,s)f(t-s,y)dsdy,
$$
exists, for a.e. $(t,x)\in \R\times (0,\infty)$, when $\mu>-1/2$.
Moreover, the operator $\mathbb{R}_\mu$ defined by
$$
\mathbb{R}_\mu(f)(t,x)=\lim_{\varepsilon\to
0^+}\int_{\Omega_\varepsilon(x)}\mathbb{K}(x,y,s)f(t-s,y)dsdy,
$$
is bounded, for every $1< p<\infty$, from $L^p(\R\times\DO)$ into
itself and from $L^1(\R\times (0,\infty)$ into
$L^{1,\infty}(\R\times (0,\infty)$ provided that $\mu>-1/2$.

 Our next objective is to complete the study of the boundedness of the operator $R_\mu$ when $-1<\mu\le -1/2$.

  We have that, for every $f\in L^2(\R\times\DO)$, $R_\mu f=\mathcal{F}^{-1}h_{\mu+2}\left(\frac{z^2}{z^2+i\rho}\mathcal{F}h_\mu(f) \right)$.

 Then, the adjoint $R_\mu^*$ of $R_\mu$ is given by $$R_\mu^*f=\mathcal{F}^{-1}h_{\mu}\left(\frac{z^2}{z^2-i\rho}h_{\mu+2}\mathcal{F}f \right)=\left[ \mathcal{F}^{-1}h_{\mu}\left(\frac{z^2}{z^2+i\rho}h_{\mu+2}\mathcal{F}\tilde{f} \right)\right]^{\sim}, \;\; f\in L^2(\R\times\DO),$$
 where $\tilde{f}(t,x)=f(-t,x)$, $t\in\R$ and $x\in\DO$.

 We consider the operator
 $\mathcal{H}_\mu f=\mathcal{F}^{-1}h_{\mu}\left(\frac{z^2}{z^2+i\rho}h_{\mu+2}\mathcal{F}{f} \right)$, $f\in L^2(\R\times\DO)$.
 Since $h_\alpha^2=I$ in $L^2(\DO)$, for every $\alpha>-1$ we can write
 $$
 \mathcal{H}_\mu=h_\mu h_{\mu+2} \mathcal{F}^{-1}h_{\mu+2}\left(\frac{z^2}{z^2+i\rho}h_{\mu}\mathcal{F}h_\mu h_{\mu+2}  \right)=S_\mu R_\mu S_\mu,
 $$
 where $S_\mu=h_\mu h_{\mu+2}$.

According to \cite[Theorem 2.1]{NS}, $S_\mu$ defines a bounded operator from $L^p(\DO)$ into itself when $1<p<\infty$, $\mu>-1/2-1/p$ and $\mu>-1$. By taking into account that $R_\mu$ is bounded from $L^p(\R\times\DO)$ into itself when $2\le p<\infty$, we conclude that $R_\mu^*$ is bounded from $L^p(\R\times\DO)$ into itself provided that $\mu>-1/2-1/p$ and $2\le p<\infty$. Duality implies that $R_\mu$ defines a bounded operator from $L^p(\R\times\DO)$ into itself provided that $1<p\le 2$ and $\mu>1/p-3/2$.

 Our results about $R_\mu$ can be summarized as follows:
 \begin{itemize}
    \item $R_\mu$ is bounded from $L^p(\R\times\DO)$ into itself when
    \begin{enumerate}
        \item $\mu>-1/2$ and $1<p<\infty$.
        \item $-1<\mu\le -1/2$ and $p>\frac{1}{\mu+3/2}$.
            \end{enumerate}
    \item $R_\mu$ is bounded from $L^1(\R\times\DO)$ into $L^{1,\infty}(\R\times\DO)$, when $\mu>-1/2$.
     \end{itemize}

    
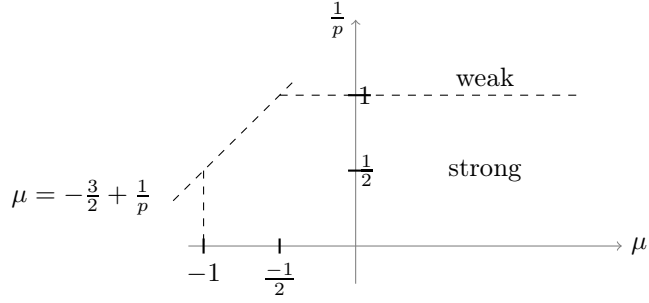
\begin{center}

\begin{tikzpicture}

    \draw[gray,->](-2.2,0)--(3.5,0);
        \node [right] at (3.5,0) {$\mu$};

    \draw[gray,->](0,-0.5)--(0,3);
    \node [left] at (0,3) {$\frac{1}{p}$};

    \draw[dashed,black](-2.4,0.6)--(-0.8,2.2);

    \draw[dashed,black](-2,1)--(-2,0);
    \draw[dashed,black](-1,2)--(3,2);
    \node[above] at (1.7,2) {weak};
    \node at (1.7,1) {strong};
\node[below right] at (-4.65,1) {$\mu=-\frac{3}{2}+\frac{1}{p}$};
\draw [thick] (-.1,1) node[right]{$\frac{1}{2}$} -- (0.2,1); \draw
[thick] (-.1,2) node[right]{${1}$} -- (0.2,2); \draw [thick]
(-1,-0.1) node[below]{$\frac{-1}{2}$} -- (-1,0.1); \draw [thick]
(-2,-0.1) node[below]{${-1}$} -- (-2,0.1);
    \end{tikzpicture}

\end{center}

\section{Proof of Theorem \ref{Iteo3}}

    We firstly consider the Riesz transform $R_\mu$ defined by
    \begin{align*}
    R_\mu (f)(t,x)&=\lim_{\epsilon\to 0^+}\int_{\Omega_\epsilon(x)}\mathbb{K}_\mu(s,x,y)f(t-s,y)dsdy\\
    &+f(t,x)\frac{1}{\sqrt{\pi}}\int_1^\infty e^{-s^2/4}ds,\,\,\,a.e. \;\;(t,x)\in\R\times \DO,
    \end{align*}
    for every $f\in L^p(\R\times (0,\infty))$ $1\le p<\infty$, where $\mathbb{K}_\mu(s,x,y)=\delta_{\mu+1}\delta_\mu W_s^\mu(x,y)$, $s,x,y\in (0,\infty)$. We have that, for each $f\in L^p(\R\times (0,\infty))$, $1\le p<\infty$,
    $$
    R_\mu (f)(t,x)=\INT\INT \mathbb{K}_\mu(s,x,y)f(t-s,y)dsdy, \;\;(t,x)\not\in\text{supp}(f).
    $$
    \begin{proposition}\label{Prop1}
        Let $\mu>-1/2$ and $1<p<\infty$. The operator $R_\mu$ can be extended to $L^q(\R, L^p(\DO))$  as a bounded operator from $L^q(\R, L^p(\DO))$ into itself, for $1<q<\infty$ and
        from  $L^1(\R, L^p(\DO))$  into $L^{1,\infty}(\R, L^p(\DO))$.
    \end{proposition}
    \begin{proof}
     We define, for every $s\in\DO$,
        $$
        \mathbb{T}_s(F)(x)= \INT \mathbb{K}_\mu(s,x,y)F(y)dy, \;\; x\in\DO,
        $$
        for every $F\in L^p(\DO)$.

        Note that, according to  \eqref{X1}, \eqref{AMI1} and \eqref{AMI2},
        \begin{equation}\label{AMI3}
        | \mathbb{K}_\mu(s,x,y)|\le C\frac{e^{-c\frac{(x-y)^2}{s}}}{s^{3/2}},\,\,\, s,x,y\in\DO.
        \end{equation}
        Then,
        $$
        \INT| \mathbb{K}_\mu(s,x,y)| dy \le C\int_{\R}\frac{e^{-c\frac{(x-y)^2}{s}}}{s^{3/2}}dy\le \frac{C}{s}, \;\; s,x\in\DO.
        $$
        and also

        $$
        \INT| \mathbb{K}_\mu(s,x,y)| dx\le \frac{C}{s}, \;\; s,y\in\DO.
        $$
        It follows that, for every $s\in\DO$, $\mathbb{T}_s$ defines a bounded operator from $L^p(\DO)$ into itself and $\|\mathbb{T}_s\|_{p\to p}\le \frac{C}{s}$.

        For every $t,s\in\R$, we define
        $$
        \mathbb{H}(t,s)(F)=\left\{ \begin{array}{c} \mathbb{T}_{t-s}(F), \;\; t>s \\
        0, \;\;\;\;\;\; \;\;\;\;\;t\le s,
        \end{array} \right.
        $$
        for $F\in L^p(\DO)$.

        Thus, for every $t,s\in\R$, $\mathbb{H}(t,s)\in\mathcal{L}(L^p(\DO))$ and $\|\mathbb{H}(t,s)\|_{p\to p}\le \frac{C}{|t-s|}$.

        We consider, for every $g\in C_c^\infty(\mathbb{R}\times (0,\infty))\subset L^p(\mathbb{R},L^p((0,\infty)))$,
        $$
        \beta_\mu(g)(t)=\int_{\R}\mathbb{H}(t,s)(g(s))ds, \;\; t\not\in\text{supp}(g).
        $$

        Note that if $g\in C_c^\infty(\mathbb{R}\times (0,\infty))$ and $ t\not\in\text{supp}(g)$, then
        \begin{align}\label{eq5.1}
        \int_{\R}\|\mathbb{H}(t,s)(g(s))\|_{L^p( \DO)}ds&\le C\int_{\text{supp}(g)}\frac{\|g(s)\|_{L^p(\DO)}}{t-s}ds\nonumber \\&\le C\left( \int_{\text{supp}(g)}\frac{ds}{|t-s|^{q'}}\right)^{1/q'}{\|g\|_{L^q(\R,L^p(\DO))}}<\infty,
        \end{align}
        and the integral $\displaystyle\int_{\R}\mathbb{H}(t,s)(g(s))ds$ converges in the $L^q(\R)$-Bochner sense.

        We established in Theorem \ref{Iteo2} that $R_\mu$ is bounded from $L^p(\R\times\DO)=L^p(\R, L^p(\DO))$ into itself.

        Let $g\in C^\infty_c(\R\times\DO)$ and $\ell\in (L^p(\DO))'=L^{p'}(\DO)$. According to the well-known properties of Bochner integrals we have that

        \begin{align*}
        \langle\ell, \int_{\R}\mathbb{H}(t,s)(g(s))ds\rangle&=\int_{\R} \langle\ell,\mathbb{H}(t,s)(g(s))  \rangle ds=\int_{-\infty}^t\INT\ell(x) \INT\mathbb{K}_\mu(t-s,x,y)g(s,y)dydxds\\
        &=\INT \ell(x)\int_{-\infty}^t\INT\mathbb{K}_\mu(t-s,x,y)g(s,y)dydsdx, \;\; t\not\in \text{supp}g.
        \end{align*}
        The interchange of the order of integration is justified by \eqref{eq5.1}. Note that
        \begin{align*}
        &\left\| \int_{-\infty}^t\INT|\mathbb{K}_\mu(t-s,x,y)||g(s,y)|dyds \right\|_{L^p(\DO)}\\
        &\le C \left(\int_{\text{supp}g(\cdot)}\frac{ds}{|t-s|^{q'}}\right)^{\frac{1}{q'}}\|g\|_{L^q(\R,L^p(\DO))}, \;\; t\not\in\text{supp}(g).
        \end{align*}
        We conclude that, for every $t\not\in\text{supp}(g),$
        $$
        \left( \int_{\R}\mathbb{H}(t,s)(g(s))ds\right)(x)=\INT\INT\mathbb{K}_\mu(s,x,y)g(t-s,y)dyds, \;\; \text{a.e}\;\; x\in\DO.
        $$
        Suppose that $t_1,t_2,s\in\DO$, being $|t_1-s|>2|t_1-t_2|$. Then, $s>\max\{t_1,t_2\}$ or $s<\min\{t_1,t_2\}$. Let $g\in L^p(\DO)$. Assume that $s<\min\{t_1,t_2\}$. We have that
        $$
        [\mathbb{H}(t_1,s)(g)-\mathbb{H}(t_2,s)(g)](x)=\INT(\mathbb{K}_\mu(t_1-s,x,y)-\mathbb{K}_\mu(t_2-s,x,y))g(y)dy, \;\; x\in\DO.
        $$
        Note that $\mathbb{H}(t_1,s)(g)-\mathbb{H}(t_2,s)(g)=0$ when $s>\max\{t_1,t_2\}$. According to \eqref{eq2.9} we get
        \begin{align*}
        |\mathbb{K}_\mu(t_1-s,x,y)-\mathbb{K}_\mu(t_2-s,x,y) |\le \left|\frac{\partial}{\partial u}\mathbb{K}_\mu(u,x,y) \right||t_1-t_2|\le C\frac{|t_1-t_2|}{(\sqrt{u}+|x-y|)^5}, \;\; x, y\in\DO,
        \end{align*}
        for some $u\in (\min\{t_1,t_2\}-s, \max\{t_1,t_2\}-s)$. Suppose that $t_1<t_2$. Then, $u>t_1-s$ and $t_2-s=t_2-t_1+t_1-s<3(t_1-s)/2<3u/2$. We obtain
        \begin{equation}\label{AMI6}
        |\mathbb{K}_\mu(t_1-s,x,y)-\mathbb{K}_\mu(t_2-s,x,y) |\le C\frac{|t_1-t_2|}{(\sqrt{t_i-s}+|x-y|)^5}, \;\; i=1,2, \; x, y\in\DO.
        \end{equation}
        It follows that
        \begin{align*}
        \INT |\mathbb{K}_\mu(t_1-s,x,y)&-\mathbb{K}_\mu(t_2-s,x,y) |dy\\
        &\le C |t_1-t_2|\INT\frac{dy}{(\sqrt{t_1-s}+|x-y|)^5}\le C\frac{|t_1-t_2|}{|t_1-s|^2}, \;\; x\in\DO.
        \end{align*}
        Also,
        $$
        \INT |\mathbb{K}_\mu(t_1-s,y,x)-\mathbb{K}_\mu(t_2-s,y,x) |dy\le C\frac{|t_1-t_2|}{|t_1-s|^2}, \;\; x\in\DO.
        $$

        We conclude that $\mathbb{H}(t_1,s)-\mathbb{H}(t_2,s)\in \mathcal{L}( L^p(\DO))$ and
        $$
        \|\mathbb{H}(t_1,s)-\mathbb{H}(t_2,s)\|_{L^p(\DO)\to L^p(\DO)}\le C\frac{|t_1-t_2|}{|t_1-s|^2}.
        $$
        By using vector valued Calder\'on-Zygmund theory we deduce that the operator $R_\mu$ can be extended to $L^q(\R, L^p(\DO))$  as a bounded operator from
        \begin{itemize}
            \item $L^q(\R, L^p(\DO))$ into itself for every $1<q<\infty$,
            \item $L^1(\R, L^p(\DO))$  into $L^{1,\infty}(\R, L^p(\DO))$.
        \end{itemize}
    \end{proof}

     Let now $s\in\DO$. We consider again the operator
     $$
     \mathbb{T}_s(F)(x)= \INT \mathbb{K}_\mu(s,x,y)F(y)dy, \;\; x\in\DO, \; F\in L^p(\DO).
     $$
According to (\ref{AMI3}) we get
$$
|\mathbb{K}_\mu(s,x,y)|\le
C\frac{e^{-c\frac{(x-y)^2}{s}}}{s^{3/2}}\le \frac{C}{s|x-y|},
\;\;x\neq y, \;\; s,x,y\in\DO,
$$

Also by (\ref{eq2.7}) and (\ref{eq2.8}) we obtain
    \begin{align*}
    |\partial_x\mathbb{K}_\mu(x,y,s)|+&|\partial_y\mathbb{K}_\mu(x,y,s)|\le \frac{C}{(\sqrt{s}+|x-y|)^4}\\
    &\le \frac{C}{s|x-y|^2}, \;\; x\neq y, \;\; s,x,y\in\DO,
    \end{align*}
    provided that $\mu>1/2$ or $\mu=-1/2$.

    Since $\mathbb{T}_s$ is bounded from, for instance, $L^2(\DO)$ into itself and $\|\mathbb{T}_s\|_{2\to 2}\le \frac{C}{s}$, Calder\'on-Zygmund theory implies that,
    for every $1<p<\infty$ and $\omega\in A_p(\DO)$, the operator $\mathbb{T}_s$ can be extended to $L^p(\DO, \omega)$ as a bounded operator from  $L^p(\DO, \omega)$ into itself and $\|\mathbb{\mathbb{T}}_s\|_{L^p(\DO, \omega)\to L^p(\DO, \omega)}\le \frac{C}{s}, $ provided that $\mu>1/2$ or $\mu=-1/2$.

        In the sequel we assume that  $\mu>1/2$ or $\mu=-1/2$.

        Let $1<p<\infty$ and $\omega\in A_p(\DO)$. We define as above, for every $t,s\in\R$,
        $$
        \mathbb{H}(t,s)(F)=\left\{ \begin{array}{c} \mathbb{T}_{t-s}(F), \;\; t>s \\
        0, \;\;\;\;\;\; \;\;\;\;\;t\le s,
        \end{array} \right.
        $$
        for $F\in L^p(\DO)$. Also, for every $g\in L^p(\mathbb{R},L^p((0,\infty),\omega))$, we consider
         $$
        \beta_\mu(g)(t)=\int_{\R}\mathbb{H}(t,s)(g(s))ds, \;\; t\not\in\text{supp}(g).
        $$

        We have that, for every $1<p<\infty$ and $\omega\in A_p(\DO)$,
        $$
        \|\mathbb{H}(t,s)\|_{L^p(\DO,\omega)\to L^p(\DO,\omega) }\le \frac{C}{|t-s|}, \;\; t,s\in\DO, \; t\neq s.
        $$

        Then, we infer that, for every $g\in C_c^\infty(\mathbb{R}\times (0,\infty))$ and $t\notin supp\,g$,
        $$
        [\beta_\mu(g)(t)](x)=R_\mu(g)(t,x),\,\,\,a.e.\,\,\,x\in DO.
        $$

        According Theorem \ref{Iteo2}, (2), $R_\mu$ is bounded from
        $L^p(\mathbb{R}\times
        \DO,W)=L^p(\mathbb{R},L^p(\DO,\omega))$ into itself, where
        $W(t,x)=\omega(x)$, $(t,x)\in\mathbb{R}\times\DO$.

    Suppose that $t_1,t_2,s\in\DO$, being $|t_1-s|>2|t_1-t_2|$. We are going to see that,
    $$
    \|\mathbb{H}(t_1,s)-\mathbb{H}(t_2,s)\|_{L^p(\DO,\omega)\to L^p(\DO,\omega)}\le C\frac{|t_1-t_2|}{|t_1-s|^2}.
    $$
    We have proved that,
            $$
        \|\mathbb{H}(t_1,s)-\mathbb{H}(t_2,s)\|_{L^p(\DO)\to L^p(\DO)}\le C\frac{|t_1-t_2|}{|t_1-s|^2}.
        $$
    From (\ref{AMI6}) we deduce that

    \begin{align*}
    |\mathbb{K}_\mu(t_1-s,x,y)-\mathbb{K}_\mu(t_2-s,x,y)|&\le C\frac{|t_1-t_2|}{(\sqrt{t_1-s}+|x-y|)^5}\\
    &\le C \frac{|t_1-t_2|}{|t_1-s|^2}\frac{1}{|x-y|},\;\;x,y\in\DO,\,\,\, x\neq y.
    \end{align*}
    Our objective is to see that
        \begin{align*}
        \left|\frac{\partial}{\partial x}\left(\mathbb{K}_\mu(t_1-s,x,y)-\mathbb{K}_\mu(t_2-s,x,y)\right)\right|\le C \frac{|t_1-t_2|}{|t_1-s|^2}\frac{1}{|x-y|^2},\;\;x,y\in\DO,\,\,\, x\neq y,
        \end{align*}
    and
        \begin{align*}
        \left|\frac{\partial}{\partial y}\left(\mathbb{K}_\mu(t_1-s,x,y)-\mathbb{K}_\mu(t_2-s,x,y)\right)\right|\le C \frac{|t_1-t_2|}{|t_1-s|^2}\frac{1}{|x-y|^2},\;\;x,y\in\DO,\,\,\, x\neq y.
        \end{align*}
    Assume that $s<\min\{t_1,t_2\}$. We can write
    $$
    \left|\frac{\partial}{\partial x}\left(\mathbb{K}_\mu(t_1-s,x,y)-\mathbb{K}_\mu(t_2-s,x,y)\right)\right|\le C \left|\frac{\partial^2}{\partial u\partial x}\mathbb{K}_\mu(u,x,y) \right||t_1-t_2|,
    $$
    for some $u\in(\min\{t_1,t_2\}-s,\max\{t_1,t_2\}-s)$.

    We have that (see (\ref{X2}))
    \begin{align*}
    \partial_x\mathbb{K}_\mu(s,x,y)&=\frac{x^{\mu+5/2}y^{\mu+1/2}}{(2s)^{\mu+3}}e^{-\frac{x^2+y^2}{4s}}\bigg(\frac{xy^6}{(2s)^4}\left(\frac{xy}{2s} \right)^{-\mu-3}I_{\mu+3}\left(\frac{xy}{2s} \right)\\
    &+\left(\frac{xy}{2s} \right)^{-\mu-2}I_{\mu+2}\left(\frac{xy}{2s} \right)\frac{xy}{2s}\left(-\frac{3y^3}{(2s)^2}+\frac{\mu+5/2}{2s}\frac{y^3}{x^2}\right)  \\
    &+\left(\frac{xy}{2s} \right)^{-\mu-1}I_{\mu+1}\left(\frac{xy}{2s} \right)\frac{xy}{2s}\left(\frac{3y}{2s}-\frac{2(\mu+5/2)y}{x^2}\right)\\
    &+\left(\frac{xy}{2s} \right)^{-\mu}I_\mu\left(\frac{xy}{2s} \right)\left( -\frac{x}{2s}+\frac{\mu+5/2}{x}\right) \bigg),\;\; x,y,s\in\DO.
    \end{align*}

    Let $x,y\in\DO$. We define the function
        \begin{align*}
    F_{x,y}(z)&=\frac{x^{\mu+5/2}y^{\mu+1/2}}{(2z)^{\mu+3}}e^{-\frac{x^2+y^2}{4z}}\bigg(\frac{xy^6}{(2z)^4}\left(\frac{xy}{2z} \right)^{-\mu-3}I_{\mu+3}\left(\frac{xy}{2z} \right)\\
    &+\left(\frac{xy}{2z} \right)^{-\mu-2}I_{\mu+2}\left(\frac{xy}{2z} \right)\frac{xy}{2z}\left(-\frac{3y^3}{(2z)^2}+\frac{\mu+5/2}{2z}\frac{y^3}{x^2}\right)  \\
    &+\left(\frac{xy}{2z} \right)^{-\mu-1}I_{\mu+1}\left(\frac{xy}{2z} \right)\frac{xy}{2z}\left(\frac{3y}{2z}-\frac{2(\mu+5/2)y}{x^2}\right)\\
    &+\left(\frac{xy}{2z} \right)^{-\mu}I_\mu\left(\frac{xy}{2z} \right)\left( -\frac{x}{2z}+\frac{\mu+5/2}{x}\right) \bigg),\;\; z\in\C\setminus(-\infty,0].
    \end{align*}
    $F_{x,y}$ is holomorphic in $\C\setminus(-\infty,0]$.

According to (\ref{P1}) and (\ref{P2}), by proceeding as in the
proof of (\ref{eq2.7}), we obtain
    $$
    |F_{x,y}(z)|\le C\frac{e^{-c\frac{(x-y)^2}{|z|}}}{|z|^2}, \;\; |Arg z|\le\frac{\pi}{4}.
    $$
    By using Cauchy integral formula we obtain that
    $$
    \left| \frac{d}{dt}F_{x,y}(t)\right|\le \frac{C}{t^3}e^{-\frac{(x-y)^2}{32t}}\le \frac{C}{(\sqrt{t}+|x-y|)^6}, \;\; t>0.
    $$
    By proceeding as above we get
    \begin{align*}
    |{\partial_x}\left(\mathbb{K}_\mu(t_1-s,x,y)-\mathbb{K}_\mu(t_2-s,x,y)\right)|\le C \frac{|t_1-t_2|}{|t_1-s|^2}\frac{1}{|x-y|^2},
    \end{align*}
    and
        \begin{align*}
        |{\partial_y}\left(\mathbb{K}_\mu(t_1-s,x,y)-\mathbb{K}_\mu(t_2-s,x,y)\right)|\le C \frac{|t_1-t_2|}{|t_1-s|^2}\frac{1}{|x-y|^2},
        \end{align*}
        for $x,y\in\DO$, $x\neq y$.

        Calder\'on-Zygmund theory leads to
        $$
        \|\mathbb{H}(t_1,s)-\mathbb{H}(t_2,s)\|_{L^p(\DO,\omega)\to L^p(\DO,\omega)}\le C\frac{|t_1-t_2|}{|t_1-s|^2}.
        $$
        In a similar way we can see that
        $$
        \|\mathbb{H}(t,s_1)-\mathbb{H}(t,s_2)\|_{L^p(\DO,\omega)\to L^p(\DO,\omega)}\le C\frac{|s_1-s_2|}{|s_1-t|^2},
        $$
        provided $|s_1-t|\ge 2|s_1-s_2|$.

            Calder\'on-Zygmund theory implies that, for every $\omega\in A^p(\DO)$, $1<p<\infty$, $R_\mu$ defines a bounded operator
            \begin{itemize}
            \item from $L^q(\R, v,L^p( \DO,\omega))$ into itself, for every $1<q<\infty$ and $v\in A^q(\R)$,
                \item   from $L^1(\R, v,L^p( \DO,\omega))$ into $L^{1,\infty}(\R, v,L^p( \DO,\omega))$, for every $v\in A^1(\R)$,
            \end{itemize}
        provided that $\mu=-1/2$ and $\mu>1/2$.

        In particular, for every $1<p<\infty$, $v\in A^p(\R)$ and $\omega\in A^p(\DO)$, $R_\mu$ defines a bounded operator from $L^p( \R\times\DO,v\omega)$ into itself, when $\mu=-1/2$ or $\mu>1/2$.

        We had proved that, for every $1<p<\infty$ and $W\in A^p_*(\R\times\DO)$, $R_\mu$ defines a bounded operator from  $L^p( \R\times\DO,W)$ into itself.

The proof of Theorem \ref{Iteo3} for the Riesz transformation
$R_\mu$ is finished.

\vspace{5mm}

Our next objective is to get mixed norm inequalities for the Riesz transform $\widetilde{R_\mu}$ defined by, for every $f\in L^p(\R\times\DO)$, $1\le p<\infty$,
    $$
    \widetilde{R_\mu}(f)(t,x)=\lim_{\epsilon\to 0^+}\int_{\Omega_\epsilon(x)}\mathcal{K}_\mu(s,x,y)f(t-s,y)dyds+f(t,x)\frac{1}{\sqrt{\pi}}\int_0^1 e^{-s^2/4}ds, \;\; \text{a.e.}\;\; (t,x)\in\R\times\DO,
    $$
    where $\mathcal{K}_\mu(s,x,y)=\frac{\partial}{\partial s} (W_s^\mu(x,y))$,   $s,x,y\in\DO$.

 \begin{proposition}\label{Prop2}
        Let $\mu>-1/2$ and $1<p<\infty$. The operator $\widetilde{R_\mu}$ can be extended to $L^q(\R, L^p(\DO))$  as a bounded operator from $L^q(\R, L^p(\DO))$ into itself, for $1<q<\infty$ and
        from  $L^1(\R, L^p(\DO))$  into $L^{1,\infty}(\R, L^p(\DO))$.
    \end{proposition}


\begin{proof} This result can be proved by proceeding as in the proof of Proposition \ref{Prop1}. We define, for every $s\in\DO$,
$$
\mathcal{T}_s(F)(x)=\INT\mathcal{K}_\mu(s,x,y)F(y)dy, \;\; x\in\DO.
$$
%
%
%
   By (\ref{F1}), we have that
    \begin{equation}\label{F3}
    |\mathcal{K}_\mu(s,x,y)|\le C\frac{e^{-c\frac{(x-y)^2}{s}}}{s^{3/2}}, \;\; s,x,y\in\DO.
    \end{equation}
    Then, for every $s\in\DO$, $\mathbb{T}_s$ defines a bounded operator from $L^p(\DO)$ into itself and $\| \mathbb{T}_s\|_{L^p(\DO)\to L^p(\DO)}\le\frac{C}{s}.$

For every $t,s\in\R$, we define
        $$
        \mathcal{H}(t,s)(F)=\left\{ \begin{array}{c} \mathcal{T}_{t-s}(F), \;\; t>s \\
        0, \;\;\;\;\;\; \;\;\;\;\;t\le s,
        \end{array} \right.
        $$
        for $F\in L^p(\DO)$.

        Thus, for every $t,s\in\R$, $\mathcal{H}(t,s)\in\mathcal{L}(L^p(\DO))$ and $\|\mathcal{H}(t,s)\|_{p\to p}\le \frac{C}{|t-s|}$.

        Let $1<q< \infty$, We consider, for every $g\in L^q(\R,L^p(\DO) )$,
        $$
        \gamma_\mu(g)(t)=\int_{\R}\mathcal{H}(t,s)(g(s))ds, \;\; t\not\in\text{supp}(g).
        $$

        In Theorem \ref{Iteo2} we proved that the Riesz
        transformation $\widetilde{R_\mu}$ is bounded $L^p(\R\times
        (0,\infty))=L^p(\R,L^p((0,\infty)))$.

        We have that, for every $t\not\in\text{supp}(g),$
        $$
        \left( \int_{\R}\mathcal{H}(t,s)(g(s))ds\right)(x)=\INT\INT\mathcal{K}_\mu(s,x,y)g(t-s,y)dyds, \;\; \text{a.e}\;\;
        x\in\DO.
        $$
        Also, if $t_1,t_2,s\in\DO$, being $|t_1-s|>2|t_1-t_2|$, $\mathcal{H}(t_1,s)-\mathcal{H}(t_2,s)\in \mathcal{L}(
        L^p(\DO))$.
        $$
        \|\mathcal{H}(t_1,s)-\mathcal{H}(t_2,s)\|_{p\to p}\le C\frac{|t_1-t_2|}{|t_1-s|^2}.
        $$
        By invoking vector valued Calder\'on-Zygmund theory we prove that the operator $\widetilde{R_\mu}$ can be extended to $L^q(\R, L^p(\DO))$  as a bounded operator from
        \begin{itemize}
            \item $L^q(\R, L^p(\DO))$ into itself for every $1<q<\infty$,
            \item $L^1(\R, L^p(\DO))$  into $L^{1,\infty}(\R, L^p(\DO))$.
        \end{itemize}
    \end{proof}

Suppose now that $\mu>1/2$ or $\mu=-1/2$.

    By (\ref{F3}), we get
    $$
    |\mathcal{K}_\mu(s,x,y)|\le\frac{C}{s|x-y|}, \;\; s,x,y\in\DO, \;\, \text{and}\, \; x\neq y.
    $$

    From (\ref{F2}), we have that
    $$
    \left|\frac{\partial}{\partial x}\mathcal{K}_\mu(s,x,y)\right|\le C \frac{e^{-c\frac{(x-y)^2}{s}}}{s^2}\le \frac{C}{s|x-y|^2}, \;\; s,x,y\in\DO.
    $$
    By symmetry, we also have that
    $$
    \left|\frac{\partial}{\partial y}\mathcal{K}_\mu(s,x,y)\right|\le \frac{C}{s|x-y|^2}, \;\; s,x,y\in\DO.
    $$
    By using Calder\'on-Zygmund theory we deduce that, for every $1<p<\infty$ and $\omega\in A^p(\DO)$, $\mathcal{T}_s$ defines a bounded operator
    from $L^p(\DO,\omega)$ into itself and $\|\mathcal{T}_s\|_{L^p(\DO,\omega)\to L^p(\DO,\omega)}\le\frac{C}{s}$, for each $s\in\DO$.

    We consider again, for every $t,s\in\R$,
    $$
    \mathcal{H}(t,s)=\left\{ \begin{array}{c} \mathcal{T}_{t-s}, \;\; t>s \\
    0, \;\;\;\;\;\; \;\;\;\;\;t\le s.
    \end{array} \right.
    $$

    Let $1<p<\infty$ and $\omega\in A^p(\DO)$. For every $t,s\in\R$, $\mathcal{H}(t,s)$ defines a bounded operator from $L^p(\DO,\omega)$ into itself and $\|\mathcal{H}(t,s)\|_{L^p(\DO,\omega)\to L^p(\DO,\omega)}\le \frac{C}{|t-s|}$.

    Let $1<q< \infty$, We consider, as above, for every $g\in L^q(\R,L^p(\DO),\omega ))$,
    $$
    \gamma_\mu(g)(t)=\int_{\R}  \mathcal{H}(t,s)(g(s))ds, \;\; t\not\in\text{supp}(g).
    $$
    The last integral is absolutely convergent in the $L^p(\DO,\omega)$-Bochner sense. Moreover, by proceeding as above we get that, if $g\in L^q(\R,L^p(\DO),\omega )$ and $t\not\in\text{supp}(g)$, then
    $$
    [\gamma_\mu(g)(t)](x)=\INT\INT  \mathcal{K}_\mu(s,x,y)g(t-s,y)dyds, \;\; a.e. \; x\in\DO.
    $$
    Suppose that $t_1,t_2,s\in\R$, being $|t_1-s|>2|t_1-t_2|$. Then, we have
    that $\mathcal{H}(t_1,s)-
\mathcal{H}(t_2,s)$ defines a bounded operator from
$L^p(\DO,\omega)$ into itself and
    $$
    \|\mathcal{H}(t_1,s)- \mathcal{H}(t_2,s)\|_{L^p(\DO,\omega)\to L^p(\DO,\omega)}\le C\frac{|t_1-t_2|}{|s-t_1|^2}.
    $$

Moreover, according to Theorem \ref{Iteo2}, (2), the operator
$\widetilde{R_\mu}$ is bounded from $L^p(\R\times\DO,
W)=L^p(\mathbb{R},L^p(\DO,\omega))$, where $W(t,x)=\omega(x)$,
$(t,x)\in \R\times \DO$, because $W\in A_p^*(\R\times\DO)$.

Again, according to Calder\'on-Zygmund theory we deduce that
$\widetilde{R_\mu}$ defines, for every $1<q<\infty$ and $v\in
A^q(\R)$, a bounded operator from $L^q(\R,v,L^p(\DO,\omega))$ into
itself and from $L^1(\R,v,L^p(\DO,\omega))$ into
$L^{1,\infty}(\R,v,L^p(\DO,\omega))$, for every $v\in A^1(\R)$.

The same remark at the end of the study of the mixed norm
inequalities for $R_\mu$ is now in order with respect to
$\widetilde{R_\mu}$.  \vspace{1cm}

\section{Proof of Theorems \ref{Iteo4}, \ref{Iteo5}, and \ref{Iteo6}}

\subsection{Proof of Theorem \ref{Iteo4}} Let $\mu>-1$. We consider the following Cauchy problem

    \begin{equation}\label{CP1}
    \left\{ \begin{array}{c}
    \partial_t u(t,x)=\Delta_\mu u(t,x)+f(t,x), \;\; t,x\in\DO , \\
    u(0,x)=g(x), \;\; x\in \DO.
    \end{array} \right.
    \end{equation}
Suppose initially that $f\in L_c^\infty(\DO\times\DO)$ and $g\in
L^\infty_c(\DO)$. We define
\begin{align*}
    u(t,x)&=\int_0^t\INT W_\tau^\mu(x,y)f(t-\tau,y) dy d\tau+\INT W_t^\mu(x,y)g(y) dy\\
    &=u_1(t,x)+u_2(t,x), \;\; t,x\in\DO.
\end{align*}
According to (\ref{V1}) the integrals defining $u_1$ and $u_2$ are
absolutely convergent for every $t,x\in\DO$.

Assume now that $f\in C_c^2(\DO\times\DO)$. By (\ref{T3}),
(\ref{A5bis}), and (\ref{A6}), we have that $\partial_t
u_2(t,x)=\Delta_\mu u_2(t,x)$, $(t,x)\in\DO\times\DO$. Moreover,

$$
\partial_tu_2(t,x)=\INT\partial_t W_t^\mu(x,y)g(y)dy, \;\, (t,x)\in\DO\times\DO,
$$
and
$$
\partial_x^i u_2(t,x)=\INT \partial_x^i W_t^\mu(x,y)g(y)dy, \;\; (t,x)\in\DO\times\DO \;\; \text{and}\;\; i=1,2.
$$
By \cite[Theorem 2.1]{BHNV} we have that $\displaystyle\lim_{t\to 0}
u_2(t,x)=g(x)$, for a.e. $x\in\DO$ and for every $x\in\DO$ provided
that $g$ is continuous on $\DO$.

By proceeding as in the first section and by taking into account
\cite[(2.17)]{PST} we can obtain that, for $i=1,2$,
$$
\partial_x^i u_1(t,x)=\displaystyle\lim_{\epsilon\to 0^+}\int_\epsilon^t \INT \partial_x^i W_\tau^\mu(x,y)f(t-\tau,y)dyd\tau, \;\; (t,x)\in\DO\times\DO.
$$
By using parametric derivation we get
\begin{align*}
\partial_t u_1(t,x)&=\int_0^t\INT W_\tau^\mu(x,y)\partial_tf(t-\tau,y)dyd\tau=-\int_0^t\INT W_\tau^\mu (x,y)(\partial_\tau f)(t-\tau,y)dyd\tau\\
&=-\int_0^t\INT(W_\tau^\mu(x,y)-W_\tau(x-y))(\partial_\tau f)(t-\tau,y)dyd\tau-\int_0^t\INT W_\tau (x-y)(\partial_\tau f)(t-\tau,y)dyd\tau\\
&=\int_0^t\INT\partial_\tau(W_\tau^\mu(x,y)-W_\tau(x-y))f(t-\tau,y)dyd\tau-\int_0^t\INT W_\tau(x-y)(\partial_\tau f)(t-\tau,y)dyd\tau\\
&=\displaystyle\lim_{\epsilon\to 0^+}\Bigg(\int_\epsilon^t \INT\partial_\tau(W_\tau^\mu(x,y)-W_\tau(x-y))f(t-\tau,y)dyd\tau-\int_\epsilon^t\INT W_\tau(x-y)(\partial_\tau f)(t-\tau,y)dyd\tau \Bigg)\\
&=\lim_{\epsilon\to 0^+}\int_{\epsilon}^{t}\INT \partial_\tau
W_\tau^\mu (x,y)f(t-\tau,y)dyd\tau+f(t,x), \;\;
(t,x)\in\DO\times\DO.
\end{align*}
Putting together the above equalities we get
$$
\partial_t u_1(t,x)=\Delta_\mu u_1(t,x)+f(t,x),\;\; (t,x)\in\DO\times\DO.
$$
Moreover, by (\ref{V1}) since $f$ has compact support, we can find
$a>0$ such that, for every $x\in (0,\infty)$, there exists $C>0$ for
which
$$
|u_1(t,x)|\le
C\int_0^t\INT\frac{e^{-c\frac{|x-y|^2}{\tau}}}{\sqrt{\tau}}dyd\tau\le
C t, \;\; 0<t<a.
$$
Then, $\displaystyle\lim_{t\to 0^+} u_1(t,x)=0, \;\; x\in\DO.$ Thus,
we prove that the function $u$ is a classical solution of
\eqref{CP1}.

\subsection{Proof of Theorem \ref{Iteo5}} Assume now that $\mu>-1/2$. It was established in (\ref{X1}), (\ref{eq2.2}, (\ref{AMI1}), and (\ref{AMI2} that
$$
|\delta_{\mu+1}\delta_\mu W_\tau^\mu(x,y)|\le
C\frac{e^{-c\frac{(x-y)^2}{\tau}}}{\tau^{3/2}}, \;\; \tau,x,y\in\DO.
$$
We can write
\begin{align*}
\delta_{\mu+1}\delta_\mu u_1(t,x)&=\frac{1}{x^2}(\mu+3/2)(\mu+5/2)u_1(t,x)-2\frac{(\mu+1)}{x}\frac{\partial}{\partial x}u_1(t,x)+\frac{\partial^2}{\partial x^2} u_1(t,x)\\
&=\lim_{\epsilon\to 0^+} \int_\epsilon^t\INT
\delta_{\mu+1}\delta_\mu W_\tau^\mu(x,y)f(t-\tau,y)dyd\tau, \;\;
t,x\in\DO.
\end{align*}
By proceeding as in \cite[page 11]{PST} we get
\begin{align*}
|(\chi_{\Omega_\epsilon(x)}(\tau,y)&-\chi_{\{\tau>\epsilon^2\}}(\tau))\chi_{\{\tau<t\}}(\tau)\delta_{\mu+1}\delta_\mu W_\tau^\mu(x,y)| \chi_{\{\tau>0\}}(\tau)\\
&\le C\chi_{\{|x-y|>\epsilon\}}(y)\chi_{(0,\epsilon^2)
}(\tau)\frac{e^{-c\frac{(x-y)^2}{\tau}}}{\tau^{3/2}}\\
& \le C\chi_{\{|x-y|>\epsilon\}}(y)\chi_{(0,\epsilon^2)
}(\tau)\frac{\tau}{|x-y|^5}\\
&\le C\chi_{\{|x-y|>\epsilon\}}(y)\chi_{(0,\epsilon^2)
}(\tau)\frac{\tau}{(|x-y|+\sqrt{\tau})^5}\\
&\le C\chi_{\{|x-y|+\sqrt{\tau}>\epsilon\}}(y)
\frac{\epsilon^2}{(|x-y|+\sqrt{\tau})^5}\\
&\le
C\frac{1}{\epsilon^3}\psi\Big(\frac{|x-y|+\sqrt{\tau}}{\epsilon}\Big),
\;\; \epsilon,\tau,x,y\in\DO,
\end{align*}
where $\psi(z)=\chi_{(\epsilon,\infty)}(z)z^{-5}$, $z\in
(0,\infty)$.

 Suppose that $f$ is  a measurable function on $\R\times\DO$
such that $f(s,y)=0$, $s<0$. We can write
\begin{align*}
&\left|\iint_{\Omega_\epsilon(x)}\delta_{\mu+1}\delta_\mu W_\tau^\mu(x,y)f(t-\tau,y)d\tau dy-\int_{\epsilon^2}^{t}\INT\delta_{\mu+1}\delta_\mu W_\tau^\mu(x,y)f(t-\tau,y)d\tau dy \right|\\
&\le C\int_{|x-y|>\epsilon}\int_{0<\tau<\epsilon^2}\frac{\epsilon^2}{(\sqrt{\tau}+|x-y|^2)^5}|f(t-\tau,y)|d\tau dy\\
&\le C\sum_{k=0}^\infty\int_{2^k\epsilon<|x-y|+\sqrt{|\tau|}<2^{k+1}\epsilon}\frac{\epsilon^2}{(\sqrt{|\tau|}+|x-y|^2)^5}|f(t-\tau,y)|d\tau dy\\
&\le C\sum_{k=0}^\infty\frac{1}{2^{2k}}\frac{1}{(2^{k}\epsilon)^3}\int_{|x-y|+\sqrt{|\tau|}<2^{k+1}\epsilon}|f(t-\tau,y)|d\tau dy\\
&\le C  \mathcal{M}(f) (x,t), \;\; \epsilon,x,t\in\DO,
\end{align*}
where $ \mathcal{M}$ denotes the centered maximal function in our
parabolic setting.

We consider the maximal operator $T_*$ defined by
$$
T_*(f)(t,x)=\sup_{0<\epsilon<\sqrt{t}}\left|\left(\int_{\Omega_\epsilon(x)}-\int_{\epsilon^2}^t\INT
\right)\delta_{\mu+1}\delta_\mu W_\tau^\mu(x,y) f(t-\tau)d\tau dy
\right|, \;\; t,x\in\DO.
$$
According to the boundedness properties of the maximal function
$\mathcal{M}$ we deduce that the operator $T_*$ is  bounded from $
L^p(\R\times\DO,\omega)$
        \begin{itemize}
            \item into $ L^p(\R\times\DO,\omega)$, for every $1<p<\infty$ and $\omega\in A_p^*(\R\times\DO)$.
                \item into $ L^{1,\infty}(\R\times\DO,\omega)$, for $p=1$ and $\omega\in A_1^*(\R\times\DO)$.
        \end{itemize}

Let $f\in{L^p(\DO^2,\omega)}$ with $1\le p<\infty$ and $\omega\in
A^p_*(\R\times\DO)$. We define $\tilde{f}$ by
$$
\tilde{f}(s,x)=\begin{cases} f(s,x),
& \mbox{ $s,x>0$,}\\
0, & \mbox{ $s\le0$, $x>0$.}
\end{cases}
$$
If $\mu>1/2$ or $\mu=-1/2$, by Theorem \ref{Iteo2}, we know that
there exists the limit
$$
\displaystyle\lim_{\epsilon\to
0^+}\int_{\Omega_\epsilon(x)}\delta_{\mu+1}\delta_\mu
W_\tau^\mu(x,y)\tilde{f}(t-\tau,y)d\tau dy, \;\; a.e.
\,\,\,(t,x)\in\DO^2.
$$
Then, a wellknown argument allows us to obtain the existence of the
limit
$$
S_\mu(f)(t,x)=\lim_{\epsilon\to
0^+}\int_{\epsilon}^t\INT\delta_{\mu+1}\delta_\mu
W_\tau^\mu(x,y)\tilde{f}(t-\tau,y)d\tau dy, \;\; a.e.
\,\,\,(t,x)\in\DO^2.
$$
By using again Theorem \ref{Iteo2} we can also prove that the above
limit exists for almost everywhere $(t,x)\in \DO^2$, when $\mu\in
(-1/2,1/2]$ and $f\in L^p(\DO^2)$.

From Theorem \ref{Iteo2} and the boundedness properties of the
operator $T_*$ we deduce that

\begin{enumerate}
\item If $\mu>1/2$ or $\mu=-1/2$. The operator $S_\mu$ is  bounded from $
L^p(\R\times\DO,\omega)$
        \begin{itemize}
            \item into $ L^p(\R\times\DO,\omega)$, for every $1<p<\infty$ and $\omega\in A_p^*(\R\times\DO)$.
                \item into $ L^{1,\infty}(\R\times\DO,\omega)$, for $p=1$ and $\omega\in A_1^*(\R\times\DO)$.
        \end{itemize}
        \item If $\mu>-1/2$, The operator $S_\mu$ is  bounded from  $ L^p(\R\times\DO)$
            \begin{itemize}
                \item into $ L^p(\R\times\DO)$, for every $1<p<\infty$.
                \item into $ L^{1,\infty}(\R\times\DO)$, for $p=1$.
            \end{itemize}
            \end{enumerate}

According to \eqref{T3}, we have that
$$
|\partial_\tau W_\tau^\mu(x,y)|\le
C\frac{e^{-c|x-y|^2/\tau}}{\tau^{3/2}}, \;\; x,y,\tau\in\DO.
$$
By proceeding as above we can prove that the desired properties for
the Riesz transforms $\widetilde{{\bf R}_\mu}$.


\subsection{Proof of Theorem \ref{Iteo6}} We now consider the Cauchy problem

    \begin{equation*}
    \left\{ \begin{array}{c}
    \partial_t u(t)=\Delta_\mu u(t)+f(t), \;\; t\in\DO , \\
    u(0)=0,\hspace{4cm}
    \end{array} \right.
    \end{equation*}
with $\mu>-1/2$ and $f\in L^p(\DO, L^q(\DO))$, being $1<p,q<\infty$.
The operator $\Delta_\mu$ generates the semigroup
$\{W_t^\mu\}_{t>0}$ in $L^q(\DO)$, where, for every $t>0$,
$$
W_t^\mu(g)(x)=\INT W_t^\mu (x,y) g(y)dy, \;\; g\in L^q(\DO).
$$
By (\ref{P1}) and (\ref{P2}) we have that
$$
|W_z^\mu(x,y)|\le C\frac{e^{-c\frac{|x-y|^2}{|z|}}}{|z|^{1/2}}, \;\;
|Arg z|<\frac{\pi}{4}, \;\; x,y\in \DO.
$$
Then, according to \cite[Theorem 3.3]{Monniaux}, $\Delta_\mu$ has
the maximal $L^p$-regularity property in $L^q(\DO)$, that is, there
exists $C>0$ such that
\begin{equation}\label{(*)}
\|\mathcal{R}_\mu f\|_{L^p(\DO,L^q(\DO))}\le C
\|f\|_{L^p(\DO,L^q(\DO))},
\end{equation}
where $$\mathcal{R}_\mu f(t)=\int_0^t\Delta_\mu
W_{t-s}(f(s))ds=\int_0^t\partial_t W_{t-s}(f(s))ds, \;\; f\in
L^p(\DO,L^q(\DO)).
$$
Note that \eqref{(*)} is a mixed-norm inequality.

\end{document}